\documentclass[a4,12pt]{article}%
\usepackage{a4wide}
\usepackage{graphicx}
\usepackage{graphics}
\usepackage{xcolor}
\usepackage{setspace,mathrsfs}
\newcommand{\Eqref}[1]{Eq.\,\eqref{#1}}
\usepackage[maxbibnames=9]{biblatex}
\bibliography{stochastic}

\usepackage{amsmath}
\usepackage{amsfonts}
\usepackage{amssymb}
\usepackage{enumerate}
\newtheorem{theorem}{Theorem}[section]

\newtheorem{definition}[theorem]{Definition}
\newtheorem{example}[theorem]{Example}

\newtheorem{lemma}[theorem]{Lemma}
\newtheorem{assumption}[theorem]{Assumption}
\newtheorem{notation}[theorem]{Notation}

\newtheorem{remark}[theorem]{Remark}
\newtheorem{comment}[theorem]{Remark}

\newenvironment{proof}[1][Proof]{\noindent\textbf{#1.} }{\ \rule{0.5em}{0.5em}}

\newcommand{\E}{{\rm \bf E}}
\renewcommand{\P}{{\rm \bf P}}
\newcommand{\len}{{\rm length}}
\newcommand{\length}{{\rm length}}

\newcommand{\prob}{{\rm \bf P}}

\newcommand{\val}{{\rm val}}

\newcommand{\supp}{{\rm supp}}
\newcommand{\dN}{{\mathbb N}}

\newcommand{\dR}{{\mathbb R}}

\newcommand{\calC}{{\cal C}}

\newcommand{\ep}{\varepsilon}

\newcommand{\black}{\color{black}}

\newcommand{\numbercellonga}[1]
{
\begin{picture}(40,20)(0,0)
\put(0,0){\framebox(40,20)}
\put(20,10){\makebox(0,0){#1}}
\end{picture}
}

\newcounter{figurecounter}
\begin{document}

\title{Equilibrium in Two-Player Stochastic Games\\ with Shift-Invariant Payoffs%
\thanks{
We thank Nicolas Vieille for useful discussions.
This work has been partly supported by COST Action CA16228 European Network for Game Theory. 
Solan acknowledges the support of the Israel Science Foundation, grant \#217/17.}}

\author{J\'{a}nos Flesch\footnote{Department of Quantitative Economics, 
Maastricht University, P.O.Box 616, 6200 MD, The Netherlands. E-mail: j.flesch@maastrichtuniversity.nl.}
\and Eilon Solan\footnote{School of Mathematical Sciences, Tel-Aviv University, Tel-Aviv, Israel, 6997800, E-mail: eilons@tauex.tau.ac.il.}}

\maketitle

\begin{abstract}
We show that every two-player stochastic game with finite state and action sets and bounded, Borel-measurable, and shift-invariant payoffs, 
admits an $\ep$-equilibrium for all $\ep>0$.
\end{abstract}

\noindent\textbf{MSC 2020:} 91A15, 91A05.

\noindent
\textbf{Keywords:} Stochastic game, equilibrium, shift-invariant payoff.

\section{Introduction}\label{sec-intro}

Stochastic games form a classical and central area of dynamic games. In each stage of a stochastic game, the play is in a given state, and each player chooses an action from her set of  available actions. The actions chosen by the players, together with the current state, determine a probability distribution according to which the new state for the next stage is chosen. The payoff of each player is a function of the infinite sequence consisting of all states visited and all actions chosen during the play of the game.

It is a long-standing open problem whether each stochastic game with a finite number of players, finite state and action sets, and bounded  Borel-measurable payoff functions admits an $\ep$-equilibrium for all $\ep>0$. While an affirmative answer has been given in the context of zero-sum games by the celebrated result of Martin \cite{martin1998determinacy} and its extension due to Maitra and Sudderth \cite{maitra1998finitely}, for the non-zero-sum case only partial answers are available. 

Traditionally, stochastic games were considered in the presence of stage payoffs,
where players evaluate the infinite stream of stage payoffs in a certain way,
e.g., the discounted, $T$-stage, or long-term average evaluations.
In his seminal paper, Shapley \cite{shapley1953stochastic} showed that zero-sum discounted stochastic games admit a stationary 0-equilibrium. Fink \cite{fink1964equilibrium} and Takahashi \cite{takahashi1964equilibrium} extended this result, and proved that a stationary 0-equilibrium exists in all multiplayer discounted stochastic games. 

With respect to the long-term average evaluation, 
Mertens and Neyman \cite{mertens1981stochastic} proved that in the zero-sum case, an $\ep$-equilibrium exists for all $\ep>0$. The $\ep$-equilibrium strategies are complex and the recommended actions may depend on the entire history. 
Vieille \cite{vieille2000one,vieille2000two} proved that all two-player stochastic games with long-term average evaluation also admit an $\ep$-equilibrium, for all $\ep>0$, which is one of the foremost results for the existence of $\ep$-equilibrium in stochastic games. In fact, Mertens and Neyman, as well as Vieille, 
showed a stronger statement: the constructed strategy pairs are not only an $\ep$-equilibrium on the infinite horizon, but also on all finite but sufficiently long horizons 
when each player maximizes her average payoff. 
For games with more than two players and the long-term average payoff, the existence of $\ep$-equilibrium is only known in special cases, see, e.g., Solan \cite{solan1999three}, Thuijsman and Raghavan \cite{thuijsman1997perfect}, and Flesch, Schoenmakers, Vrieze \cite{flesch2008stochastic}. For an overview we refer to Solan and Vieille \cite{solan2015stochastic}, Ja\'skiewicz and Nowak \cite{jaskiewicz2018non}, and Levy and Solan \cite{levy2020stochastic}. 

There are various other payoff functions that play a prominent role in the literature. For instance, the limsup and the liminf payoffs (e.g., Maitra and Sudderth \cite{maitra1993borel}), where the payoff is equal to the limit superior or limit inferior of the stream of stage payoffs, respectively. Another important group of payoffs are the ones used in the computer science literature, including reachability, safety, B\"{u}chi, co-B\"{u}chi, parity, Streett, and M\"{u}ller payoffs 
(e.g., Gr\"adel and Ummels \cite{gradel2008solution}, Chatterjee and Henzinger \cite{chatterjee2012survey}, Bruy\`{e}re \cite{Bruyere21}). Also here, the existence of $\ep$-equilibrium has only been established in special cases, for example for two-player stochastic games with parity payoffs, cf.~Chatterjee \cite{chatterjee2005}.  

Many of the previously mentioned payoff functions share a common feature: the payoff is fully determined by the tail of the infinite play. 
This property is called shift-invariance (or prefix-independence); we refer to Section \ref{sect-model} for the formal definition. 
Indeed, the long-term average payoff, the limsup and liminf payoffs, the B\"{u}chi, co-B\"{u}chi, parity, Streett, and M\"{u}ller payoffs are all shift-invariant. One should note that, although a shift-invariant payoff is determined by the tail of the play, single stages remain to play a role, as moving to unfavorable states in the game can limit the possibilities of the players for the remaining part of the infinite play, and thereby having an effect on its tail. 
The discounted payoff, and also the reachability and safety payoffs, are not shift-invariant. 

The main result of the paper is the following existence result: All two-player stochastic games with finite state and action sets and bounded, Borel-measurable, and shift-invariant payoffs, admit an $\ep$-equilibrium for all $\ep>0$. 
This result is an extension of Vieille \cite{vieille2000one,vieille2000two},
who proved the result for the long-term average payoff.

The main line of the proof follows that of Vieille \cite{vieille2000one} and directly invokes the result of Vieille \cite{vieille2000two}. There are however major differences between our proof and that of Vieille. In both proofs, an important technique is to use related zero-sum games in which one of the players is maximizing her own payoff while her opponent is minimizing the same. In the proof of Vieille, the results of Mertens and Neyman \cite{mertens1981stochastic} are used to obtain good strategies in these zero-sum games. In our proof, where the payoff functions are more general, 
we rely on the result of Martin \cite{martin1998determinacy}
for zero-sum games with general payoff functions. 
As a particular consequence, the limit of $\delta$-optimal strategies in these zero-sum games, as $\delta$ vanishes, is a single stationary strategy in Vieille's proof, whereas we obtain an entire collection of mixed actions in each state in the limit. 
This difference makes several arguments in Vieille's proof inapplicable in our setup.

Another essential difference appears in the part of the construction where the players need to detect deviations of their opponents. In Vieille's proof, each player needs to check the action frequencies of her opponent to detect deviations. 
Since the players use almost stationary strategies, this can be done by using statistical tests based on the law of large numbers. In our proof, due to the complexity of the payoff functions, the players need to detect deviations from non-stationary strategies. This is a major complication, and this part of the proof requires new ideas and techniques. In particular, we make use of inner-regularity of probability measures and L\'{e}vy's zero-one law for the detection of deviations. We will provide a more detailed account of the main ideas and steps of our proof in Section~\ref{sec-2}.

We remark that tail-measurability of the payoffs, which is slightly weaker than shift-invariance (see also Section \ref{sec-discuss}), was recently used in two related papers for 
repeated games, i.e., stochastic games with only one state. In Ashkenazi-Golan, Flesch, Predtetchinski, and Solan \cite{ashkenazigolan2021equilibria}, it is shown that every repeated game with countably many players, finite action sets, and bounded, Borel-measurable, and tail-measurable payoffs, admits an $\ep$-equilibrium for all $\ep>0$. 
%
In Ashkenazi-Golan, Flesch, Predtetchinski, and Solan \cite{ashkenazigolan2021regularity}, the regularity properties of the minmax and maxmin values are investigated in multiplayer repeated games, and they are applied to show the existence of $\ep$-equilibria in different classes of repeated games. In particular, it is shown that every repeated game with finitely many players, finite action sets and bounded, upper semi-analytic, and tail-measurable payoffs, admits an $\ep$-equilibrium for all $\ep>0$. \medskip

The paper is structured as follows. In Section~\ref{sec-2}, we present the model and the main result. In Section~\ref{sec-proof}, we provide the proof. 
Additional discussions appear in Section \ref{sec-discuss}.

\section{Model and Main Result}\label{sec-2}

\subsection{The game}\label{sect-model}

Let $\dN=\{1,2,\ldots\}$. For a nonempty finite set $X$, let $\Delta(X)$ denote the set of probability distributions on $X$.

\begin{definition}[stochastic game]
\label{def:stochastic:game}
A two-player \emph{stochastic game}\footnotemark\,is a tuple\\ $\Gamma = (I, S, (A_i)_{i\in I}, p, (f_i)_{i\in I})$,
where
\begin{itemize}
\item   $I = \{1,2\}$ is the set of players.
\item   $S$ is a nonempty finite set of states.
\item   $A_i$ is a nonempty and finite set of actions of player $i$, for each $i \in I$. Let $A=A_1\times A_2$ denote the set of action pairs.
\item   $p : S\times A \to \Delta(S)$ is the transition function. For states $s,s'\in S$ and action pair $a\in A$, we denote by $p(s'\mid s,a)$ the probability of state $s'$ under $p(s,a)$. Let $\mathscr{R}$ denote the set of runs, i.e., all sequences $(s^1,a^1,s^2,a^2,\ldots)\in (S\times A)^\infty$ such that $p(s^{t+1}\mid s^t,a^t)>0$ for each $t\in\dN$.
\item   $f_i : \mathscr{R} \to \dR$ is a payoff function for player~$i$, for each $i \in I$.
\end{itemize}
\end{definition}
\footnotetext{Since the payoff function is not derived from stage payoffs, this model is also called a \emph{multistage stochastic game}.}

The game is played in stages in $\dN$. 
The play starts in a given initial state $s^1$.
In each stage $t \in \dN$ the play is in some state $s^t\in S$. Each player $i \in I$ selects an action $a_i^t \in A_i$, simultaneously with the other player. This induces an action pair $a^t=(a_i^t)_{i\in I}$, which is observed by both players. Then, the state $s^{t+1}$ for stage $t+1$ is drawn from the distribution $p(\cdot\mid s^t,a^t)$ and is observed by both players.  \medskip

\begin{comment}[state independent actions] \rm
In our model, the action sets $A_1$ and $A_2$ are independent of the state. 
This assumption is used to simplify the exposition, and all the statements and proofs in the paper can be extended to stochastic games in which the action sets are finite yet depend on the state.
\end{comment}

\subsection{Histories} 

\begin{definition}[history]
A \emph{history} in stage $t\in\dN$ is a sequence $(s^1,a^1,\ldots,s^{t-1},a^{t-1},s^t)\in (S\times A)^{t-1}\times S$ such that $p(s^{k+1}\mid s^k,a^k)>0$ for each $k=1,\ldots,t-1$. The set of histories in stage $t$ is denoted by $H^t$ and the set of histories is denoted by $H = \bigcup_{t=1}^\infty H^t$.
The current stage (or length) of a history $h\in H^t$ is denoted  by $\len(h)=t$ and the final state of $h$ is denoted by $s_h$. 
\end{definition}

\begin{comment}[on positive transitions]\rm
According to the definition, a finite sequence $(s^1,a^1,\ldots,s^{t-1},a^{t-1},s^t)\in (S\times A)^{t-1}\times S$ where 
$p(s^{k+1}\mid s^k,a^k)=0$ for some $k=1,\ldots,t-1$ is not a history.
Similarly, an infinite sequence 
$(s^1,a^1,s^2,a^2,\ldots)\in (S\times A)^\infty$ such that $p(s^{t+1}\mid s^t,a^t)=0$ for some $t\in\dN$ is not a run,
and the payoff functions are not defined over such infinite sequences.
We selected this definition for convenience only. For example, under this definition, we do not need to define strategies when play leaves an absorbing state; a state is absorbing if the play remains in this state with probability 1 regardless the action pair. 
\end{comment}

\begin{notation}[$h\preceq h'$, $hh'$]\rm 
For two histories $h,h'\in H$, we write $h\preceq h'$ if $h'$ extends $h$ (with possibly $h=h'$), and we write $h\prec h'$ if $h\preceq h'$ and $h\neq h'$. If the final state $s_h$ of $h$ coincides with the first state of $h'$, then we write $hh'$ for the concatenation of $h$ with $h'$. Similarly, for a history $h\in H$ and a run $r\in\mathscr{R}$, we write $h\prec r$ if $r$ extends $h$, and if the final state $s_h$ of $h$ coincides with the first state of $r$, then we write $hr$ for the concatenation of $h$ with $r$.
\end{notation}

\begin{definition}[history remaining in a set of states]
For a nonempty set $C\subseteq S$ of states, we say that a history $h$ \emph{remains in $C$} if all states along $h$ are in $C$.
\end{definition}

\begin{definition}[subgame]
Each history induces a subgame of $\Gamma$. Given $h\in H$, the subgame that starts at $h$ is the game $\Gamma^h = (I,S,(A_i)_{i \in I},p, (f_{i,h})_{i \in I})$ having $s_h$ as the initial state, where $f_{i,h}(r) := f_i(hr)$ for each run $r\in\mathscr{R}$ that starts in state $s_h$.\medskip
\end{definition}

\subsection{Strategies}

\begin{notation}[$\mathcal{B} (\mathscr{R})$]\rm 
We endow the set $S$ of states and the set $A$ of action pairs with the discrete topology and the set $(S\times A)^\infty$ with the product topology. Thus, the open sets of $(S\times A)^\infty$ are the unions of cylinder sets, i.e., sets of the form $\{r \in (S\times A)^\infty \colon h \prec r\}$, where $h\in H$.
The set $\mathscr{R}$ of runs is a closed subset of $(S\times A)^\infty$, which we endow with the subspace topology. We denote by $\mathcal{B} (\mathscr{R})$ the corresponding Borel sigma-algebra on $\mathscr{R}$. 
\end{notation}

A payoff function $f_i$ is called \emph{Borel-measurable} if it is measurable w.r.t.~$\mathcal{B} (\mathscr{R})$.

\begin{definition}[mixed action]
A \emph{mixed action} for player $i\in I$ in state $s\in S$ is a probability distribution $x_i(s)$ on $A_i$. The set of mixed actions for player $i$ in state $s$ is thus $\Delta(A_i)$.  
The probability that $x_i(s)$ places on action $a_i\in A_i$ is denoted by $x_i(s,a_i)$. 
The \emph{support} of the mixed action $x_i(s)$ is $\supp(x_i(s)):=\{a_i\in A_i\colon x_i(s,a_i)>0\}$. The support of a pair of mixed actions $x(s)=(x_1(s),x_2(s))$ in state $s$ is $\supp(x(s)):=\supp(x_1(s))\times \supp(x_2(s))\subseteq A$. 
\end{definition}

\begin{notation}[$p(s'\mid s,x(s))$]\rm 
For a pair of mixed actions $x(s)=(x_1(s),x_2(s))$, we denote the probability of moving from state $s$ to state $s'$ under $x(s)$ by
\[p(s'\mid s,x(s))\,:=\,\sum_{(a_1,a_2)\in A} x_1(s,a_1)\cdot x_2(s,a_2)\cdot p(s'\mid s,a_1,a_2).\]
\end{notation}

\begin{definition}[strategy]
A (behavior) \emph{strategy} of player~$i$ is a function $\sigma_i : H \to \Delta(A_i)$. We denote by $\sigma_i(a_i \mid h)$ the probability on the action $a_i\in A_i$
under $\sigma_i(h)$. The interpretation of $\sigma_i$ is that if history $h$ arises, then $\sigma_i$ recommends to select an action according to the mixed action  $\sigma_i(h)$. We denote by $\Sigma_i$ the set of strategies of player~$i$.
\end{definition}

\begin{definition}[stationary strategy]
A strategy $\sigma_i$ of player $i$ is called \emph{stationary}, if the mixed actions prescribed by $\sigma_i$ only depend on the current state. Thus, a stationary strategy of player $i$ can be identified with an element of $\prod_{s\in S}\Delta(A_i)$.
\end{definition}

\begin{definition}[continuation strategy]
The \emph{continuation of a strategy $\sigma_i$} in the subgame that starts at a history $h\in H$ is denoted by $\sigma_{i,h}$; 
this is a function that maps each history $h'$ having $s_h$ as its first state to the mixed action $\sigma_{i,h}(h')=\sigma_i(hh')\in\Delta(A_i)$.\medskip
\end{definition}

\subsection{Expected payoffs and equilibrium}

\begin{notation}[$\P_{s,\sigma_1,\sigma_2}$, $\E_{s,\sigma_1,\sigma_2}$]\rm 
By Kolmogorov's extension theorem, each  strategy pair $(\sigma_1,\sigma_2)$ together with an initial state $s$ induces a unique probability measure $\P_{s,\sigma_1,\sigma_2}$ on $(\mathscr{R},\mathcal{B} (\mathscr{R}))$. The corresponding expectation operator is denoted by $\E_{s,\sigma_1,\sigma_2}$.
\end{notation}

If player $i$'s payoff function $f_i$ is bounded and Borel-measurable, then
player $i$'s expected payoff under $(\sigma_1,\sigma_2)$ is
\[\E_{s,\sigma_1,\sigma_2}[f_i]\,=\,\int_{r\in\mathscr{R}}f_i(r)\ \P_{s,\sigma_1,\sigma_2}(dr).\]

\begin{notation}[$\P_{h,\sigma_1,\sigma_2}$, $\E_{h,\sigma_1,\sigma_2}$]\rm 
In the subgame $\Gamma^h$, for any history $h\in H$, each strategy pair $(\sigma_1,\sigma_2)$ induces a unique probability measure $\P_{h,\sigma_1,\sigma_2}$ on $(\mathscr{R},\mathcal{B} (\mathscr{R}))$. The corresponding expectation operator is denoted by $\E_{h,\sigma_1,\sigma_2}$. 
\end{notation}

If player $i$'s payoff function $f_i$ is bounded and Borel-measurable, then
player $i$'s expected payoff under $(\sigma_1,\sigma_2)$ in $\Gamma^h$ is
\[\E_{h,\sigma_1,\sigma_2}[f_i]\,=\,\int_{r\in\mathscr{R}}f_{i,h}(r)\ \prob_{s_h,\sigma_{1,h},\sigma_{2,h}}(dr).\]

\begin{definition}[$\ep$-equilibrium]
Suppose that each player $i$'s payoff $f_i$ is bounded and Borel-measurable. Let $\ep\geq 0$. A strategy pair $(\sigma^*_1,\sigma^*_2)$ is called an \emph{$\ep$-equilibrium for the initial state $s\in S$}, if we have $\E_{s,\sigma^*_1,\sigma^*_2}[f_1] \,>\, \E_{s,\sigma_1,\sigma^*_2}[f_1] - \ep$ for each strategy $\sigma_1 \in \Sigma_1$ and $\E_{s,\sigma^*_1,\sigma^*_2}[f_2] \,>\, \E_{s,\sigma^*_1,\sigma_2}[f_2] - \ep$ for each strategy $\sigma_2 \in \Sigma_2$. A strategy pair $(\sigma^*_1,\sigma^*_2)$ is called an \emph{$\ep$-equilibrium} if it is an $\ep$-equilibrium for each initial state $s\in S$.
\end{definition}

\begin{definition}[shift-invariance]
A payoff function $f_i$ is called \emph{shift-invariant} (or prefix-independent), if 
\[f_i(s^1,a^1,s^2,a^2,s^3,a^3,\ldots)\,=\,f_i(s^2,a^2,s^3,a^3,\ldots)\] holds for every run $(s^1,a^1,s^2,a^2,s^3,a^3,\ldots) \in \mathscr{R}$. 
\end{definition}

Equivalently, $f_i$ is shift-invariant if whenever two runs have the form $hr$ and $h'r$, i.e., they only differ in the prefixes $h$ and $h'$, then $f_i(hr)=f_i(h'r)$. 
The set of shift-invariant functions is not included by, neither does it include,
the set of Borel-measurable functions, 
\color{black}
see Rosenthal \cite{rosenthal1975nonmeasurable} and Blackwell and Diaconis \cite{blackwell1996non}. 

Many evaluation functions in the literature of dynamic games are shift-invariant, such as the long-term average payoffs (e.g., Mertens and Neyman \cite{mertens1981stochastic,mertens1982stochastic}) and the limsup of stage payoffs (e.g., Maitra and Sudderth \cite{maitra1993borel}).
Various classical winning conditions in the computer science literature, such as the B\"{u}chi, co-B\"{u}chi, parity, Streett, and Müller  (e.g., Horn and Gimbert \cite{Horn2008optimal}, Chatterjee and Henzinger \cite{chatterjee2012survey}, Bruy\`{e}re \cite{Bruyere21}) are also shift-invariant.
The discounted payoff (e.g., Shapley \cite{shapley1953stochastic}) is not shift-invariant.\medskip

\subsection{Maxmin value} 

\begin{definition}[maxmin value]
Suppose that player $i$'s payoff function $f_i$ is bounded and Borel-measurable. The \emph{maxmin value} of player $i$ for initial state $s\in S$ is the quantity
\begin{equation}
\label{def-maxmin}
v_i(s):=\sup_{\sigma_i\in\Sigma_i}\inf_{\sigma_j\in\Sigma_j}\E_{s,\sigma_i,\sigma_j}[f_i]\,=\,\inf_{\sigma_j\in\Sigma_j}\sup_{\sigma_i\in\Sigma_i}\E_{s,\sigma_i,\sigma_j}[f_i],
\end{equation}
where player $j$ denotes player $i$'s opponent.
\end{definition}

The second equality in \Eqref{def-maxmin} follows by Martin \cite{martin1998determinacy} or by Maitra and Sudderth \cite{maitra1998finitely}. In other words, $v_i(s)$ is the value of the zero-sum game that is derived from the stochastic game by assuming that player $i$ is maximizing her payoff given by $f_i$ whereas her opponent is minimizing player $i$'s payoff. The maxmin value $v_i(h)$ of player $i$ in the subgame that starts at history $h$ is defined similarly.

If the payoff function $f_i$ is also shift-invariant, then the payoff given by $f_i$ is independent of the initial segment of the play, and hence the maxmin value of player $i$ is equal across subgames that start in the same state. Formally, the following statement holds.

\begin{lemma}\label{const-surr} Suppose that player $i$'s payoff function $f_i$ is bounded, Borel-measurable, and shift-invariant. Then, for each state $s\in S$ and each history $h\in H$ with $s_h=s$, it holds that $v_i(h)=v_i(s)$.
\end{lemma}

\subsection{Main Result}
\label{section:main:result}

The main result of the paper is the following existence result.

\begin{theorem}
\label{theorem:1}
Every two-player stochastic game $\Gamma$ with finite state and action sets and bounded,  Borel-meausurable, and shift-invariant payoffs, admits an $\ep$-equilibrium for all $\ep > 0$.
\end{theorem}

As described in the introduction, 
Theorem~\ref{theorem:1} was proven by Vieille \cite{vieille2000one,vieille2000two} for the case when $f_i$ is the long-term average payoff.
In fact, Vieille \cite{vieille2000one,vieille2000two} proved a stronger statement: a uniform $\ep$-equilibrium exists,
namely, a strategy profile that is an $\ep$-equilibrium in every $T$-stage game,
provided $T$ is sufficiently large,
and also in every discounted game, provided the players are sufficiently patient.

We will now give a detailed account of the main ideas and steps of our proof.
As the description reveals, 
our proof follows that of Vieille \cite{vieille2000one} and directly uses the main result of Vieille \cite{vieille2000two}. Yet, there are major differences between the two proofs, as already indicated in Section \ref{sec-intro}. 
To describe Vieille's proof, we need the concepts of absorbing and nonabsorbing states.

\begin{definition}[absorbing state, nonabsorbing state]
A state $s\in S$ is called \emph{absorbing} if $p(s\mid s,a)=1$ for each action pair $a\in A$. 
Otherwise, the state is called \emph{nonabsorbing}. 
We denote by $S^*$ the set of absorbing states, and by $S_0=S\setminus S^*$ the set of nonabsorbing states.
\end{definition}

\medskip
\noindent\textsc{Comparison with Vieille's proof.}  
The proof of Vieille has two major steps:
\begin{itemize}
    \item[(V.1)]
    In \cite{vieille2000two}, Vieille proves that every stochastic game with the long-term average payoff that satisfies the following properties admits an $\ep$-equilibrium, for every $\ep > 0$:
    (a) in nonabsorbing states, the stage payoffs of both players are 0;
    (b) in absorbing states, the stage payoff of player~2 is positive;
    and (c) player 2 has a stationary strategy such that regardless of the initial state and player 1's strategy, the play eventually reaches an absorbing state with probability 1.
    \item[(V.2)]
    In \cite{vieille2000one}, Vieille proves that (V.1) implies that every stochastic game with the long-term average payoff admits an $\ep$-equilibrium, for every $\ep > 0$.
\end{itemize}

To prove Theorem~\ref{theorem:1} for shift-invariant payoffs, we rely on (V.1), 
and prove that (V.1) also implies the statement of  Theorem~\ref{theorem:1}. In other words, we prove an implication that is stronger than the one in (V.2).

The heart of the proof of (V.2) is the following \emph{property of the alternatives},
see Proposition~32 in \cite{vieille2000one}.
Let $\overline x_i$ be an accumulation point of discounted stationary optimal strategies of player~$i$, as the players become more patient.
Let $C \subseteq S_0$ be a set of nonabsorbing states with the following properties:
\begin{enumerate}
\item[{(1)}] Communication: The set $C$ communicates under $(\overline x_1,\overline x_2)$;
namely, for every two distinct states $s,s' \in C$ there is a stationary strategy pair
$(y_1,y_2)$ that (i) ensures that the play never leaves $C$,
(ii) ensures that the play reaches $s'$ with probability 1, when the initial state is $s$,
and (iii) satisfies that the support of $y_i(s)$ contains the support of $\overline x_i(s)$,
for each player $i\in\{1,2\}$ and each state $s \in C$.
\item[{(2)}] The maxmin value of each player $i\in\{1,2\}$ is constant over $C$, denoted by $v_i(C)$, i.e., $v_i(s)=v_i(C)$ for each state $s\in C$. 
\item[{(3)}] Blocking to player 2: When player~1 plays the stationary strategy $\overline x_1$,
for each state $s \in C$ and each action $a_2 \in A_2$ 
that leads the play outside $C$ with positive probability,
the expected continuation maxmin value of player~2 is strictly smaller than $v_2(C)$.
\item[{(4)}] Blocking to player 1: The property analogous to (3) holds, with the roles of the two players exchanged.
\end{enumerate}
The property of the alternatives states that in this case, at least one of the following alternatives holds:
\begin{itemize}
\item[(A.1)] For every initial state in $C$, the game admits an $\ep$-equilibrium for every $\ep > 0$.
\item[(A.2)] The set $C$ is jointly controlled: Roughly speaking, there is a probability distribution~$q$ on the set of states outside $C$, i.e., $q\in \Delta(S\setminus C)$, and for every $\lambda>0$ there is a strategy pair $(\sigma_1,\sigma_2)$, such that whenever the initial state is in $C$, the following properties hold under $(\sigma_1,\sigma_2)$: 
\begin{itemize}
\item[(i)] The play leaves $C$ with probability 1 and the probability that first state the play reaches outside $C$ is $s$ is equal to $q(s)$, 
for every $s \not\in C$, and 
\item[(ii)] For each $i\in\{1,2\}$, the expectation of player $i$'s maxmin value in the first state the play reaches outside $C$ is at least $v_i(C)$ and also at least the payoff that player $i$ could gain by deviating from $\sigma_i$ before the play leaves $C$, up to $\lambda$. 
\end{itemize}
\end{itemize}
Conditions (1)--(4) and property (A.2) depend on the transition function and on the maxmin values of the players,
but are independent of the actual payoff function.
Hence, to prove the analog of the main step of \cite{vieille2000one} for shift-invariant payoff functions,
it is sufficient to prove that if conditions~(1)--(4) hold and property~(A.2) does not hold, then property~(A.1) holds.
The main challenge in our proof is to prove this implication,
and it arises since when payoffs are shift-invariant, 
there is no analogue to the discounted payoff.

To prove the property of the alternatives,
Vieille \cite{vieille2000one} uses a celebrated result of Mertens and Neyman \cite{mertens1981stochastic},
stating that when the payoff function is the long-term average payoff, each player $i$ has a $\delta$-optimal strategy that is a perturbation of $\overline x_i$.
When the payoff function is shift-invariant,
a $\delta$-optimal strategy exists by Martin \cite{martin1998determinacy} or Maitra and Sudderth \cite{maitra1998finitely},
yet this strategy need not be a perturbation of a stationary strategy.
In particular, the proof of Vieille \cite{vieille2000one} for the property of the alternatives does not carry over to shift-invariant payoffs.
Below we will expand on the way we overcome this issue, see Remark~\ref{remark:explanation}.
We now turn to explain the structure of the proof.

\bigskip

\noindent\textsc{Structure and sketch of the proof.} 
We start with some simplifying assumptions (Assumptions \ref{As1} and \ref{As2} below), 
which we can make without loss of generality. 
Most importantly, we assume that each initial state for which an $\ep$-equilibrium exists for all $\ep>0$ is absorbing. 
Our proof is by contradiction: We assume that there is an initial state for which there is no $\ep$-equilibrium for small $\ep>0$. That is, we assume that 
there is at least one nonabsorbing state.

Section \ref{sec-prelnot} is devoted to preliminary notations. In Section \ref{sec-comm}, we define the notion of a communicating set of states, which formalizes the idea that the players can reach any state from any other state inside this set, without leaving it. 
In Section \ref{sec-exits}, we categorize mixed action pairs under which a set of states is left with a positive probability, into so-called unilateral exits and joint exits: 
the mixed action pair is a unilateral exit if leaving the set of states is due to the action of one of the players, and it is a joint exit if the set of states is left when both players play certain actions. 
In Section \ref{sec-contr}, we examine when such an exit is satisfactory,
that is, it yields a high expected continuation maxmin value. 
If an exit is satisfactory to both players, then we call the set of states controlled; either unilaterally controlled by one of the players or jointly controlled. If a player has no satisfactory unilateral exit, then the set of states is called blocked to this player. 

In Section \ref{section:martin}, we define the concept of
subgame-perfect $\delta$-maxmin strategies,
namely,
a strategy of a player that guarantee that her expected payoff in each subgame is at least her maxmin value up to an error-term $\delta>0$. 
We invoke results from Martin \cite{martin1998determinacy} to construct for each player a specific subgame-perfect $\delta$-maxmin strategy that is suited for our purposes.
These subgame-perfect $\delta$-maxmin strategies will be analogous to the strategies constructed by Mertens and Neyman \cite{mertens1981stochastic} and used in \cite{vieille2000one}.

In Section \ref{sec-contr-sgperf}, we relate these strategies to controllability of a set of states. This section contains one of the key steps of our proof, Lemma \ref{lemma:alternatives}, which is the analog of the property of the alternatives mentioned earlier. 
We discuss the proof of this lemma in detail in Section \ref{sec-contr-sgperf}.

The rest of the construction follows the arguments in Vieille \cite{vieille2000one}.
In Sections \ref{sec-firstfam}, \ref{sec-secondfam}, and \ref{sec-thirdfam}, we define three families of sets of states.
In the $\ep$-equilibrium that we will construct,
states in these sets will be dummy: 
the play will visit them finitely many times so they will not affect the overall payoff.
In Section \ref{sec-auxgame} we define an auxiliary recursive game.
By Vieille~\cite{vieille2000two} this auxiliary game admits an $\ep$-equilibrium for every $\ep > 0$, and we will show that this $\ep$-equilibrium can be adapted to an $\ep$-equilibrium in the original game.
Since this is in contradiction with our initial assumption, the proof of Theorem \ref{theorem:1} will be complete. 

\section{Proof}\label{sec-proof}

In this section we prove Theorem \ref{theorem:1}. To simplify the exposition of the proof, we will make two assumptions without loss of generality.

\begin{assumption}\label{As1} The payoff functions satisfy the following properties: $f_1 \leq -1$, $f_2 \geq 1$, and $f_1$ and $f_2$ have a finite range.
\end{assumption}

We argue why this assumption can be made without loss of generality. 
The requirement that $f_1 \leq -1$ and $f_2 \geq 1$ can be achieved by adding a constant to the payoffs, 
which do not affect the set of $\ep$-equilibria.
The requirement that $f_1$ and $f_2$ have a finite range
can be achieved by rounding $f_i(r)$ to the largest multiple of $\ep$ smaller than or equal to $f_i(r)$, for each run $r\in\mathscr{R}$.
This operation does not affect the shift-invariance of the function $f_i$,
and any $\ep$-equilibrium in the game with payoffs rounded down is a $3\ep$-equilibrium of the original game.

\begin{assumption}\label{As2} 
Let $s \in S_0$.
If for every $\ep > 0$ there exists an $\ep$-equilibrium for the initial state $s$, 
then $s$ is absorbing, i.e., $s\in S^*$. 
Moreover, all runs that reach state $s$ give the same payoff, denoted by $(\gamma_1^s,\gamma_2^s)$: for every run $r\in\mathscr{R}$ that reach state $s$ and for every player $i\in\{1,2\}$ we have $f_i(r)=\gamma_i^s$.
\end{assumption}

The idea behind this assumption is standard (e.g., Vieille \cite{vieille2000one}). We argue briefly why it can be made without loss of generality.

Take a two-player stochastic game $\Gamma$ satisfying the conditions of Theorem \ref{theorem:1}.  Fix an accumulation point $(\gamma_1^s,\gamma_2^s)$ of the sets of $\ep$-equilibrium payoffs, as $\ep\to 0$, when the initial state is $s$. 
We turn state $s$ into an absorbing state and, if the play visits state $s$, we let the payoff be $(\gamma_1^s,\gamma_2^s)$. 
More precisely, let $\Gamma'=(I,S,(A_i)_{i\in I},p',(f'_i)_{i\in I})$ be the stochastic game that is derived from $\Gamma$ as follows: (i) $p'(s\mid s,a)=1$ for all $a\in A$, and $p'(\cdot\mid w,a)=p(\cdot\mid w,a)$, for all $w\in S\setminus\{s\}$ and $a\in A$, 
(ii) for each player $i$, $f'_i(r)=\gamma_i^s$ if the run $r\in \mathscr{R}'$ in $\Gamma'$ visits state $s$, and otherwise $f'_i(r)=f_i(r)$. Then, (i) the game $\Gamma'$ satisfies the conditions of Theorem \ref{theorem:1}, and (ii) if $\Gamma'$ admits an $\ep$-equilibrium for some initial state $s$ for all $\ep>0$, 
then $\Gamma$ 
admits an $\ep$-equilibrium for the initial state $s$ for all $\ep>0$. 
To see (ii), notice that any $\ep$-equilibrium in $\Gamma'$ can be turned into an $(\ep+\delta)$-equilibrium in $\Gamma$, for any $\delta>0$. Indeed, once the play reaches state $s$, by using shift-invariance, we can replace the continuation strategies with a $\frac{\delta}{2}$-equilibrium
for the initial state $s$ with payoff $\delta$-close to $(\gamma_1^s,\gamma_2^s)$.
Consequently, for the proof of Theorem \ref{theorem:1}, we can make Assumption \ref{As2} without loss of generality.

\begin{comment}
If $s$ is an absorbing state, then by Ashkenazi-Golan, Flesch, Predtetchinski,
and Solan \cite{ashkenazigolan2021equilibria}, there exists an $\ep$-equilibrium for the initial state $s$, for every $\ep > 0$. 
Hence, by Assumption \ref{As2}, (i) a state $s$ is absorbing if and only there exists an $\ep$-equilibrium for the initial state $s$, for every $\ep > 0$, and (ii) once the run reaches an absorbing state, the payoffs do not depend on the continuation of the run.  
\end{comment}

We fix a stochastic game $\Gamma$ satisfying the conditions of Theorem \ref{theorem:1} and Assumptions \ref{As1} and \ref{As2} for the rest of the proof. As mentioned earlier, the proof of Theorem \ref{theorem:1} is by contradiction, 
and we make the following contrapositive assumption.

\begin{assumption}[by contradiction]\label{As3}
The set $S_0$ of nonabsorbing states is nonempty: 
For some initial state and some $\ep > 0$, the game $\Gamma$ does not admit an $\ep$-equilibrium.
\end{assumption}

\subsection{Preliminary Notations}\label{sec-prelnot}

We will denote a general player by $i$, and the other player by $j$. 

\begin{notation}[bound on payoffs]\rm 
Let $M \in \dR$ be a bound on the payoffs:
\[ M := \max_{i\in\{1,2\}}\max_{r\in\mathscr{R}} |f_i(r)|.\]
\end{notation}

\begin{notation}[$L_\infty$-distance]\rm 
For two vectors $\xi,\xi' \in \dR^k$, where $k\in \dN$, we denote by $d(\xi,\xi')$ the $L_\infty$-distance between $\xi$ and $\xi'$:
\[d(\xi,\xi'):=\max_{\ell\in\{1,\ldots,k\}} |\xi(\ell)-\xi'(\ell)|.\] 
For a state $s\in S$, mixed action $x_i(s)\in\Delta(A_i)$ of player $i$, and a nonempty set of mixed actions $X_i(s)\subseteq \Delta(A_i)$ of player $i$, we denote the $L_\infty$-distance of $x_i(s)$ from $X_i(s)$ by
\[d(x_i(s),X_i(s))\,:=\,\inf_{y_i(s)\in X_i(s)} d(x_i(s),y_i(s)).\]
Given, in addition, another nonempty set of mixed actions $Y_i(s)\subseteq \Delta(A_i)$, we define
\[d(X_i(s),Y_i(s))\,:=\, \sup_{x_i(s)\in X_i(s)}d(x_i,Y_i(s)).\]
This should not be confused with the Hausdorff distance between sets.
\end{notation}

\begin{notation}[expectation given state and actions]\rm 
For every function $g : S \to \dR$, state $s \in S$, and mixed actions $x_1(s) \in \Delta(A_1)$ and $x_2(s) \in \Delta(A_2)$, we define
\begin{equation}\label{oneshot-g}
\E[g \mid s,x_1(s),x_2(s)] := \sum_{s' \in S} p(s' \mid s,x_1(s),x_2(s))\cdot g(s'),
\end{equation}
which is the expected value of $g$ in the state that is visited in one step from state $s$ if the players use the mixed actions $x_1(s)$ and $x_2(s)$. 
\end{notation}

In particular, $\E[v_i \mid s,x_1(s),x_2(s)]$ is the expected maxmin value of player $i$ in the next state if the players use the mixed actions $x_1(s)$ and $x_2(s)$ in state $s$. 
By the definition of the maxmin value,
for every state $s \in S$, player $i\in\{1,2\}$, and mixed action
$x_j(s) \in \Delta(A_j)$, there exists $a_i \in A_i$ such that
\begin{equation}
\label{equ:minmax}
\E[v_i \mid s,a_i,x_j(s)] \geq v_i(s).
\end{equation}
Indeed, suppose the opposite: for some state $s\in S$, player $i\in\{1,2\}$, and mixed action
$x_j(s) \in \Delta(A_j)$, we have $\delta:=v_i(s)-\max_{a_i\in A_i}\E[v_i \mid s,a_i,x_j(s)]>0$. Suppose that the game starts in state $s^1=s$ and player $j$ plays the following strategy $\sigma_j$: in stage 1, $\sigma_j$ uses the mixed action $x_j(s)$, and from stage 2 on $\sigma_j$ makes sure that player $i$'s payoff is at most $v_i(s^2)+\frac{\delta}{2}$
(recall that $s^2$ denotes the state in stage 2). 
Against $\sigma_j$, player $i$'s payoff is at most $\max_{a_i\in A_i}\E[v_i+\frac{\delta}{2} \mid s,a_i,x_j(s)]=\max_{a_i\in A_i}\E[v_i\mid s,a_i,x_j(s)]+\frac{\delta}{2}=v_i(s)-\frac{\delta}{2}$,
which contradicts the definition of $v_i(s)$, cf.~\Eqref{def-maxmin}.

\begin{notation}[expectation given history and actions]\rm 
For every function $D : H \to \dR$, history $h\in H$, and mixed actions $x_1(s_h) \in \Delta(A_1)$ and $x_2(s_h) \in \Delta(A_2)$, define
\begin{equation}\label{oneshot-D}
\E[D \mid h,x_1(s_h),x_2(s_h)] := \sum_{(s',a)\in S\times A} x_1(s_h,a_1)\cdot x_2(s_h,a_2)\cdot p(s' \mid s_h,a_1,a_2)\cdot D(h,a_1,a_2,s'),
\end{equation}
which is the expected value of $D$ at the history in the next stage if at history $h$ the players use the mixed actions $x_1(s_h)$ and $x_2(s_h)$.
\end{notation}

\begin{notation}[$G_i(d,h)$, $\val_i(D,h)$]\rm For a function $D : H \to \dR$, a player $i\in\{1,2\}$, and a history $h\in H$, let $G_i(D,h)$ denote the one-shot zero-sum game in which the players' action sets are $A_1$ and $A_2$, 
the payoff is given by \Eqref{oneshot-D},  
and player $i$'s goals is to maximize the payoff (and player $j$'s goal is to minimize it). 
Let $\val_i(D,h)$ denote the value of $G_i(D,h)$.
\end{notation}

\begin{notation}[exit time $\theta_C^{exit}$]\rm 
For a nonempty set $C\subseteq S_0$, we denote the exit time from $C$ by
\[ \theta_C^{exit} := \min\{ t \in \dN \colon s^t \not\in C\}, \]
with the convention that the minimum of the empty set is infinity.
\end{notation}

\subsection{Communicating Sets}
\label{sec-comm}

In this subsection, we define the notion of communicating sets. Intuitively, a communicating set is a set of states in which, without leaving the set, it is possible to move, possibly in a number of steps, from any state to any other state using mixed action pairs that lie in a pre-specified set of mixed-action pairs.

Let $\mathcal{X}$ denote the collection of all nonempty sets $X$ of the form $X = \prod_{s \in S_0} \prod_{i=1,2} X_i(s)$, where $X_i(s)\subseteq \Delta(A_i)$ for each state $s\in S_0$ and each player $i\in\{1,2\}$. 
Each element of $X$ can be seen as a stationary strategy pair in the game, as the actions in absorbing states do not matter. 
We denote $X(s) = X_1(s) \times X_2(s)$ for each $s \in S_0$ and $X_i=\prod_{s \in S_0} X_i(s)$ for each $i\in\{1,2\}$. 

\begin{definition}[perturbation] 
Let $X \in \mathcal{X}$. 
\begin{enumerate}
\item Consider a state $s\in S_0$ and a mixed action $y_i(s)\in \Delta(A_i)$ for player $i\in\{1,2\}$. We say that $y_i(s)$ is a \emph{perturbation of a mixed action $y_i'(s)\in\Delta(A_i)$} if $\supp(y'_i(s))\subseteq\supp(y_i(s))$. We say that $y_i(s)$ is a \emph{perturbation of $X_i(s)$} if $y_i(s)$ is a perturbation of some $x_i(s)\in X_i(s)$.
\item Consider a stationary strategy $y_i=(y_i(s))_{s\in S_0}$ of player $i\in\{1,2\}$. We say that $y_i$ is a \emph{perturbation of a stationary strategy $y'_i=(y'_i(s))_{s\in S_0}$} if $y_i(s)$ is a perturbation of $y'_i(s)$ for each state $s\in S_0$. We say that $y_i$ is a \emph{perturbation of $X_i$} if $y_i$ is a perturbation of some $x_i\in X_i$.
\item Perturbations for pairs of mixed actions and for stationary strategy pairs is defined similarly, by requiring the corresponding property for each player separately.
\end{enumerate}
\end{definition}

Note that if $y_i(s)$ is a perturbation of $X_i(s)$, then for every $\rho>0$ there is a perturbation $y^\rho_i(s)$ of $X_i(s)$ such that $\supp(y^\rho_i(s))=\supp(y_i(s))$ and $d(y^\rho_i(s),X_i(s))\leq \rho$. Indeed, let
$x_i(s)\in X_i(s)$ satisfy $\supp(x_i(s))\subseteq\supp(y_i(s))$, and for $\rho\in(0,1)$, take $y^\rho_i(s)=\rho \cdot y_i(s)+(1-\rho)\cdot x_i(s)$. A similar statement holds if $y$ is a perturbation of $X$.

\begin{definition}[communicating set]
\label{definition:communicating}
A nonempty set $C\subseteq S_0$ is \emph{$X$-communicating}, where $X \in \mathcal{X}$, if the following two conditions hold:
\begin{enumerate}
\item   Under mixed actions in $X$, the set $C$ cannot be left: $p(C \mid s,x(s)) = 1$ for every~$s \in C$ and every $x(s)\in X(s)$.
\item   For every $s,s' \in C$, there exists a perturbation $y$ of $X$ that allows moving from $s$ to $s'$ without leaving $C$: 
\begin{equation}
\label{eqcomm}
\prob_{s,y}\big(\theta_C^{exit} = \infty\text{ and }\exists t\in\mathbb{N}: s^t=s'\big) \,=\, 1.
\end{equation}
\end{enumerate}
\end{definition}

We mention equivalent alternatives for the conditions in Definition~\ref{definition:communicating}. Condition~1 of Definition~\ref{definition:communicating} could be equivalently stated as follows: For every stationary strategy pair $x\in X$ and every initial state $s\in C$ we have $\prob_{s,x}(\theta_C^{exit} < \infty) = 0$. Condition~2 of Definition~\ref{definition:communicating} could be equivalently stated as follows: For every $s,s' \in C$ and every $\rho > 0$, there exists a perturbation $y$ of $X$ such that \Eqref{eqcomm} holds, and in addition $d(y_i(s),X_i(s))\leq\rho$ for every state $s\in S_0$ and 
every player $i\in\{1,2\}$.

\subsection{Exits from Subsets of States}
\label{sec-exits}

In this subsection, we define two related notions of exit. The terminology of exit was first used by Solan \cite{solan1999three}, and it refers to a pair of actions or mixed actions in a state under which a set of states  is left with a positive probability. First we define unilateral exits, where an action of one of the players is used to leave the set of states under consideration. 

\begin{definition}[unilateral exit]
\label{def:unilateral:exit}
Let $C \subseteq S_0$ be a nonempty set, $X \in \mathcal{X}$, and $i \in \{1,2\}$.
A triplet $(s,a_i,x_j(s))\in C\times A_i\times X_j(s)$
is an \emph{$X$-unilateral exit} of player~$i$ from $C$ if
there exists $x_i(s) \in X_i(s)$ such that
$p(C \mid s,x_i(s),x_j(s)) = 1$ and
$p\bigl(C \mid s,a_i,x_j(s)\bigr) <1$.
\end{definition}

Thus, the set $C$ can be left through the pair $(a_i,x_j(s))$ at state $s$. It is also required that player $i$ has a mixed action $x_i(s)$ that keeps the play in $C$ against $x_j(s)$. 
Let $\mathscr{E}_i(C,X)$ denote the set of all $X$-unilateral exits of player~$i$ from the set $C$.
Note that the set $\mathscr{E}_i(C,X)$ is not necessarily compact.

Now we define joint exits, where the exit only takes place if both players use a certain action pair.

\begin{definition}[joint exit]
\label{def:joint:exit}
Let $C \subseteq S_0$ be a nonempty set and $X \in \mathcal{X}$.
A triplet $(s,a_1,a_2)\in C\times A_1\times A_2$
is an \emph{$X$-joint exit} from $C$ if
there exist $x_1(s) \in X_1(s)$ and $x_2(s) \in X_2(s)$ such that
\begin{itemize}
\item   $p(C \mid s,x_i(s),x_j(s)) = 1$,
\item   $p(C \mid s,a_i,x_j(s)) = 1$ for each $i\in\{1,2\}$,
\item   $p(C \mid s,a_i,a_j) < 1$.
\end{itemize}
\end{definition}

Let $\mathscr{E}_{12}(C,X)$ denote the set of all $X$-joint exits from the set $C$. Note that $\mathscr{E}_{12}(C,X)$ is a finite set, which might be empty.

The next lemma provides a sufficient condition under which a player has a best unilateral exit, 
which maximizes her continuation maxmin value among all her unilateral exits.

\begin{lemma}\label{maxexit}
Let $C \subseteq S_0$ be $X$-communicating for some $X \in \mathcal{X}$. Suppose that $X(s)$ is compact for each $s\in C$ and that there is an exit $(s,x_i(s),a_j)\in\mathscr{E}_j(C,X)$ such that
\begin{equation}
\label{atleastmaxv}
\E[v_j\mid s,x_i(s),a_j]\,\geq\,\max_{s'\in C}v_j(s').
\end{equation}
Then, there is an exit $(s^*,x^*_i(s^*),a^*_j)\in\mathscr{E}_j(C,X)$ that maximizes $\E[v_j\mid\cdot]$ among all exits in $\mathscr{E}_j(C,X)$: for each exit $(s',x'_i(s'),a'_j)\in\mathscr{E}_j(C,X)$
\[\E[v_j\mid s^*,x^*_i(s^*),a^*_j]\,\geq\,\E[v_j\mid s',x'_i(s'),a'_j].\]
\end{lemma}

\begin{proof}
Let
\begin{equation}
\label{maxexiteq}
\gamma\,:=\,\sup_{(s',x'_i(s'),a'_j)\in\mathscr{E}_j(C,X)}\E[v_j\mid s',x'_i(s'),a'_j].
\end{equation}
By \Eqref{atleastmaxv}, we have $\gamma\geq \max_{s'\in C}v_j(s')$. We argue that the maximum in \Eqref{maxexiteq} is attained at some exit. 

If $\gamma = \max_{s'\in C}v_j(s')$, then by \Eqref{atleastmaxv}, the maximum in \Eqref{maxexiteq} is attained at $(s,x_i(s),a_j)$.

So, assume that $\gamma > \max_{s'\in C}v_j(s')$. Let $((s^n,x^n_i(s^n),a^n_j))_{n \in \dN}$ be a sequence of exits in $\mathscr{E}_j(C,X)$
such that 
$\lim_{n \to \infty} \E[v_j\mid s^n,x^n_i(s^n),a^n_j] = \gamma$.
Let $(s^*,x^*_i(s^*),a^*_j)$ be an accumulation point of the sequence $((s^n,x^n_i(s^n),a^n_j))_{n \in \dN}$. 
Since $C$ and $A_2$ are finite, we can assume w.l.o.g.~that $s^n = s^*$ and $a_j^n=a_j^*$ for every $n \in \dN$, and 
since $X_i(s^*)$ is compact, we have $x^*_i(s^*) \in X_i(s^*)$. Consider the triplet $(s^*,x^*_i(s^*),a^*_j)$. By continuity we have
\begin{equation}\label{stars}
\E[v_j\mid s^*,x^*_i(s^*),a^*_j] \,=\, \lim_{n \to \infty} \E[v_j\mid s^*,x^n_i(s),a^*_j] \,=\, \gamma \,>\,\max_{s'\in C}v_j(s').
\end{equation}
It remains to show that $(s^*,x^*_i(s^*),a^*_j) \in \mathscr{E}_j(C,X)$. By \Eqref{stars}, we have $p(C \mid s^*,x^*_i(s^*),a^*_j) < 1$.
So we only need to check that there exists $x^*_j(s^*) \in X_j(s^*)$ such that 
\begin{equation}
    \label{equ:16}
    p(C \mid s^*,x^*_i(s^*),x^*_j(s^*)) = 1.
\end{equation}
For each $n \in \dN$ we have $(s^*,x^n_i(s^*),a^*_j) \in \mathscr{E}_j(C,X)$,
and therefore there is $x^n_j(s^*) \in X_j(s^*)$ such that $p(C \mid s^*,x^n_i(s^n),x^n_j(s))= 1$. As $X_j(s^*)$ is compact, the sequence $(x_j^n(s^*))_{n\in\dN}$ has an accumulation point $x_j^*(s^*)\in X_j^*(s^*)$. By continuity,  \Eqref{equ:16} follows.
\end{proof}

\subsection{Controllable and Blocked Sets}
\label{sec-contr}

In this subsection, we define the notions of controllable sets and blocked sets.

\begin{notation}[$H_i(C,X)$]\rm 
For every nonempty set $C \subseteq S_0$, every player $i \in \{1,2\}$, and every
$X\in \mathcal{X}$,
define
\begin{equation}\label{defH}
 H_i(C,X) := \sup_{(s,a_i,x_j(s))\in C\times A_i\times X_j(s)} \E[v_i \mid s,a_i,x_j(s)].
\end{equation}
\end{notation}

The quantity $H_i(C,X)$ is the maximal expected continuation maxmin value for player~$i$ that she can have when the current state is in $C$ and the other player is restricted to using mixed actions in $X_j(s)$, for $s\in C$. Note that $H_i(C,X)$ is independent of $X_i(s)$, for $s\in C$.
By \Eqref{equ:minmax} we have
\begin{equation}
\label{equ:91}
H_i(C,X) \,\geq\, \max_{s \in C} v_i(s).
\end{equation}

\begin{definition}[set controlled by a player]
\label{def:controlled}
Let $C \subseteq S_0$ be $X$-communicating for some $X \in \mathcal{X}$, and let $i \in \{1,2\}$.
The set $C$ is \emph{$X$-controlled by player~$i$} if
there exists an exit $(s,a_i,x_j(s))\in \mathscr{E}_i(C,X)$
such that
$\E[v_k \mid s,a_i,x_j(s)] \geq H_k(C,X)$ for each $k\in\{1,2\}$.
\end{definition}

The set $C$ is $X$-controlled by player~$i$ if she has an $X$-unilateral exit from $C$ that is beneficial to both players: 
this unilateral exit gives each player $k$ an expected continuation maxmin value that is at least $H_k(C,X)$, and therefore by \Eqref{equ:91}, at least player $k$'s maximal maxmin value in $C$. Note that if this $X$-unilateral exit is $(s,a_i,x_j(s))$, then by \Eqref{defH}, we have $\E[v_i \mid s,a_i,x_j(s)] = H_i(C,X)$.

\begin{definition}[jointly controlled set]
Let $C \subseteq S_0$ be $X$-communicating for some $X \in \mathcal{X}$,
and let $i \in \{1,2\}$.
The set $C$ is \emph{$X$-jointly controlled} if there exists a probability distribution $\mu$ over $\mathscr{E}_{12}(C,X)$
such that for each $k\in\{1,2\}$,
\begin{equation}
\label{eq:jointcont}
\sum_{(s,a_1,a_2) \in \mathscr{E}_{12}(C,X)} \mu(s,a_1,a_2)\cdot \E[v_k \mid s,a_1,a_2] \ \geq\ H_k(C,X).
\end{equation} 
\end{definition}

In words, the set $C$ is $X$-jointly controlled if there is a probability distribution over the $X$-joint exits from $C$ such that the expectation of each player $k$'s continuation maxmin value is at least $H_k(C,X)$.

\begin{definition}[set blocked to a player]
\label{defblocked}
Let $C \subseteq S_0$ be a nonempty set, $X\in \mathcal{X}$,
and $i \in \{1,2\}$. The set $C$ is \emph{$X$-blocked to player~$i$} if for each 
$X$-unilateral exit $(s,a_i,x_j(s)) \in \mathscr{E}_i(C,X)$ of player~$i$,
\begin{equation}
\label{equ:100}
\E[v_i \mid s,a_i,x_j(s)] \,<\, \max_{s'\in C}v_i(s').
\end{equation}
\end{definition}

The set $C$ is blocked to player~$i$ if she has no good unilateral exit: each unilateral exit of player $i$ yields to her an expected continuation maxmin value that is strictly below her best maxmin value in the set $C$.

\begin{lemma}
\label{remark:blocked}
Let $C \subseteq S_0$ be $X$-blocked to player~$i$, for some $X\in \mathcal{X}$ and
 $i \in \{1,2\}$. 
 Then,
 \begin{enumerate}
\item $C$ is not $X$-controlled by player~$i$.
\item Assume, in addition, that $C$ is closed under $X$, i.e., $p(C\mid s,x(s))=1$ for each $s\in C$ and $x(s)\in X(s)$. Then $H_i(C,X) = \max_{s \in C} v_i(s)$.
 \end{enumerate}   
\end{lemma}

\begin{proof} Let $C \subseteq S_0$ be $X$-blocked to player~$i$.\smallskip

\noindent\textsc{Proof of 1.} For every $(s,a_i,x_j(s))\in\mathscr{E}_i(C,X)$, we have by Eqs.~\eqref{equ:100} and \eqref{equ:91}, 
 \[\E[v_i\mid s,a_i,x_j(s)]\,<\,\max_{s' \in C} v_i(s')\,\leq\,H_i(C,X).\]
Hence, $C$ is not $X$-controlled by player~$i$.\smallskip

\noindent\textsc{Proof of 2.} Assume, in addition, that $C$ is closed under $X$. By \Eqref{equ:91}, it suffices to show that $H_i(C,X) \,\leq\, \max_{s' \in C} v_i(s')$. To this end, we take an arbitrary $(s,a_i,x_j(s))\in C\times A_i\times X_j(s)$, and prove that
\begin{equation}
\label{toshoweqH}
\E[v_i\mid s,a_i,x_j(s)]) \,\leq\, \max_{s' \in C} v_i(s').
\end{equation} 
 If $p(C\mid s,a_i,x_j(s))=1$, then \Eqref{toshoweqH} holds by \Eqref{oneshot-g}. 
 So, suppose that $p(C\mid s,a_i,x_j(s))<1$. Since $C$ is closed under $X$, for any $x_i(s)\in X_i(s)$ we have $p(C\mid s,x_i(s),x_j(s))=1$. Hence, $(s,a_i,x_j(s))\in \mathscr{E}_i(C,X)$, and therefore \Eqref{equ:100} implies \Eqref{toshoweqH}. 
\end{proof}

\subsection{Subgame-Perfect $\delta$-Maxmin Strategies}
\label{section:martin}

A strategy for a player is called subgame-perfect $\delta$-maxmin if it ensures that this player's expected payoff is at least her maxmin value minus $\delta$ in all subgames.

\begin{definition}[subgame-perfect $\delta$-maxmin strategy]
Let $\delta > 0$ and $i \in \{1,2\}$.
A strategy $\sigma_i^\delta$ for player $i$ is \emph{subgame-perfect $\delta$-maxmin}
if for every history $h\in H$ and every strategy $\sigma_j$ of player $j$
\[ \E_{h,\sigma_i^\delta,\sigma_j}[f_i] \,\geq\, v_i(s_h) - \delta. \]
\end{definition}

Each player has a subgame-perfect $\delta$-maxmin strategy, for each $\delta>0$, by Mashiah-Yaakovi \cite{mashiah2015correlated} or by Flesch, Herings, Maes, and Predtetchinski \cite{flesch2021subgame}. Note that if $\sigma_i^\delta$ is subgame-perfect $\delta$-maxmin, then for every history $h\in H$ and every action $a_j\in A_j$ of player $j$,
\begin{equation}\label{deltamaxminvalue}
\E[v_i\mid s_h,\sigma_i^\delta(h),a_j] \,\geq\, v_i(s_h) - \delta.
\end{equation}
Indeed, if this inequality was not true, then $\sigma^\delta_i$ would not guarantee a payoff of at least $v_i(s_h)-\delta$ in the subgame starting at $h$; the formal proof is similar to that of \Eqref{equ:minmax}.

We now define for each player and each $\delta>0$ a subgame-perfect $\delta$-maxmin strategy that has specific properties. The definition is based on the following lemma, which follows from Martin \cite{martin1998determinacy} and Maitra and Sudderth \cite{maitra1998finitely}; see also Theorem 3.6 in Ashkenazi-Golan, Flesch, Predtetchinski, and Solan \cite{ashkenazigolan2021equilibria}.

\begin{lemma}
\label{theorem:martin}
Let $\delta > 0$ and $i \in \{1,2\}$.
There is a function $D_i^\delta : H \to [-M,M]$ with the following properties:
\begin{enumerate}
\item $D_i^\delta(h)\leq v_i(s_h)$ for every $h \in H$.
\item $D_i^\delta(h)\leq\val_i(D_i^\delta,h)$ for every $h \in H$.
\item $v_i(s_h) - \delta \leq \val_i(D_i^\delta,h)$ for every $h \in H$.
\item Let $h'\in H$ be a history. If a strategy pair $(\sigma_i,\sigma_j) \in \Sigma_i \times \Sigma_j$ satisfies
\begin{equation}\label{submar}\val_i(D^\delta_i,h)\,\leq\, \E[D^\delta_i\mid h,\sigma_i(h),\sigma_j(h)],\ \quad\forall h \in H\text{ with }h'\preceq h,
\end{equation}
then
\[\val_i(D^\delta_i,h)\,\leq\, \E_{h,\sigma_i,\sigma_j}[f_i],\ \quad\forall h \in H\text{ with }h'\preceq h. \]
\end{enumerate}
\end{lemma}

Lemma~\ref{theorem:martin} states that, for each $\delta>0$ and player $i$, there is a function $D^\delta_i$ that assigns a real number $D^\delta_i(h)$ to each history $h$ with the following properties: 
(1) $D^\delta_i(h)$ is not more than player $i$'s maxmin value at history $h$, 
(2) the value $\val_i(D_i^\delta,h)$ of the one-shot game $G_i(D^\delta_i,h)$, which is induced by $D_i^\delta$ at history $h$, is at least $D^\delta_i(h)$, 
(3) and $\val_i(D_i^\delta,h)$ is also at least the maxmin value of player $i$ at history $h$, up to $\delta$, 
(4) under a pair of strategies $(\sigma_i,\sigma_j)$ in the subgame at $h'$, if in each one-shot game $G_i(D^\delta_i,h)$, the mixed actions played give an expected payoff of at least $\val_i(D_i^\delta,h)$, then in the stochastic game, for each history $h$ extending $h'$, player $i$'s expected payoff in the subgame at $h$ is also at least $\val_i(D^\delta_i,h)$.

Note that if we set $D_i^\delta(h) = v_i(s_h)$,
then properies~(1)--(3) hold, yet property~(4) does not hold.
To ensure that property~(4) holds as well we have to set $D_i^\delta(h)$ to be slightly lower than $v_i(s_h)$,
as expressed by properties~(1) and~(3). 

Lemma~\ref{theorem:martin} follows from Martin \cite{martin1998determinacy} for repeated games, i.e., when $|S|=1$. Indeed, Martin \cite{martin1998determinacy} considers an auxiliary win-lose game between two artifical players, called player I and player II. The function $D^\delta_i$ is in essence a strategy of player I in this game such that this strategy is winning in each subgame. For a precise exposition, we refer to Ashkenazi-Golan, Flesch, Predtetchinski, and Solan \cite{ashkenazigolan2021equilibria}; properties (2)--(4) directly follow from their lemma, whereas property (1) follows from their construction. The extension to a finite state space $S$ follows by the arguments of Maitra and Sudderth \cite{maitra1998finitely}.

From now on we fix a specific function $D^\delta_i$ for each $\delta>0$ and each player $i\in\{1,2\}$ given by Lemma~\ref{theorem:martin}.
We now describe a specific subgame-perfect $\delta$-maxmin strategy that is based on $D_i^\delta$.

\begin{lemma}
\label{lemma:sigma}
Let $\delta > 0$ and $i \in \{1,2\}$.
Let $\sigma_i^\delta$ be a strategy for player $i$ such that,
for every history $h \in H$, $\sigma_i^\delta(h)$
is an optimal mixed action of player~$i$ in the one-shot zero-sum game $G_i(D^\delta_i,h)$. Then, the strategy $\sigma_i^\delta$ is subgame-perfect $\delta$-maxmin. Moreover, for every history $h$ and every action $a_j\in A_j$,
\begin{equation}\label{propsubgameopt}
D^{\delta}_i(h)\,\leq\,\val_i(D^{\delta}_i,h)\,\leq\,\E[D^{\delta}_i \mid h,\sigma_i^{\delta}(h),a_j].
\end{equation}
\end{lemma}

\begin{proof}
For every history $h \in H$ and action $a_j \in A_j$, \Eqref{propsubgameopt} follows by 
property 2 of Lemma~\ref{theorem:martin} and by the fact that $\sigma_i^\delta(h)$ is an optimal mixed action of player~$i$ in $G_i(D^\delta_i,h)$.

It remains to prove that $\sigma_i^\delta$ is subgame-perfect $\delta$-maxmin.
Take any strategy $\sigma_j$ for player $j$. 
For each history~$h$, as $\sigma_i^\delta(h)$ is an optimal mixed action of player~$i$ in $G_i(D^\delta_i,h)$, 
\[\val_i(D^{\delta}_i,h)\,\leq\,\E[D^{\delta}_i \mid h,\sigma_i^{\delta}(h),\sigma_j(h)].\]
Since $h$ is arbitrary, 
the strategy pair $(\sigma_i,\sigma_j)$ satisfies \Eqref{submar} for each history $h$ 
(when $h'$ is the initial history, consisting of the initial state only). Hence, by properties 3 and 4 of Lemma~\ref{theorem:martin}, for each history $h$, we have $v_i(s_h)-\delta\leq \val_i(D^\delta_i,h)\leq \E_{h,\sigma_i^\delta,\sigma_j}[f_i]$. Thus, $\sigma_i^\delta$ is subgame-perfect $\delta$-maxmin.
\end{proof}\medskip

\begin{notation}[the strategy $\sigma_i^\delta$]\rm 
\label{notation:sigma}
For each $\delta>0$ and each player $i \in \{1,2\}$
fix a strategy $\sigma_i^\delta$ given by Lemma~\ref{lemma:sigma}.
These strategies will be fixed from now on.
\end{notation}

\subsection{Controllability w.r.t.~Subgame-Perfect $\delta$-Maxmin Strategies}\label{sec-contr-sgperf}


\begin{notation}[The sets $Y_i^\delta(s)$ and $Y_i(s)$]\rm 
For each payer $i\in\{1,2\}$, each $\delta>0$, and each state $s \in S_0$, let $Y_i^\delta(s)$ denote the range of $\sigma_i^\delta$, restricted to histories that end in state~$s$:
\[ Y_i^\delta(s) \,:=\, \{ \sigma_i^\delta(h) \colon h \in H,\,s_h=s\} \,\subseteq\, \Delta(A_i), \]
where $\sigma_i^\delta$ is defined in Notation~\ref{notation:sigma}.
Let $Y_i(s)$ denote the set of all accumulation points of $(Y_i^\delta(s))_{\delta > 0}$ as $\delta$ goes to 0:
\[ Y_i(s) \,:=\, \lim_{\delta \to 0} Y_i^\delta(s) \,=\, \bigcap_{\delta > 0}\; \overline{\bigcup_{\delta'\in(0,\delta)} Y_i^{\delta'}(s)} \,\subseteq\, \Delta(A_i), \]
where $\overline{\bigcup_{\delta'\in(0,\delta)} Y_i^{\delta'}(s)}$ is the closure of $\bigcup_{\delta'\in(0,\delta)} Y_i^{\delta'}(s)$.
The set $Y_i(s)$ is compact and
\begin{equation}
\label{equ:limit:y}
\lim_{\delta \to 0} d(Y^\delta_i(s),Y_i(s)) \,=\, 0.
\end{equation}
\end{notation}

We briefly argue that \Eqref{equ:limit:y} holds.
Indeed, suppose by way of contradiction that $\eta:=\limsup_{\delta \to 0} d(Y^\delta_i(s),Y_i(s)) > 0$.
Then, there exist a sequence $(\delta_n)_{n \in \dN}$ of positive reals converging to 0
and a sequence $(x_i^{n}(s))_{n \in \dN}$ such that $x_i^{n}(s) \in Y_i^{\delta_n}(s)$ for each $n \in \dN$
and $\lim_{n \to \infty} d(x_i^{n}(s),Y_i(s)) =\eta$.
By taking a subsequence if necessary, we can assume that $(x_i^{n}(s))_{n \in \dN}$ converges to some $x_i(s)\in Y_i(s)$. 
By continuity, we obtain $d(x_i(s),Y_i(s))=\eta>0$, which is a contradiction. 

\begin{notation}[The product set $Y$]\rm 
Let 
\[Y\,:=\,\prod_{s\in S_0}\,\prod_{i=1,2}Y_i(s)\in\mathcal{X}.\]
\end{notation}

The following result is based on continuity and compactness of $Y$.
It states that if $C\subseteq S_0$ is $Y$-communicating, then for $\delta>0$ small, whenever an action $a_i\in A_i$ is played with large probability under some $y_i^\delta(s)\in Y_i^\delta(s)$, the action $a_i$ keeps the play in $C$ with large probability against any mixed action $y_j(s)\in Y_j(s)$ of player $j$.

\begin{lemma}
\label{lemma:67}
Let $C\subseteq S_0$ be $Y$-communicating,
and let $i \in \{1,2\}$. For every $\eta_1,\eta_2 > 0$ there is $w =w(\eta_1,\eta_2)> 0$ such that for every $\delta \in (0,w)$, every $s \in C$, every $y_i^\delta(s) \in Y_i^\delta(s)$, every $y_j(s) \in Y_j(s)$, and every $a_i \in A_i$, if $y^\delta_i(s,a_i) \geq \eta_1$ then $p\bigl((S \setminus C) \mid s,a_i,y_j(s)\bigr) \leq \eta_2$.
\end{lemma}

\begin{proof}
Suppose the opposite.
Then, there are $\eta_1,\eta_2 > 0$, $s \in C$, $a_i \in A_i$,
and a sequence $(\delta_k)_{k \in \dN}$ converging to 0,
such that for every $k \in \dN$ there are
$y_i^k(s) \in Y_i^{\delta_k}(s)$ and $y_j^k(s) \in Y_j(s)$ that satisfy
$y_i^k(s,a_i) \geq \eta_1$ and
$p((S\setminus C) \mid s,a_i,y_j^k(s)) \geq \eta_2$.

Assume w.l.o.g.~that the sequences $(y_i^k(s))_{k \in \dN}$ and $(y_j^k(s))_{k \in \dN}$
converge to some $y_i^*(s)\in \Delta(A_i)$ and $y_j^*(s)\in \Delta(A_j)$, respectively. 
Since $\lim_{k \to \infty} \delta_k = 0$ and since $Y_j(s)$ is compact, 
we have $y^*_i(s)\in Y_i(s)$ and $y^*_j(s) \in Y_j(s)$.
Hence, because $C$ is $Y$-communicating, $p((S\setminus C) \mid s,y^*_i(s),y^*_j(s)) = 0$.

On the other hand,
\[ p\bigl((S \setminus C) \mid s,y_i^k(s),y_j^k(s)\bigr)
\,\geq\, y_i^k(s,a_i) \cdot p\bigl((S \setminus C) \mid s,a_i,y_j^k(s)\bigr)
\,\geq\, \eta_1 \cdot \eta_2,\]
and letting $k\to\infty$ we obtain
$p\bigl((S \setminus C) \mid s,y^*_i(s),y^*_j(s)\bigr) \,\geq\, \eta_1 \cdot \eta_2>0$,
a contradiction.
\end{proof}\\

The next result states that when at state $s$ player~$j$ plays a mixed action in $Y_j(s)$,
she guarantees that the expected continuation maxmin value is at least $v_j(s)$.

\begin{lemma}
\label{lemma:expected:minmax}
For every $s \in S$, every $i\in \{1,2\}$, every $a_i \in A_i$, and every $y_j(s) \in Y_j(s)$,
\begin{equation}
\label{equ:81}
\E[v_j \mid s,a_i,y_j(s)] \,\geq\, v_j(s).
\end{equation}
\end{lemma}

\begin{proof}
Since $y_j(s) \in Y_j(s)$, there is a sequence $(\delta_k)_{k \in \dN}$ converging to 0 and a sequence $(h_k)_{k \in \dN}$ of histories ending in state $s$, such that $y_j(s) = \lim_{k \to \infty} \sigma_j^{\delta_k}(h_k)$. Because $\sigma^{\delta_k}_j$ is subgame-perfect $\delta^k$-maxmin, by \Eqref{deltamaxminvalue} we have for every $k\in\dN$ and action $a_i\in A_i$ of player $i$,
\[\E[v_j\mid s,a_i,\sigma_j^{\delta_k}(h_k)] \,\geq\, v_j(s) - \delta_k.\]
The claim follows by taking the limit as $k\to\infty$.
\end{proof}\bigskip

The following result states that if $C$ is $Y$-communicating, 
neither player has a good $Y$-unilateral exit from $C$, and the maxmin value of each player is constant over $C$, 
then $C$ must be $Y$-jointly controlled.
It is the analog of Proposition~32 in Vieille \cite{vieille2000one}, called the property of alternative.
Yet, our proof greatly differs from that of Proposition~32 in Vieille \cite{vieille2000one}.

\begin{lemma}
\label{lemma:alternatives}
Let $C \subseteq S_0$ be such that
(1) $C$ is $Y$-communicating,
(2) $C$ is $Y$-blocked to each player, and
(3) $v_1$ and $v_2$ are constant over $C$, i.e., $v_i(s)=v_i(s')$ for all $s,s'\in C$ and $i\in\{1,2\}$.
Then, the set $C$ is $Y$-jointly controlled.
\end{lemma}

\begin{remark}\rm
\label{remark:explanation}
We here explain the most important differences between our proof and the proof of Proposition~32 in Vieille \cite{vieille2000one}. 
These two proofs differ in two main aspects: 
(i) 
As mentioned in Section~\ref{section:main:result}, 
in the setup of Vieille \cite{vieille2000one}, the set $Y$  is a singleton. In particular, this implies that the difference 
between the right-hand side and the left-hand side of
Eq.~\eqref{equ:100} has a uniform positive lower-bound.
In our setup, there is no such uniform positive lower-bound.
(ii) Both proofs contain an important step showing that, if the set $C$ is not jointly controlled, then there is an $\ep$-equilibrium for the initial states in the set $C$, for each $\ep>0$. 
In Vieille \cite{vieille2000one},
payoffs are long-run average,
and such an $\ep$-equilibrium can be constructed using a periodic construction,
which cycles among the states in $C$.
In our setup, payoffs are general,
and so no periodic construction will do.
We therefore use tools developed by Martin~\cite{martin1998determinacy}, 
as discussed in Section~\ref{section:martin}.

Now we futher discuss each proof separately with respect to these aspects. \smallskip

\noindent\textsc{The proof of Proposition~32 in Vieille \cite{vieille2000one}.} In Vieille \cite{vieille2000one}, each player is maximizing the long-run average of her stage payoffs. In this case, each player $i$ has a subgame-perfect $\delta$-maxmin strategy $\sigma^\delta_i$ for each $\delta>0$ with the following properties: (1) At each history $h$, the mixed action $\sigma^\delta_i(h)$ is optimal for player $i$ in the zero-sum stochastic game $\Gamma_{i,\lambda_{h,\delta}}$, where $\Gamma_{i,\lambda_{h,\delta}}$ is derived from the original game $\Gamma$ as follows: $\lambda_{h,\delta}$ is a fictitious discount factor depending on $h$ and $\delta$, player $i$ is maximizing her own $\lambda_{h,\delta}$-discounted payoff, while her opponent is minimizing player $i$'s $\lambda_{h,\delta}$-discounted payoff. (2) As $\delta$ goes to 0, the maximal fictitious discount factor that is used by $\sigma^\delta_i$ goes to 0, i.e., $\lim_{\delta\to 0}\sup\{\lambda_{h,\delta}:h\in H\}=0$. The existence of such a strategy $\sigma^\delta_i$ is due to the results of Mertens and Neyman \cite{mertens1981stochastic}; see also Section 6.2 in Vieille \cite{vieille2000one}.

(i) In a zero-sum stochastic game, there is a semi-algebraic function that assigns a stationary discounted optimal strategy
to each discount factor, see for instance Neyman \cite{neyman2003real}, Section 6.3.1 in Vieille \cite{vieille2000one},
or Chapter~7 in Solan \cite{maschler2013game}. 
It follows that the set $Y_i(s)$ 
can be taken to be
a singleton for each player $i\in\{1,2\}$ and each state $s\in C$. Hence, the set $Y$ is a singleton as well. 

In particular, as $Y$ is a singleton, the number of $Y$-unilateral exits of each player is finite.
As a result, because $C$ is $Y$-blocked to player~$i$, the difference 
between the right-hand side and the left-hand side of
in Eq.~\eqref{equ:100} has a uniform positive lower-bound.

(ii) In Lemma 37 in Vieille \cite{vieille2000one}, the following claim is shown under conditions similar to the ones in our Lemma \ref{lemma:alternatives}: if the set $C$ is not jointly controlled, then there is an $\ep$-equilibrium for the initial states in the set $C$, for each $\ep>0$. Assuming that $C$ is not jointly controlled, for large discount factors, under the pair of stationary discounted optimal strategies, the play remains forever in the set $C$ with positive probability. Using this, Vieille's proof is based on analysing the behavior of the stationary discounted optimal strategies of the players and their limit, the unique stationary strategy in $Y$.\smallskip

\noindent\textsc{The proof of Lemma \ref{lemma:alternatives}.} In this case, the payoff functions of the players are only assumed to be bounded, Borel-measurable, and shift-invariant.

(i) The set $Y$ need not be a singleton, in contrast with Vieille's proof. As a consequence, the number of $Y$-unilateral exits of a player may be infinite, and hence the difference 
between the right-hand side and the left-hand side of
in Eq.~\eqref{equ:100} may not have a uniform positive lower-bound.

(ii) As in the proof in Vieille \cite{vieille2000one}, 
we need to show a similar claim, cf.~Step 3 in the proof of Lemma \ref{lemma:alternatives}: if the set $C$ is not jointly controlled, then there is an $\ep$-equilibrium for the initial states in the set $C$, for each $\ep>0$ (which leads to a contradiction with Assumption \ref{As3}). In our proof, if $C$ is not jointly controlled, then, for small $\delta>0$, under some modifications of the pair $(\sigma^\delta_1,\sigma^\delta_2)$ of subgame-perfect $\delta$-optimal strategies, the play remains forever in the set $C$ with positive probability. 
Our proof to show that an $\ep$-equilibrium exists in this case relies on  different techniques: inner regularity of probability measures and L\'{e}vy's zero-one law. 
\end{remark}

\begin{proof}[Proof of Lemma~\ref{lemma:alternatives}] 
Let $C\subseteq S_0$ satisfy Conditions~(1), (2), and~(3) in the statement of Lemma~\ref{lemma:alternatives}. For each $i\in\{1,2\}$, we denote by $v_i(C)$ the common maxmin value of player $i$ for the initial states in $C$. \medskip

\noindent\textbf{Step 0}: Fixing a constant $\alpha$, and for all $\rho>0$ fixing positive reals $\eta^\rho$, $\kappa^\rho$, and $\delta^\rho$.\smallskip

\noindent\textsc{Fixing $\alpha$:} Fix $\alpha:=\max_{i\in\{1,2\}} |A_i|$. \smallskip

\noindent\textsc{Fixing $\eta^\rho$ for each $\rho>0$:} Take any $\rho>0$. For every $i\in\{1,2\}$, the set of $Y$-unilateral exits of player~$i$ that lead the play outside $C$ with probability at least $\rho$ is
\[ \mathscr{E}_i^\rho(C,Y)\, :=\, \bigl\{(s,a_i,y_j(s)) \in C\times A_i\times Y_j(s) \colon p(C \mid s,a_i,y_j(s)) \leq 1-\rho\bigr\}; \]
thus $\mathscr{E}_i^\rho(C,Y)\subseteq \mathscr{E}_i(C,Y)$.
Note that $\mathscr{E}_i^\rho(C,Y)$ is compact, as $C$ and $A_i$ are finite,
$Y_j(s)$ is compact,
and $p(C \mid s,a_i,\cdot)$ is continuous.
For each $(s,a_i,y_j(s)) \in \mathscr{E}_i^\rho(C,Y)$, it holds that $\E[v_i\mid s,a_i,y_j(s)]\,<\,v_i(C)$, because $C$ is $Y$-blocked to player $i$.
Let 
\[\eta^\rho\,:=\,\min_{i\in\{1,2\}}\,\min_{(s,a_i,y_j(s)) \in \mathscr{E}_i^\rho(C,Y)}\Big(v_i(C)-\E[v_i\mid s,a_i,y_j(s)]\Big) > 0,\]
where the inequality holds by continuity and the compactness of $\mathscr{E}_i^\rho(C,Y)$. 
The function $\rho \mapsto \eta^\rho$ is monotone 
nondecreasing: if $0<\rho'\leq\rho$ then $\mathscr{E}_i^{\rho'}(C,Y)\supseteq \mathscr{E}_i^\rho(C,Y)$, and hence $\eta^{\rho'}\leq\eta^\rho$. 
By definition, for each player $i\in\{1,2\}$ and each exit $(s,a_i,y_j(s)) \in \mathscr{E}_i^\rho(C,Y)$ we have
\begin{equation}\label{choice-eta} 
\E[v_i \mid s,a_i,y_j(s)] \,\leq\, v_i(C) - \eta^\rho.
\end{equation}
\noindent\textsc{Fixing $\kappa^\rho$ for each $\rho>0$.} Take any $\rho>0$. Let $\kappa^\rho\in(0,1)$ be small enough so that 
\begin{equation}
\label{choicekappa}
\rho-(\alpha+1)\kappa^\rho>0\quad\text{ and }\quad\alpha\kappa^\rho\cdot 4M-\eta^{\overline{\rho}}<0,
\end{equation} 
where $\overline{\rho}=\rho-\alpha\kappa^\rho$. We argue briefly that the second inequality is satisfied for small $\kappa^\rho\in(0,1)$. 
Indeed, choose $\kappa^\rho\in(0,1)$ 
sufficiently small
so that $\alpha\kappa^\rho\cdot 4M-\eta^{\frac{1}{2}\rho}<0$ and $\overline{\rho}\geq \frac{1}{2}\rho$. Then, $\eta^{\overline{\rho}}\geq \eta^{\frac{1}{2}\rho}$, and hence $\alpha\kappa^\rho\cdot 4M-\eta^{\overline{\rho}}\leq\alpha\kappa^\rho\cdot 4M-\eta^{\frac{1}{2}\rho}<0$, as desired.\smallskip

\noindent\textsc{Fixing $\delta^\rho$ for each $\rho>0$.} Take any $\rho>0$. By Eq.~\eqref{equ:limit:y}, there is $\delta^\rho \in (0,\kappa^\rho)$ such that for every $\delta \in (0,\delta^\rho)$, every state $s\in S_0$, and each player $i\in\{1,2\}$, we have
\begin{equation}
\label{equ:5.1}
d(Y^\delta_i(s), Y_i(s)) \,<\, \kappa^\rho.
\end{equation}
Assume w.l.o.g.~that $\delta^\rho > 0$ is sufficiently small such that 
\begin{equation}\label{ineq20}
\delta^\rho\,<\,w\Big(\frac{\kappa^\rho}{\alpha},\rho-(\alpha+1)\kappa^\rho\Big),
\end{equation}
where the function $w(\cdot,\cdot)$ is given in Lemma~\ref{lemma:67}, and 
\begin{equation}\label{combineq}
\delta^\rho+\alpha\kappa^\rho\cdot 4M-\eta^{\overline{\rho}}<0,
\end{equation}
which is possible due to \Eqref{choicekappa}. 
By Lemma~\ref{lemma:67} and \Eqref{ineq20}, for every $\delta \in (0,\delta^\rho)$, every $s\in C$, 
every $y_i^\delta(s) \in Y_i^\delta(s)$,
every $y_j(s) \in Y_j(s)$,
and every $a_i \in A_i$ such that $y_i^\delta(s,a_i)  \geq \kappa^\rho/\alpha$,
we have
\begin{equation}
\label{equ:14.1}
p\bigl( (S \setminus C) \mid s,a_i,y_j(s)\bigr) \leq \rho - (\alpha+1)\kappa^\rho < \rho. 
\end{equation}

\bigskip
\noindent\textbf{Step 1}: Definition of a strategy pair $\pi^{\rho,\delta}=(\pi^{\rho,\delta}_1,\pi^{\rho,\delta}_2)$ for each $\rho>0$ and $\delta \in (0,\delta^\rho)$.

Let $\rho>0$ and $\delta\in (0,\delta^\rho)$. In this step we define a strategy pair $\pi^{\rho,\delta}$,
which is derived from $\sigma^\delta$
by setting to 0 the probability to play actions that lead outside $C$ with high probability under $\sigma^\delta(h)$, for every history $h$ (and normalizing the remaining probabilities).

For each history $h \in H$, denote by $B^{\rho,\delta}_i(h)$ the set of actions of player~$i$ that, when faced with $\sigma^\delta_j(h)$,
lead the play outside the set $C$ with probability at least $\rho$:
\begin{equation}
\label{equ:B}
B^{\rho,\delta}_i(h) \,:=\,\left\{ a_i \in A_i \colon
p\bigl( (S\setminus C)\mid s_h,a_i,\sigma^\delta_j(h)\bigr) \geq\rho\right\}. 
\end{equation}

We argue that, for each $h\in H$ with final state $s_h\in C$, the mixed action $\sigma^\delta_i(h)$ assigns a low probability to actions in $B^{\rho,\delta}_i(h)$:
\begin{equation}
\label{equ:6.1}
\sigma^\delta_i\left(B^{\rho,\delta}_i(h) \mid h\right) \,<\, \kappa^\rho.
\end{equation}
Indeed, let $a_i \in B^{\rho,\delta}_i(h)$, so that
$p\bigl( (S \setminus C) \mid s_h,a_i,\sigma^\delta_j(h)) \,\geq\, \rho$.
By \Eqref{equ:5.1}, there is $y_j(s_h)\in Y_j(s_h)$ satisfying $d(\sigma^\delta_j(h),y_j(s_h)) < \kappa^\rho$.
Then
\begin{eqnarray} 
\nonumber
p \bigl( (S \setminus C) \mid s_h,a_i,y_j(s_h)) &\geq& 
p\bigl( (S \setminus C) \mid s_h,a_i,\sigma^\delta_j(h))- \alpha\kappa^\rho\\
&\geq& \rho - \alpha\kappa^\rho \,>\, \rho - (\alpha+1)\kappa^\rho\,>\,0,
\label{equ:15.1}
\end{eqnarray}
where the last inequality is by \Eqref{choicekappa}. 
\Eqref{equ:14.1} implies that $\sigma^\delta_i(a_i \mid h) < \kappa^\rho/\alpha$.
Eq.~\eqref{equ:6.1} follows since $\alpha \geq |A_i|$.

For each $i\in \{1,2\}$, we now define a strategy $\pi^{\rho,\delta}_i$ that follows $\sigma_i^\delta$,
except that at each history $h$ the strategy $\pi^{\rho,\delta}_i$ does not play actions in $B_i^{\rho,\delta}(h)$. Recall that a history is said to remain in $C$ if all states along this history belong to $C$. 
We define
$\pi^{\rho,\delta}_i(h) := \sigma^\delta_i(h)$ for all histories that do \emph{not} remain in $C$,
and for histories $h$ that remain in $C$ we define
\begin{equation}
\label{equ:6.2}
\pi^{\rho,\delta}_i(a_i \mid h) := \left\{
\begin{array}{lll}
0, & \ \ \ \ \ & \text{if }a_i \in B_i^{\rho,\delta}(h),\\
\frac{\sigma^\delta_i(a_i \mid h)}{\sigma^\delta_i\left(A_i\setminus B_i^{\rho,\delta}(h) \mid h\right)}, & & \text{if }a_i \not\in B_i^{\rho,\delta}(h),
\end{array}
\right.
\end{equation}
which is exactly the mixed action $\sigma_i^\delta(h)$ conditioned on not choosing actions in $B_i^{\rho,\delta}(h)$.
Eq.~\eqref{equ:6.1} implies that the denominator in Eq.~\eqref{equ:6.2} is positive,
hence $\pi_i^{\rho,\delta}$ is well defined. Eq.~\eqref{equ:6.1} also implies that for each player $i\in\{1,2\}$ and each history $h\in H$,
\begin{equation}
\label{equ:6.3}
d(\sigma^\delta_i(h), \pi^{\rho,\delta}_i(h)) \,\leq\, \kappa^\rho.
\end{equation}
For each history $h\in H$ and each action $a_i \in A_i\setminus B_i^{\rho,\delta}(h)$, we have
\begin{equation}
\label{equ:201}
p \bigl( (S \setminus C) \mid s_h,a_i,\pi^{\rho,\delta}_j(h)\bigr) \,\leq\,
p \bigl( (S \setminus C) \mid s_h,a_i,\sigma^\delta_j(h) \bigr)  + \alpha\kappa^\rho \,\leq\, \rho + \alpha\kappa^\rho,
\end{equation}
where the first inequality is by \Eqref{equ:6.3} and the second inequality is by \Eqref{equ:B}.

\bigskip
\noindent\textbf{Step 2}: For every $\rho>0$ and every $\delta\in(0,\delta^\rho)$, the following properties hold for the strategy pair $(\pi^{\rho,\delta}_1,\pi^{\rho,\delta}_2)$:
\begin{enumerate}
\item The payoff is high:
for each player $i \in \{1,2\}$ and each history $h'\in H$,
\begin{equation}
\label{equ:6.4}
\E_{h',\pi^{\rho,\delta}_1,\pi^{\rho,\delta}_2}[f_i]\, \geq\, \val_i(D^\delta_i,h'),
\end{equation}
and if, moreover, $h'$ remains in $C$, then
\begin{equation}
\label{equ:6.44}
\E_{h',\pi^{\rho,\delta}_1,\pi^{\rho,\delta}_2}[f_i] \,\geq\, v_i(C) - \delta.
\end{equation}
\item The maxmin values cannot decrease in expectation much during the play: Consider any random variable $\xi:\mathscr{R}\to\dN$  with the following property:%
\footnote{This property means that $\xi$ is a finite stopping time where stopping in stage $t$ is independent of the actions in stage $t$.}
For every run $r=(s^1,a^1,s^2,a^2,\ldots)$, if $\xi(r)=t$, then $\xi(r')=t$ for every run $r'=(s'^1,a'^1,s'^2,a'^2,\ldots)$ that coincides with $r$ up to the state in stage $t$, i.e., $s^1=s'^1,a^1=a'^1,\ldots,s^t=s'^t$. Let $h\in H$ be a history such that $\xi(r)>\len(h)$ for all (or equivalently, for some) run $r\succ h$. Then, for each player $i\in\{1,2\}$,
\begin{equation}
\label{stoptime}
v_i(s_h)-\delta\,\leq\,\E_{h,\pi^{\rho,\delta}_1,\pi^{\rho,\delta}_2}\big[v_i(s^\xi)\big];
\end{equation}
here $s^{\xi}$ is the state in stage $\xi$.
\end{enumerate}
Fix $\rho>0$, $\delta\in(0,\delta^\rho)$, and $i \in \{1,2\}$.

First we prove that for each history $h\in H$ we have
\begin{equation}
\label{equ:72}
\val_i(D^\delta_i,h)\,\leq\,\E[D_i^\delta \mid h,\pi^{\rho,\delta}_i(h),\pi^{\rho,\delta}_j(h)].
\end{equation}
Let $h\in H$. If $h$ does not remain in $C$, then $\pi^{\rho,\delta}_i(h) = \sigma^\delta_i(h)$, and hence \Eqref{equ:72} follows from the choice of $\sigma_i^\delta$, cf.~Lemma \ref{lemma:sigma}. 

Assume now that $h$ remains in $C$,
and let $a_i \in B_i^{\rho,\delta}(h)$. By \Eqref{equ:5.1}, there is $y_j(s_h) \in Y_j(s_h)$ such that $d\left(\sigma^\delta_j(h),y_j(s_h)\right) < \kappa^\rho$. 
By the definition of $B_i^{\delta,\rho}(h)$ and \Eqref{choicekappa}, 
\begin{equation}
\label{closetorho}
 p\bigl( (S \setminus C) \mid s_h,a_i,y_j(s_h)) \,\geq\, 
 p\bigl( (S \setminus C) \mid s_h,a_i,\sigma^\delta_j(h)) -\alpha\kappa^\rho
 \,\geq\, \rho-\alpha\kappa^\rho\,>\,0.
\end{equation}
This implies that $(s_h,a_i,y_j(s_h))\in \mathscr{E}_i^{\overline{\rho}}(C,Y)$ for $\overline{\rho}=\rho-\alpha\kappa^\rho$.
We then have
\begin{eqnarray}
\label{equ:long:1}
\E[D_i^\delta \mid h,a_i,\pi^{\rho,\delta}_j(h)]
&\leq&
\E[D_i^\delta \mid h,a_i,\sigma^{\delta}_j(h)] + \alpha\kappa^\rho\cdot 2 M\\\label{equ:long-e}
&\leq&
\E[D_i^\delta \mid h,a_i,y_j(s_h)] + \alpha\kappa^\rho\cdot 4M\\
\label{equ:long:2}
&\leq&
\E[v_i \mid s_h,a_i,y_j(s_h)] + \alpha\kappa^\rho\cdot 4M \\
\label{equ:long:3}
&\leq& v_i(C) - \eta^{\overline{\rho}} + \alpha\kappa^\rho\cdot 4M\\
\label{equ:long:4}
&\leq& \val_i(D^\delta_i,h)+\delta  - \eta^{\overline{\rho}} + \alpha\kappa^\rho\cdot 4M\\
\label{equ:long:5}
&<& \val_i(D^\delta_i,h),
\end{eqnarray}
where
Eq.~\eqref{equ:long:1} holds by Eq.~\eqref{equ:6.3};
Eq.~\eqref{equ:long-e} holds by the choice of $y_j(s_h)$;
Eq.~\eqref{equ:long:2} holds
by Property 1 of Lemma~\ref{theorem:martin};
Eq.~\eqref{equ:long:3} holds 
since $C$ is $Y$-blocked to player~$i$, see \Eqref{choice-eta};
Eq.~\eqref{equ:long:4} holds by Property 3 of Lemma~\ref{theorem:martin};
and Eq.~\eqref{equ:long:5} holds by \Eqref{combineq}.

By Lemma \ref{lemma:sigma} and since the expectation is linear, it follows that
\begin{eqnarray*}
\val_i(D^{\delta}_i,h) &\leq& \E[D_i^\delta \mid h,\sigma_i^\delta(h),\pi^{\rho,\delta}_j(h)] \\
&=& \sum_{a_i\in B_i^{\rho,\delta}(h)} \sigma_i^\delta(a_i\mid h)\cdot\E[D_i^\delta \mid h,a_i,\pi^{\rho,\delta}_j(h)] \\
& & +\ \sigma^\delta_i\big(A_i\setminus B^{\rho,\delta}_i(h) \mid h\big)\cdot \E[D_i^\delta \mid h,\pi_i^{\rho,\delta}(h),\pi^{\rho,\delta}_j(h)].
\end{eqnarray*}
In view of Eqs.~\eqref{equ:long:1}--\eqref{equ:long:5}, we conclude that \Eqref{equ:72} holds, as desired.\smallskip

We now prove Property 1 of Step 2. 
For every history $h'\in H$, \Eqref{equ:6.4} holds by Property 4 of Lemma~\ref{theorem:martin} and \Eqref{equ:72}. Subsequently, for every history $h'\in H$, \Eqref{equ:6.44} follows by Property 3 of Lemma~\ref{theorem:martin}. \smallskip

We turn to prove Property 2 of Step 2. 
Fix any random variable $\xi$ and any history $h\in H$ satisfying the conditions in Property 2 of Step 2, and any player $i\in\{1,2\}$. 
By Property 2 of Lemma~\ref{theorem:martin} and by \Eqref{equ:72}, for every history $h'\in H$
\begin{equation}
\label{subm}
D^\delta_i(h')\,\leq\,\val_i(D^\delta_i,h')\,\leq\,\E[D^\delta_i\mid h',\pi_i^{\rho,\delta}(h'),\pi_j^{\rho,\delta}(h')],
\end{equation}
which implies that $D_i^\delta(\cdot)$ is a submartingale under $(\pi_i^{\rho,\delta},\pi_j^{\rho,\delta})$. 
We therefore have 
\begin{align*}
v_i(s_h)-\delta\,&\leq\,\val_i(D^\delta_i,h)\\
&\leq\,\E[D^\delta_i\mid h,\pi_i^{\rho,\delta}(h),\pi_j^{\rho,\delta}(h)]\\
&\leq\,\E_{h,\pi_i^{\rho,\delta},\pi_j^{\rho,\delta}}[D_i^\delta(h^\xi)]\\
&\leq\,\E_{h,\pi_i^{\rho,\delta},\pi_j^{\rho,\delta}}[v_i(s^\xi)],
\end{align*}
where the first inequality follows by Property 3 of Lemma~\ref{theorem:martin}, the second inequality follows by \Eqref{subm}, the third inequality follows by the optional stopping theorem for submartingales, and the last inequality follows by Property 1 of Lemma~\ref{theorem:martin}. 
The proof of Property 2, and thereby also the proof of Step 2, are complete. \bigskip

For every initial state $s \in C$, every $\rho>0$, and every $\delta\in(0,\delta^\rho)$,
the probability under $(\pi^{\rho,\delta}_1,\pi^{\rho,\delta}_2)$ that the play eventually leaves $C$
is
\begin{equation}
\label{def-nu}
\nu_s^{\rho,\delta}\, :=\, \prob_{s,\pi^{\rho,\delta}_1,\pi^{\rho,\delta}_2}(\theta^{exit}_C < \infty).
\end{equation}
For every initial state $s \in C$ and every $\rho>0$,
let $\nu^\rho_s$ be an accumulation point of $(\nu_s^{\rho,\delta})_{\delta \in(0,\delta^\rho)}$ as $\delta$ goes to 0. 
For every initial state $s \in C$,
let $\nu_s$ be an accumulation point of $(\nu_s^\rho)_{\rho>0}$ as $\rho$ goes to 0.

\bigskip
\noindent\textbf{Step 3}: $\nu_{s} = 1$ for each initial state $s\in C$.

Assume by way of contradiction that there is an initial state $s\in C$ such that $\nu_s<1$.
We will prove that, for each $\ep>0$, there exists a state $s^*\in C$ such that if the initial state is $s^*$, then there exists an $\ep$-equilibrium.
Since $C$ is finite, this will imply that there is an initial state $s^{**}\in C$ for which there is an $\ep$-equilibrium for each $\ep>0$. As $s^{**}\in C\subseteq S_0$, this will lead to a contradiction with Assumption \ref{As2}.

The idea of the argument is as follows.
Since $\nu_s<1$, it follows that for $\rho$ and $\delta$ sufficiently small, 
$\nu_s^{\rho,\delta}$ is bounded away from 1.
Thus, under $(\pi^{\rho,\delta}_1,\pi^{\rho,\delta}_2)$ the play always remains in $C$ with probability bounded away from 0.
This implies that there is a history $h = h(\rho,\delta)$ that starts at $s$ and such that,
in the subgame that starts at $h$,
under $(\pi^{\rho,\delta}_1,\pi^{\rho,\delta}_2)$ the play remains in $C$ with probability arbitrarily close to 1.
Since $\pi^{\rho,\delta}_1$ and $\pi^{\rho,\delta}_2$ are subgame-perfect $\delta$-maxmin strategies,
the payoff to each player in the subgame that starts at $h$ under $(\pi^{\rho,\delta}_1,\pi^{\rho,\delta}_2)$
is at least the maxmin value at $s_h$ minus $\delta$.
Since $C$ is blocked to both players,
no player can improve her payoff by much by playing an action she is not supposed to play,
provided such a deviation triggers punishment at the maxmin level 
(which coincides with the minmax level as the game involves two players).
A deviation that may be profitable to player~$i$ is to play within the support of $\pi^{\rho,\delta}_i$
in a way that increases her expected payoff.
When the payoff function is the long-term average payoff, 
standard statistical tests can be employed to guarantee that no player~$i$ deviates from $\pi^{\rho,\delta}_i$ in a way that improves her payoff.
In our case, the payoff function is general, hence standard statistical tests will not do.
We therefore use a new statistical test, that is suited for general payoff functions.
Since the payoff function is Borel measurable with finite range (cf.~Assumption~\ref{As1}),
there is a history $h'$ that extends $h$ such that,
in the subgame that starts at $h'$,
with high probability under $(\pi^{\rho,\delta}_1,\pi^{\rho,\delta}_2)$ the payoff is some constant $g \in \dR^2$.
That is, if the history $h'$ occurs, and the players follow $(\pi^{\rho,\delta}_1,\pi^{\rho,\delta}_2)$,
then with high probability their payoff will be $g$.
Since $\pi^{\rho,\delta}_1$ and $\pi^{\rho,\delta}_2$ are subgame-perfect $\delta$-maxmin strategies, $g_i \geq v_i - \delta$
for $i\in\{1,2\}$.
Since the space $\mathscr{R}$ is a Polish space, the measure $\prob_{h',\pi^{\rho,\delta}_1,\pi^{\rho,\delta}_2}$ is regular,
and there is a closed subset $R$ of the set $\{r \in \mathscr{R} \colon r \succ h', f_1(r) = g_1, f_2(r) = g_2\}$
whose measure under $\prob_{h',\pi^{\rho,\delta}_1,\pi^{\rho,\delta}_2}$ is high.
Since $R$ is closed, its complement is open --
it is the union of cylinder sets.
This means that if the run generated by the players is outside $R$,
we know it in finite time.
When this happens,
the players will switch to punishment strategies.
This threat of punishment will ensure that no player has an incentive to deviate.
Since payoffs are shift-invariant,
if the initial state is $s_{h'}$ and each player $i$ follows the strategy $\pi^{\rho,\delta}_i$ induced in the subgame that starts at $h'$ while utilizing the test described above,
the resulting strategy profile is an $\ep$-equilibrium.

We will now formalize the ideas described above.
Fix $\ep>0$. 
Let $\rho > 0$ and $\delta\in(0,\delta^\rho)$ be sufficiently small such that $\nu_s^{\rho,\delta}<1$ and for each player $i\in\{1,2\}$,
\begin{equation}
\label{smallrho11}
(1-\rho-\alpha\kappa^\rho)\cdot v_i(C)+(\rho+\alpha\kappa^\rho)\cdot M\,<\,v_i(C)+\frac{\ep}{2}
\end{equation}
and
\begin{equation}
\label{overdelta}
\max\{ f_i(r) \colon r\in\mathscr{R},\ f_i(r) < v_i(C)\} \,<\, v_i(C) - \delta.
\end{equation}
\Eqref{smallrho11} holds for small $\rho>0$ as $\kappa^\rho<\rho$ by \Eqref{choicekappa}, and \Eqref{overdelta} holds for small $\delta>0$, as by Assumption~\ref{As1},
the range of each payoff function $f_i$ is finite. 

In the proof of this step, the parameters
$s$, $\rho$, and $\delta$ are fixed, and 
hence we will not stress the dependency on these parameters of new quantities and objects that are defined below.

Let $\mu\in(0,1)$ be small enough so that 
\begin{equation}
\label{choice-mu}
\sqrt\mu +\mu\,<\,\frac{\ep}{4M},
\end{equation} and for each player $i\in\{1,2\}$ the following inequality hold:
\begin{equation}
\label{smallrho}
(1-\sqrt\mu-\rho- \alpha \kappa^\rho)\cdot v_i(C)+(\sqrt\mu+\rho+ \alpha\kappa^\rho)\cdot M+\mu\,\leq\,v_i(C)+\frac{\ep}{2}
\end{equation}
and
\begin{equation}
\label{choice-ep}
(1-\mu)\cdot \max\{ f_i(r) \colon r\in\mathscr{R},\ f_i(r) < v_i(C)\} +\mu\cdot M\,<\,v_i(C)-\delta;
\end{equation}
such a $\mu$ exists due to Eqs.~\eqref{smallrho11} and \eqref{overdelta}. 
By Eq.~\eqref{choice-mu} and Assumption \ref{As1}, we have 
$\mu<\sqrt\mu<\frac{\ep}{4}$,
and therefore by Eq.~\black\eqref{choicekappa},
\begin{equation}
\label{bigineq}
-\eta^{\overline{\rho}}+\mu+\alpha\kappa^\rho\cdot 4M\,\leq\,\frac{\ep}{4}.
\end{equation}

\noindent\textbf{Step 3.1}. Selecting a state $s^*\in C$. 

Because $\nu_s^{\rho,\delta}<1$ and the payoff function $f=(f_1,f_2)$ has finite range, there is a payoff pair $g\in\dR^2$ in the range of $f$ such that
 \[\prob_{s,\pi^{\rho,\delta}_1,\pi^{\rho,\delta}_2}\big(\theta_C^{exit} = \infty \hbox{ and } f(r) = g\big) \,>\,0,\]
 where $r$ denotes the random variable for the run, taking value in $\mathscr{R}$. 
 It follows from L\'{e}vy's zero-one law that there exists a history $h^*\in H$ that 
starts at state~$s$ and satisfies the following conditions:
\begin{itemize}
    \item[(a)] $h^*$ remains in $C$;
    \item[(b)] in the subgame starting at $h^*$, with high probability under $(\pi^{\rho,\delta}_1,\pi^{\rho,\delta}_2)$, the run always remains in $C$ and the payoff is equal to $g$:
    \begin{equation}
    \label{equ:36.1}
    \prob_{h^*,\pi^{\rho,\delta}_1,\pi^{\rho,\delta}_2}\big(\theta_C^{exit} = \infty \hbox{ and } f(r) = g\big) \,>\,1-\mu.
    \end{equation}
    
\end{itemize}

Denote by $s^*=s_{h^*}$ the final state of the history $h^*$. 
By condition (a) for $h^*$, we have $s^*\in C$. 
We argue that for each player $i\in\{1,2\}$,
\begin{equation}
\label{indrat}
g_i\,\geq\, v_i(C).
\end{equation}
Indeed, suppose by way of contradiction that $g_i< v_i(C)$. 
Then, by property (b) for $h^*$ and \Eqref{choice-ep}, we have $\E_{h^*,\pi^{\rho,\delta}_1,\pi^{\rho,\delta}_2}[f_i]\leq (1-\mu)\cdot g_i+\mu\cdot M<v_i(C)-\delta$. This however contradicts \Eqref{equ:6.44}.
\medskip

\noindent\textbf{Step 3.2.} Defining tests to detect deviations. 

Let $\widetilde\sigma$ denote the continuation of the strategy pair $(\pi^{\rho,\delta}_1,\pi^{\rho,\delta}_2)$ at history $h^*$, that is, $\widetilde\sigma=(\pi^{\rho,\delta}_{1,h^*},\pi^{\rho,\delta}_{2,h^*})$.
We will show that $\widetilde\sigma$, once supplied with tests and threats of punishment,
is an $\ep$-equilibrium, and thereby complete the proof of Step 3.

There are three ways in which a player, say, player~$j$, can deviate from $\widetilde\sigma$:
\begin{enumerate}
\item[D0)] 
Player~$j$ may play an action that she is not supposed to play, that is, 
at a history $h$ that remains in $C$, she may play an action $a_j$ satisfying $\widetilde\sigma_j(a_j \mid h) = 0$.
Such a deviation is identified immediately, and since $C$ is blocked to player~$j$, 
it is not profitable as soon as it triggers punishment at the maxmin level.
\item[D1)]
Player~$j$ may play in such a way that (a) the play remains in $C$, yet (b) her payoff is higher than $g_j$.
We will see that such a deviation can be detected in finite time,
in a way that a false detection of deviation occurs with small probability.
\item[D2)]
Player~$j$ may play in such a way that she increases the probability to leave $C$.
To deter such a deviation,
player~$i$ will calculate at each stage $t$ a certain quantity,
which roughly measures the probability that the play leaves $C$ if the players play the actions given by the actual history.
Under $\widetilde\sigma$, with high probability this quantity is small throughout the play.
Once this quantity is large, player~$j$ will be punished.
Since at stage $t$ player~$i$ plays actions not in $B_i^{\rho,\delta}(h^t)$,
this quantity cannot increase significantly in a single stage,
which will imply that player~$j$ cannot make the play leave $C$ with high probability without being detected and punished.
\end{enumerate}

While detection of deviations of type~(D0) is standard,
detection of deviations of types~(D1) and~(D2) is new as far as we know.

\medskip
\noindent\textbf{Identifying deviations of type~(D0)}

\medskip

Let $\theta_0$ denote the stopping time that indicates that at the previous stage the history remained in $C$ but an action was played that has probability 0 under $\widetilde\sigma$:
\[\theta_0\,:=\,\min\big\{t\geq 2\colon h^{t-1}\text{ remains in }C\text{\ and\ }\widetilde\sigma_k(a_k^{t-1}\mid h^{t-1})=0\text{\ for some\ }k\in\{1,2\}\big\},\]
with the convention that the minimum of the empty set is infinity. 
Thus, $\theta_0$ identifies the stage right after a deviation of type~(D0).

\medskip
\noindent\textbf{Identifying deviations of type~(D1)}

\medskip

Since $\mathscr{R}$ is a Polish space, the measure $\prob_{s^*,\widetilde\sigma}$ is regular. 
By property (b) for $h^*$,
there exists a closed set
\begin{equation}
\label{defR}
R \,\subseteq\, \{\theta_C^{exit}= \infty \hbox{ and } f(r) = g\} 
\end{equation}
such that 
\begin{equation}
\label{propR}
\prob_{s^*,\widetilde\sigma}(R) \,\geq\, 1-\mu.
\end{equation} 

The complement $\mathscr{R}\setminus R$ of the set $R$ is an open set. Consider all histories $h\in H$ for which the cylinder set $\calC(h)=\{r \in \mathscr{R}\colon h \prec r\}$ is included in $\mathscr{R}\setminus R$, 
i.e., $\calC(h)\subseteq \mathscr{R}\setminus R$, and let $H_{R}$ denote the set of minimal
histories among these for the extends-relation $\prec$. Thus,
\[ \mathscr{R}\setminus R\, =\, \bigcup_{h \in H_{R}} \calC(h).\]
By definition, the set $H_{R}$ does not contain two histories $h$ and $h'$ such that $h'$ strictly extends $h$, i.e., $h\prec h'$. 
Also, there is no history $h\notin H_R$ such that each extension of $h$ with one stage, i.e., each history $h'$ with $h\prec h'$ and $\len(h')=\len(h)+1$, is contained in $H_R$.
Intuitively, a history in $H_{R}$ arises as soon as the play leaves the set $C$, 
or if the history remains in $C$ but it becomes unlikely that the payoff $g$ will be realized. For deviations of type (D1), we are interested in the latter.

Let $\theta_1$ be the stopping time for the first stage in which the history remains in $C$ and it belongs to $H_R$:
\[ \theta_1\,:=\, \min\big\{t \in \dN \colon h^t\text{ remains in }C\text{\ and\ }h^t \in H_{R}\big\}. \]
The stopping time $\theta_1$ will be used to identify deviations of type~(D1).

The following property of $\theta_1$ will be useful: if for a run $r\in\mathscr{R}$ we have $\theta_1(r)=\infty$, then either $r$ leaves the set $C$ or $f(r)=g$. We argue by contradiction.
Suppose that $r$ remains in $C$ and $f(r)\neq g$.
We will show that $\theta_1(r)<\infty$. Since $f(r)\neq g$, we have $r\notin R$ by \Eqref{defR}. This means that there is a history $h\prec r$ such that $h\in H_R$. Since $r$ remains in $C$, so does $h$. But then $\theta_1(r)= \length(h)<\infty$ by the definition of $\theta_1$.  
\medskip

\noindent\textbf{Identifying deviations of type~(D2)}\medskip

For an initial state $s^1\in C$, let $H_{s^1}$ denote the set of histories that start in state $s^1$. For each player $i\in\{1,2\}$ and strategy $\sigma_i$ of player $i$, we will define a function $\zeta_{i,\sigma_i}: H_{s^1} \to [0,1]$,
which will in some sense measure the probability that the play could have left $C$ along the given history.  

Consider a player $i\in\{1,2\}$, a strategy $\sigma_i$ of player $i$, and a history  $h=(s^1,a_1^1,a_2^1,\ldots,\allowbreak s^{t-1},a_1^{t-1},a_2^{t-1},s^{t})$. Let $k_h$ denote the number of stages in which the play along $h$ stays in $C$: if the history $h$ remains in $C$ then $k_h=t$, and otherwise $k_h<t$ is such that $s^1,s^2,\ldots,s^{k_h}\in C$ yet $s^{k_h+1}\notin C$. Define 
\begin{equation}
\label{equ:zeta}
\zeta_{i,\sigma_i}(h) \,:=\, \sum_{\ell=1}^{k_h}p((S \setminus C) \mid s^\ell,\sigma_i(s^1,a_1^1,a_2^1,\ldots,s^{\ell}),a^\ell_j).
\end{equation}
The quantity $\zeta_{i,\sigma_i}(h)$ should be thought of as a fictitious probability; intuitively, it expresses the probability of the counter-factual event that along $h$, although the play did not leave the set $C$ in the first $k_h$ stages, it could have left it, when player $i$ follows $\sigma_i$ and the other player, player $j$, follows her actions in $h$. 
The quantity $\zeta_{i,\sigma_i}(h)$ takes the realized actions of player $j$ into account and not her strategy, so that we can use it to detect deviations of type (D2) by player $j$.

The relation between $\zeta_{i,\sigma_i}(\cdot)$ and the probability to leave the set $C$ is captured by the following property:
for every bounded stopping time $\theta:\mathscr{R}\to\dN$,
and every strategy pair $\sigma \in \Sigma_1 \times \Sigma_2$,  
\begin{equation}
\label{equ:zeta:e} 
\E_{s^1,\sigma}[\zeta_{i,\sigma_i}(h^\theta)] \,=\, \prob_{s^1,\sigma}(\theta_C^{exit} < \theta).
\end{equation}
\Eqref{equ:zeta:e} can be verified by induction on  the stopping time $\theta$. 
Indeed, 
recalling that $s^1 \in C$, 
\Eqref{equ:zeta:e} holds for $\theta\equiv 1$ (i.e., the function being constant 1), as in this case both sides of the equation are equal to 0. 
Suppose that $\theta\geq 2$ and we have already proven \Eqref{equ:zeta:e}  
for all 
strategy pairs $\sigma \in \Sigma_1 \times \Sigma_2$ and all
stopping times $\theta'<\theta$ (i.e., $\theta'(r)\leq\theta(r)$ for all $r\in\mathscr{R}$, and with a strict inequality for at least one $r\in\mathscr{R}$). 
In particular, \Eqref{equ:zeta:e} holds for $\theta$
when the initial history is $(s^1,a,s)$ for some $a \in A$ and $s \in C$. 
Then 
\begin{align*}
\prob_{s^1,\sigma}(\theta_C^{exit} < \theta)\,&=\,p\bigl((S\setminus C)\mid s^1,\sigma(s^1)\bigr)\\
&\qquad+\sum_{a\in A,s\in C}\sigma(a)\cdot p(s\mid s^1,a)\cdot \prob_{s^1,\sigma_{(s^1,a,s)}}(\theta_C^{exit} < \theta)\\
&=\,p\bigl((S\setminus C)\mid s^1,\sigma(s^1)\bigr)\\
&\qquad+\sum_{a\in A,s\in C}\sigma(a)\cdot p(s\mid s^1,a)\cdot \E_{s^1,\sigma_{(s^1,a,s)}}[\zeta_{i,\sigma_{i,(s^1,a,s)}}(h^{\theta})]\\
&=\,\E_{s^1,\sigma}[\zeta_{i,\sigma_i}(h^{\theta})].
\end{align*}
This proves \Eqref{equ:zeta:e}, as desired.


To detect deviations of type (D2), 
we will be interested in the quantity $\zeta_{i,\widetilde \sigma_i}(h)$ when the initial state is $s^*$ and player~$i$ follows the strategy $\widetilde\sigma_i$,
which was defined at the beginning of Step~3.2. 
In fact, a deviation of player~$j$ will be announced as soon as $\zeta_{i,\widetilde\sigma_i}(h)$ becomes high.
To ensure that this test is proper, we will prove the following:
(i) Under $\widetilde\sigma$, the probability that $\zeta_{i,\widetilde\sigma_i}(h)$ is high, is low (cf.~\Eqref{equ:51.1}). 
By Markov inequality this will imply that a false detection of deviation is low. (ii) For any profitable deviation of player~$j$, the increase in $\zeta_{i,\widetilde\sigma_i}$ in any single stage is low (cf.~\Eqref{smalljump}).  This will imply that in any single stage, player~$j$ cannot significantly increase the probability that the play leaves $C$,
and so by punishing player~$j$ 
once $\zeta_{i,\widetilde\sigma_i}(h)$ surpasses some properly chosen threshold, 
we can ensure that her payoff is not much higher than her maxmin value in $C$.

The next inequality shows that if the initial state is $s^*$ and the players follow
the strategy pair $\widetilde\sigma$
defined at the beginning of Step~3.2, 
then with high probability, $\zeta_{i,\widetilde\sigma_i}(\cdot)$ remains small during the game.
For every bounded stopping time $\theta$ we have
\begin{eqnarray}
\nonumber 
\prob_{s^*,\widetilde\sigma} \left( \zeta_{i,\widetilde\sigma_i}(h^\theta) \geq \sqrt\mu\right) 
&\,\leq\,& 
\frac{\E_{s^*,\widetilde\sigma}[\zeta_{i,\widetilde\sigma_i}(h^\theta)
]}{\sqrt\mu} \,=\,\frac{\prob_{s^*,\widetilde\sigma}(\theta_C^{exit} < \theta)}{\sqrt\mu} \\
&\,\leq\,& \frac{\prob_{s^*,\widetilde\sigma}(\theta_C^{exit} < \infty)}{\sqrt\mu} \,\leq\, \sqrt\mu,
\label{equ:51.1}
\end{eqnarray}
where the first inequality is by Markov's inequality, the equality is by \Eqref{equ:zeta:e}, and the last inequality follows as \Eqref{equ:36.1} implies that $\prob_{s^*,\widetilde\sigma}(\theta_C^{exit} < \infty)=\prob_{h^*,\pi_1^{\rho,\delta},\pi_2^{\rho,\delta}}(\theta_C^{exit} < \infty)\leq\mu$.

Let $\theta_2$ be the stopping time that indicates the first stage in which $\zeta_i(h)$ exceeds $\sqrt\mu$:
\[ \theta_2\, :=\, \min\bigl\{ t \in \dN \colon \zeta_{i,\widetilde\sigma_i}(h^t) > \sqrt\mu\bigr\}. \]
The stopping time $\theta_2$ will be used to identify deviations of type~(D2).

Note that if $\theta_2=t$ and $\theta_0>t$, then by \Eqref{equ:201} we have
\begin{equation}
\label{smalljump}
\zeta_{i,\widetilde\sigma_i}(h^t)\,\leq\, \sqrt\mu+\rho+\alpha\kappa^\rho;
\end{equation}
that is, as long as $\theta_0$ does not occur, $\zeta_{i,\widetilde\sigma_i}(\cdot)$ can increase in each stage by at most $\rho+\alpha\kappa^\rho$.\medskip

\noindent\textbf{Step 3.3.} Definition of a strategy pair $\sigma^*=(\sigma^*_1,\sigma^*_2)$. 

Let $\sigma^*=(\sigma_1^*,\sigma_2^*)$ be 
the strategy pair that follows $\widetilde\sigma$ 
until time 
\[ \theta_{pun} := \min\{\theta_0, \theta_1,\theta_2\}. \]
At stage $\theta_{pun}$,
each player $i\in\{1,2\}$ switches to a punishment strategy against the other player, 
that is, to a strategy that guarantees that player $j$'s payoff is at most $v_j(s^{\theta_{pun}})+\mu$
(recall that  $s^{\theta_{pun}}$ is the state at stage $\theta_{pun}$).

We note that by \Eqref{propR} and \Eqref{equ:51.1},
$\prob_{s^*,\sigma^*}(\theta_0 = \infty) = 1$,
$\prob_{s^*,\sigma^*}(\theta_1 = \infty) \geq 1-\mu$,
$\prob_{s^*,\sigma^*}(\theta_2 = \infty) \geq 1-\sqrt\mu$.
It follows that
under $\sigma^*$, each player $i$'s expected payoff satisfies 
\begin{align}
\nonumber
\E_{s^*,\sigma^*}[f_i]\,&\geq\,\prob_{s^*,\sigma^*}(\theta_{pun} = \infty)\cdot g_i -(1-\prob_{s^*,\sigma^*}(\theta_{pun}= \infty))\cdot M\\
\nonumber
&\geq\,(1-\sqrt\mu-\mu)\cdot g_i-(\sqrt\mu+\mu)\cdot M \\
&=\, g_i-(\sqrt\mu+\mu)\cdot(g_i+M)\,\geq\, g_i - \frac{\ep}{2},\label{payoff-atleast2}
\end{align}
where the last inequality holds by \Eqref{choice-mu}.


\medskip

\noindent\textbf{Step 3.4.} $\sigma^*$ is an $\ep$-equilibrium for the initial state $s^*$.

We will show that neither player can profit more than $\ep$ by deviating individually from $\sigma^*$.
Fix a player $i\in\{1,2\}$ and a strategy $\sigma_i$ of player~$i$.
Define the following four exhaustive and pairwise disjoint events,
which indicate which test has failed:
\begin{eqnarray*}
A_0 &:=& \{ \theta_{pun} = \theta_0 < \infty\},\\
A_1 &:=& \{ \theta_{pun} = \theta_1 < \theta_0\},\\
A_2 &:=& \{ \theta_{pun} = \theta_2 < \min\{\theta_0,\theta_1\}\},\\
A_3 &:=& (A_0 \cup A_1 \cup A_2)^c\,=\,\{\theta_{pun}=\infty\}.
\end{eqnarray*}
We examine player $i$'s expected payoff under $(\sigma_i,\sigma_j^*)$ conditioned on each of these events separately, and prove that for each $k=0,1,2$, if $\prob_{s^*,\sigma_i,\sigma_j^*}(A_k) > 0$ then
\begin{equation}
\label{ineq-suff}
\E_{s^*,\sigma_i,\sigma_j^*}[f_i\mid A_k]\,\leq\,v_i(C)+\frac{\ep}{2}
\end{equation}
and if $\prob_{s^*,\sigma_i,\sigma_j^*}(A_3) > 0$ then
\begin{equation}
\label{ineq-suff1}
\P_{s^*,\sigma_i,\sigma_j^*}(A_3)\cdot \E_{s^*,\sigma_i,\sigma_j^*}[f_i\mid A_3]\,\leq\,\P_{s^*,\sigma_i,\sigma_j^*}(A_3)\cdot g_i+\frac{\ep}{2}.
\end{equation}

Since
\[\E_{s^*,\sigma_i,\sigma_j^*}[f_i]\,=\,\sum_{\ell=0}^3\P_{s^*,\sigma_i,\sigma_j^*}(A_\ell)\cdot \E_{s^*,\sigma_i,\sigma_j^*}[f_i\mid A_\ell],\]
\Eqref{indrat},
\Eqref{ineq-suff}, and \Eqref{ineq-suff1} will imply that $\E_{s^*,\sigma_i,\sigma_j^*}[f_i]\leq g_i+\frac{\ep}{2}$.
Hence, by  \Eqref{payoff-atleast2}, we will obtain  $\E_{s^*,\sigma^*}[f_i]\geq \E_{s^*,\sigma_i,\sigma_j^*}[f_i]-\ep$, which will complete the proof of Step 3.4, and thereby the proof of Step 3 as well.\smallskip

\noindent\textbf{Conditioning on $A_0$.} 
To show that \Eqref{ineq-suff} holds for $A_0$, 
suppose that player $i$ deviates at stage $\theta_{pun}-1=\theta_0-1$ from $\sigma_i^*$ by choosing an action $a_i\in A_i$ with $\sigma_i(a_i\mid h)=0$, where for simplicity 
we denote by $h = h^{\theta_{pun}-1}$ 
the history at stage $\theta_{pun}-1$. Since $\sigma_j^*$ switches from $\sigma_j$ to a punishment strategy at stage $\theta_{pun}$, 
to prove \Eqref{ineq-suff} for $A_0$ we need to verify that
\begin{equation}
\label{ineq-A0}
\E[v_i+\mu \mid s_h,a_i,\sigma_j^*(h)] \,\leq\,
v_i(C)+\frac{\ep}{2}.
\end{equation}
By the definitions of $\sigma^*_j$ and $\widetilde\sigma_j$ in Steps 3.3 and 3.2, respectively, we have 
\begin{equation}
\label{eq-sigmas}    
\sigma^*_j(h)\,=\,\widetilde\sigma_j(h)\,=\,\pi_j^{\rho,\delta}(h^*h);
\end{equation} 
here as usual, $h^*h$ denotes the concatenation of $h^*$ with $h$. The histories $h^*h$ and $h$ have the same final state, i.e., $s_{h^*h}=s_h$, and since $\theta_{pun}=\theta_0$, the history $h$ remains in $C$, and in particular, $s_h\in C$. We show that \Eqref{ineq-A0} holds by distinguishing between two cases, depending on the action $a_i$.

Suppose first that $a_i \in B_i^{\rho,\delta}(h^*h)$. By \Eqref{equ:5.1}, there is $y_j(s_h) \in Y_j(s_h)$ such that $d\left(\sigma^\delta_j(h^*h),y_j(s_h)\right) < \kappa^\rho$. 
Similarly to \Eqref{equ:15.1}, 
\begin{equation}
\label{closetorho}
 p\bigl( (S \setminus C) \mid s_h,a_i,y_j(s_h)) \,\geq\, \rho-\alpha\kappa^\rho\,>\,0.
\end{equation}
This implies that $(s_h,a_i,y_j(s_h))\in \mathscr{E}_i^{\overline{\rho}}(C,Y)$ for $\overline{\rho}=\rho-\alpha\kappa^\rho$.
We then have
\begin{eqnarray}
\E[v_i+\mu \mid s_h,a_i,\sigma_j^*(h)] &=& 
\E[v_i+\mu \mid s_h,a_i,\pi^{\rho,\delta}_j(h^*h)]\label{long11}\\
\label{long4}&\leq&
\E[v_i+\mu \mid s_h,a_i,\sigma^{\delta}_j(h^*h)] + \alpha\kappa^\rho\cdot 2 M\\ 
\label{equ:long-e1}
&\leq&
\E[v_i+\mu \mid s_h,a_i,y_j(s_h)] + \alpha\kappa^\rho\cdot 4M\\
\nonumber
&=&
\E[v_i \mid s_h,a_i,y_j(s_h)] +\mu+ \alpha\kappa^\rho\cdot 4M \\
\label{equ:long:31}
&\leq& v_i(s_h) - \eta^{\overline{\rho}} +\mu+ \alpha\kappa^\rho\cdot 4M\\
\label{equ:long:311}
&=& v_i(C) - \eta^{\overline{\rho}} +\mu+ \alpha\kappa^\rho\cdot 4M\\
\label{long-last}
&\leq& v_i(C)+\frac{\ep}{2},
\end{eqnarray}
where
\Eqref{long11} holds by \Eqref{eq-sigmas}; Eq.~\eqref{long4} holds by Eq.~\eqref{equ:6.3};
Eq.~\eqref{equ:long-e1} holds by the choice of $y_j(s_h)$;
Eq.~\eqref{equ:long:31} holds since $(s_h,a_i,y_j(s_h))\in \mathscr{E}_i^{\overline{\rho}}(C,Y)$ and
by \Eqref{choice-eta}; Eq.~\eqref{equ:long:311} holds as $v_i$ is a constant over $C$ by assumption; and \Eqref{long-last} holds by \Eqref{bigineq}. 
Thus, \Eqref{ineq-A0} is proven in this case.

Suppose now that $a_i \in A_i\setminus B_i^{\rho,\delta}(h^*h)$. 
By Eqs.~\eqref{equ:201} and \eqref{eq-sigmas}, under $(a_i,\sigma_j^*(h))$ the next state is in $C$ with probability at least $1-\rho-\alpha\kappa^\rho$. Hence, by applying \Eqref{smallrho} we have
\begin{align*}
\E[v_i+\mu \mid s_h,a_i,\sigma_j^*(h)] \,&\leq\,
(1-\rho-\alpha\kappa^\rho)\cdot v_i(C)+(\rho+\alpha\kappa^\rho) \cdot M+\mu\\
&\leq\,v_i(C)+\frac{\ep}{2}.
\end{align*}
Thus, \Eqref{ineq-A0} is proven in this case as well.

\smallskip

\noindent\textbf{Conditioning on $A_1$.} 
We show that \Eqref{ineq-suff} holds for $A_1$.
We have  
\[
\E_{s^*,\sigma_i,\sigma_j^*}[f_i\mid A_1]\,\leq\,v_i(C)+\mu\,\leq\,v_i(C)+\frac{\ep}{2}.
\]
Indeed, the first inequality holds because when $A_1$ occurs, the history 
$h^{\theta_{pun}}$ at stage $\theta_{pun}=\theta_1$ is in $C$, 
and 
at $h^{\theta_{pun}}$ the strategy
$\sigma_j^*$ switches to a punishment strategy. The second inequality holds since $\mu\leq \ep/2$; cf.~the beginning of Step 3.\medskip 

\noindent\textbf{Conditioning on $A_2$.} 
We show that \Eqref{ineq-suff} holds for $A_2$.
We have 
\begin{align*}
\E_{s^*,\sigma_i,\sigma_j^*}[f_i\mid A_2] \,&\leq\,(1-\sqrt\mu-\rho-\alpha\kappa^\rho)\cdot (v_i(C)+\mu)+(\sqrt\mu+\rho+\alpha\kappa^\rho)\cdot M\\
&\leq\,v_i(C)+\frac{\ep}{2}.
\end{align*}
Indeed, the first inequality holds, because when $A_2$ occurs, by \Eqref{smalljump}, the state at stage $\theta_{pun}=\theta_2$ is in $C$ with probability at least $1-\sqrt\mu-\rho-\alpha\kappa^\rho$. The second inequality holds by \Eqref{smallrho}. \medskip

\noindent\textbf{Conditioning on $A_3$.} 
We show that \Eqref{ineq-suff1} holds.
On the event $A_3$ we have $\theta_1=\infty$. Hence, as mentioned above, either the play leaves the set $C$ or the payoff is equal to $g$. Partition $A_3$ into two sets: \[A_3^1\,:=\,\{r\in A_3\colon r\text{ leaves }C\}\quad\text{and}\quad A_3^2\,:=\,\{r\in A_3\colon r\text{ remains in }C\text{ and }f(r)= g\}.\]
On $A_3$ we have $\theta_2=\infty$. Hence, \Eqref{equ:zeta:e} and \Eqref{choice-mu} imply that 
\[\P_{s^*,\sigma_i,\sigma_j^*}(A_3^1)\,\leq\,\sqrt\mu\,\leq\,\frac{\ep}{4M}.\]
Let $\mathbf{1}_W:\mathscr{R}\to\{0,1\}$ denote the characteristic function of a set $W\subseteq\mathscr{R}$. We have 
\begin{align*}\P_{s^*,\sigma_i,\sigma_j^*}(A_3)\cdot \E_{s^*,\sigma_i,\sigma_j^*}[f_i\mid A_3]\,&=\,\E_{s^*,\sigma_i,\sigma_j^*}[f_i\cdot \mathbf{1}_{A_3}]\\
&=\,\E_{s^*,\sigma_i,\sigma_j^*}[f_i\cdot \mathbf{1}_{A^1_3}]+\E_{s^*,\sigma_i,\sigma_j^*}[f_i\cdot \mathbf{1}_{A^2_3}]\\
&\leq\,\P_{s^*,\sigma_i,\sigma_j^*}(A^1_3)\cdot M+\P_{s^*,\sigma_i,\sigma_j^*}(A^2_3)\cdot g_i\\
&=\,\P_{s^*,\sigma_i,\sigma_j^*}(A^1_3)\cdot M+\Big(\P_{s^*,\sigma_i,\sigma_j^*}(A_3)-\P_{s^*,\sigma_i,\sigma_j^*}(A^1_3)\Big)\cdot g_i\\
&\leq\,\P_{s^*,\sigma_i,\sigma_j^*}(A_3)\cdot g_i+\P_{s^*,\sigma_i,\sigma_j^*}(A^1_3)\cdot 2M\\
&\leq\,\P_{s^*,\sigma_i,\sigma_j^*}(A_3)\cdot g_i+\frac{\ep}{2},
\end{align*}
which proves \Eqref{ineq-suff1}.
\bigskip

%
To complete the proof of Lemma~\ref{lemma:alternatives} we will prove the following.

\bigskip
\noindent\textbf{Step 4:} The set $C$ is $Y$-jointly controlled.

The idea of the proof of Step 4 is as follows. For every $\rho>0$ and $\delta\in(0,\delta^\rho)$, we will define the mapping $\lambda^{\rho,\delta}$ on the
triplets $(s,a_1,a_2)\in C\times A$, where $\lambda^{\rho,\delta}(s,a_1,a_2)$ is the probability under $(s^1,\pi_1^{\rho,\delta},\pi_2^{\rho,\delta})$ that the play will eventually leave the set $C$ from state $s$ through the action pair $(a_1,a_2)$,
where $s^1$ is some fixed state in $C$.
By taking an accumulation point of $\lambda^{\rho,\delta}$, as $\rho$ and $\delta$ go to 0, we obtain a probability distribution that, as we show, only places positive probability on triplets $(s,a_1,a_2)$ that are $Y$-joint exits from $C$. 
A variation of 
this probability distribution will be the one that we need to show that $C$ is $Y$-jointly controlled. 

Define the random variable $t_*$ by
\[ t_* = \theta^{exit}_C - 1.\]
If the play leaves $C$ (that is, $\theta^{exit}_C<\infty$) then $t_*$ is the stage when this happens; if the play does not leave $C$, then $t_*=\infty$. 
Note that $t_*$ is not measurable w.r.t.~the information that the players have at the end of stage $t_*$, because the transition may be random.

Fix an initial state $s^1\in C$. 
Let $L_C$ be the set of triplets $(s,a_1,a_2)\in C\times A$ such that
$p\bigl((S \setminus C) \mid s,a_1,a_2\bigr) > 0$.

\medskip
\noindent\textbf{Step 4.1:} Defining probability distributions $\lambda$ and $\widetilde\lambda$ over $L_C$.

For every $\rho>0$, $\delta\in(0,\delta^\rho)$, and $(s,a_1,a_2)\in L_C$,
define
\[ \lambda^{\rho,\delta}(s,a_1,a_2)\, :=\, \prob_{s^1,\pi_1^{\rho,\delta},\pi_2^{\rho,\delta}}( s^{t_*} = s, a_1^{t_*} = a_1, a_2^{t^*} = a_2). \]
By the definition of $\nu_{s^1}^{\rho,\delta}$, cf.~\Eqref{def-nu}, we have
\begin{equation}
\label{lambda-nu}
\lambda^{\rho,\delta}(L_C)\,=\,\prob_{s^1,\pi_i^{\rho,\delta},\pi_j^{\rho,\delta}}(\theta_C^{exit}<\infty)\,=\,\nu_{s^1}^{\rho,\delta}.
\end{equation}
For every $\rho>0$, let $\lambda^\rho$ be an accumulation point of $(\lambda^{\rho,\delta})_{\delta \in(0,\delta^\rho)}$ as $\delta$ goes to 0. 
Let $\lambda$ be an accumulation point of $(\lambda^\rho)_{\rho>0}$ as $\rho$ goes to 0.

For every $\rho>0$, $\delta\in(0,\delta^\rho)$, and $(s,a_1,a_2)\in L_C$, define
\[ \widetilde\lambda^{\rho,\delta}(s,a_1,a_2) \,:=\, 
\frac{ \frac{\lambda^{\rho,\delta}(s,a_1,a_2)}{p\bigl((S \setminus C) \mid s,a_1,a_2\bigr)}}
{\sum\left\{
\frac{\lambda^{\rho,\delta}(s',a'_1,a'_2)}{p\bigl((S \setminus C) \mid s',a'_1,a'_2\bigr)} : (s',a'_1,a'_2) \in L_C\right\}}.
\]
Note that $p\bigl((S \setminus C) \mid s,a_1,a_2\bigr) > 0$ for every $(s,a_1,a_2) \in L_C$,
hence $\widetilde\lambda^{\rho,\delta}$ is well defined,
and by definition it is a probability distribution.

The quantity $\lambda^{\rho,\delta}(s,a_1,a_2)$ is the probability that the triplet $(s,a_1,a_2)$ is played just before the play leaves $C$;
the quantity $\widetilde\lambda^{\rho,\delta}(s,a_1,a_2)$ is the frequency of times needed to play the triplet $(s,a_1,a_2)$ to ensure that the probability that the triplet $(s,a_1,a_2)$ is played just before the play leaves $C$ is exactly $\lambda^{\rho,\delta}(s,a_1,a_2)$.
\medskip 

\noindent\textbf{Step 4.2:} $\lambda$ is a probability distribution supported by the set $\mathscr{E}_{12}(C,Y)$ of $Y$-joint exits from $C$. 

We have 
\begin{equation}
\label{limit-lambda}
\lambda(L_C)\,=\,\lim_{\rho\to 0}\big(\lim_{\delta\to 0} \lambda^{\rho,\delta}(L_C)\big)\,=\,\lim_{\rho\to 0}\big(\lim_{\delta\to 0} \nu_{s^1}^{\rho,\delta}\big)\,=\,\nu_{s^1}\,=\,1,
\end{equation}
where the second equality follows by \Eqref{lambda-nu}, and the last equality is due to Step 3. Consequently, $\lambda$ is a probability distribution over the set $L_C$. It remains to prove that if $(s,a_1,a_2)\in L_C$ satisfies $\lambda(s,a_1,a_2)>0$, then $(s,a_1,a_2)\in \mathscr{E}_{12}(C,Y)$.

Let $(s,a_1,a_2)\in L_C$ such that $\lambda(s,a_1,a_2)>0$. 
For convenience, assume
that $\lambda^\rho(s,a_1,a_2)>0$ for each $\rho>0$ and $\lambda^{\rho,\delta}(s,a_1,a_2)>0$ for each $\rho>0$ and $\delta \in(0,\delta^\rho)$.

For each $\rho>0$ and $\delta \in(0,\delta^\rho)$, as $\lambda^{\rho,\delta}(s,a_1,a_2)>0$, there is a history $h^{\rho,\delta} \in H$ that remains in $C$ such that (a) $s_{h^{\rho,\delta}} = s$,
(b) $\pi^{\rho,\delta}_1(a_1 \mid h^{\rho,\delta}) > 0$,
and (c) $\pi^{\rho,\delta}_2(a_2 \mid h^{\rho,\delta}) > 0$. By \Eqref{equ:6.2}, 
$a_i \not\in B^{\rho,\delta}_i(h^{\rho,\delta})$ for each $i\in\{1,2\}$, and hence by \Eqref{equ:B} we obtain
\begin{equation}
\label{equ:901}
p\bigl( (S \setminus C) \mid s,a_i,\sigma^\delta_j(h^{\rho,\delta})\bigr) \,<\,\rho.
\end{equation}

For each player $i\in\{1,2\}$ and each $\rho>0$, let $y_i^\rho(s)$ be an accumulation point of $(\sigma^\delta_j(h^{\rho,\delta}))_{\delta \in(0,\delta^\rho)}$ as $\delta$ goes to 0. 
Since $Y_i(s)$ is compact, $y_i^\rho(s)\in Y_i(s)$, and by \Eqref{equ:901},
\[p\bigl( (S \setminus C) \mid s,a_i,y_j^\rho(s)\bigr) \,\leq\, \rho.\] 
Let $y_i(s)$ be an accumulation point of $(y_i^\rho(s))_{\rho>0}$ as $\rho$ goes to 0. 
We have $y_i(s)\in Y_i(s)$ because $Y_i(s)$ is compact, and moreover for each player $i\in\{1,2\}$,
\begin{equation}
\label{equ:902}
p\bigl( (S \setminus C) \mid s,a_i,y_j(s)\bigr)\,=\,0.
\end{equation}
Since \Eqref{equ:902} holds for $i\in\{1,2\}$, it follows that
$(s,a_1,a_2)\in \mathscr{E}_{12}(C,Y)$, as desired.\medskip

\noindent\textbf{Step 4.3:} The set $C$ is $Y$-jointly controlled. 

By Step 4.2, $\lambda$ 
can be viewed as
a probability distribution over $\mathscr{E}_{12}(C,Y)$.
We will show that $C$ is $Y$-jointly controlled with the probability distribution ${\widetilde\lambda} \in \Delta(\mathscr{E}_{12}(C,Y))$,
defined for each $(s,a_1,a_2) \in \mathscr{E}_{12}(C,Y)$ by  
\[ \widetilde\lambda(s,a_1,a_2) := 
\frac{ \frac{\lambda(s,a_1,a_2)}{p\bigl((S \setminus C) \mid s,a_1,a_2\bigr)}}
{\sum\left\{
\frac{\lambda(s',a'_1,a'_2)}{p\bigl((S \setminus C) \mid s',a'_1,a'_2\bigr)} : (s',a'_1,a'_2) \in \mathscr{E}_{12}(C,Y)\right\}}.
\]
Note that $\widetilde\lambda(s,a_1,a_2)=\lim_{\rho\to 0}\big(\lim_{\delta\to 0} \widetilde\lambda^{\rho,\delta}(s,a_1,a_2)\big)$ for each $(s,a_1,a_2) \in \mathscr{E}_{12}(C,Y)$.  
To show that $C$ is $Y$-jointly controlled with $\widetilde\lambda$, we need to verify that for each player $i\in\{1,2\}$ we have
\begin{equation}\label{toshowjointly}
H_i(C,Y)\,\leq\,\sum_{(s,a_1,a_2)\in\mathscr{E}_{12}(C,Y)}{\widetilde\lambda}(s,a_1,a_2)\cdot \E[{v}_i \mid s,a_1,a_2].
\end{equation}
Fix $i\in\{1,2\}$.
For each $n\in\dN$, let $\xi_n$ be the random variable defined by $\xi_n:=\min\{\theta_C^{exit},n\}$. 
By Property 2 of Step 2 (with $h$ being the history at stage 1 and thus $s_h=s^1$), for each $\rho>0$ and each $\delta\in(0,\delta^\rho)$,
\begin{align}
\nonumber
v_i(C)-\delta\,&\leq\,\E_{s^1,\pi_i^{\rho,\delta},\pi_j^{\rho,\delta}}[v_i(s^{\xi_n})]\\
\label{equ:4:1}
&=\,\prob_{s^1,\pi_i^{\rho,\delta},\pi_j^{\rho,\delta}}(\theta_C^{exit}\leq n)\cdot \E_{s^1,\pi_i^{\rho,\delta},\pi_j^{\rho,\delta}}[v_i(s^{\theta_C^{exit}})\mid \theta_C^{exit}\leq n]\\
&\quad+\prob_{s^1,\pi_i^{\rho,\delta},\pi_j^{\rho,\delta}}(\theta_C^{exit}> n)\cdot v _i(C).
\nonumber
\end{align}
For each $(s,a_1,a_2)\in L_C$, let $\ell(s,a_1,a_2)$ denote the expected maxmin value at the first state reached once the play leaves $C$, 
conditionally on leaving $C$ through $(s,a_1,a_2)$:
\[\ell(s,a_1,a_2)\,=\,\sum_{s'\in S\setminus C}v_i(s')\cdot \frac{p(s'\mid s,a_1,a_2)}{p((S\setminus C)\mid s,a_1,a_2)}.\]
Then,
\begin{align*}\E_{s^1,\pi_i^{\rho,\delta},\pi_j^{\rho,\delta}}[v_i(s^{\theta_C^{exit}})\mid \theta_C^{exit}<\infty]\,&=\,\frac{\sum_{(s,a_1,a_2)\in L_C} \lambda^{\rho,\delta}(s,a_1,a_2)\cdot \ell(s,a_1,a_2)}{\P_{s^1,\pi_i^{\rho,\delta},\pi_j^{\rho,\delta}}(\theta^{exit}_C<\infty)}.
\end{align*}
 By taking the limit as $n$ goes to infinity in \Eqref{equ:4:1}, and by \Eqref{lambda-nu}, 
\begin{align}
\label{longineq}
v_i(C)-\delta\,&\leq\,\prob_{s^1,\pi_i^{\rho,\delta},\pi_j^{\rho,\delta}}(\theta_C^{exit}<\infty)\cdot \E_{s^1,\pi_i^{\rho,\delta},\pi_j^{\rho,\delta}}[v_i(s^{\theta_C^{exit}})\mid \theta_C^{exit}<\infty]\\
&\quad\ \ +\prob_{s^1,\pi_i^{\rho,\delta},\pi_j^{\rho,\delta}}(\theta_C^{exit}=\infty)\cdot v _i(C)\nonumber\\
&=\,\sum_{(s,a_1,a_2)\in L_C} \lambda^{\rho,\delta}(s,a_1,a_2)\cdot \ell(s,a_1,a_2)\nonumber\\
&\quad\ \ +\prob_{s^1,\pi_i^{\rho,\delta},\pi_j^{\rho,\delta}}(\theta_C^{exit}=\infty)\cdot v_i(C)\nonumber\\
&=\,\sum_{(s,a_1,a_2)\in L_C} \frac{\lambda^{\rho,\delta}(s,a_1,a_2)}{p((S\setminus C)\mid s,a_1,a_2)}\cdot\sum_{s'\in S\setminus C}v_i(s')\cdot p(s'\mid s,a_1,a_2)\nonumber\\
&\quad\ \ +\prob_{s^1,\pi_i^{\rho,\delta},\pi_j^{\rho,\delta}}(\theta_C^{exit}=\infty)\cdot v _i(C)\nonumber\\
& =\,\sum_{(s,a_1,a_2)\in L_C} \frac{\lambda^{\rho,\delta}(s,a_1,a_2)}{p((S\setminus C)\mid s,a_1,a_2)}\cdot\Big(\E[v_i \mid s,a_1,a_2]-p(C\mid s,a_1,a_2)\cdot v_i(C)\Big)\nonumber\\ 
& \quad\ \ +\left(1-\sum_{(s,a_1,a_2)\in L_C}\lambda^{\rho,\delta}(s,a_1,a_2)\right)\cdot v _i(C)\nonumber\\
&=\,
\sum_{(s',a'_1,a'_2) \in L_C} \frac{\lambda^{\rho,\delta}(s,a'_1,a'_2)}{p((S\setminus C)\mid s,a'_1,a'_2)} \cdot
\sum_{(s,a_1,a_2) \in L_C} \widetilde\lambda^{\rho,\delta}(s,a_1,a_2)\cdot 
\E[v_i \mid s,a_1,a_2]\nonumber\\
&+\left(1-\sum_{(s',a'_1,a'_2) \in L_C} \frac{\lambda^{\rho,\delta}(s,a'_1,a'_2)}{p((S\setminus C)\mid s,a'_1,a'_2)} \cdot
\sum_{(s,a_1,a_2) \in L_C} \widetilde\lambda^{\rho,\delta}(s,a_1,a_2)\right)\cdot v _i(C).
\nonumber
\end{align}
By taking the limit in \Eqref{longineq} as $\rho$ and $\delta$ go to 0, and denoting
\[Z:=\sum_{(s',a'_1,a'_2) \in \mathscr{E}_{12}(C,Y)} \frac{\lambda(s,a'_1,a'_2)}{p((S\setminus C)\mid s,a'_1,a'_2)},\]
we obtain
\[0\,\leq\,Z\cdot\left(
\sum_{(s,a_1,a_2) \in \mathscr{E}_{12}(C,Y)} \widetilde\lambda(s,a_1,a_2)\cdot 
\E[v_i \mid s,a_1,a_2]-\sum_{(s,a_1,a_2) \in \mathscr{E}_{12}(C,Y)} \widetilde\lambda(s,a_1,a_2)\cdot v_i(C)\right).\]
Since $Z>0$ and $\widetilde\lambda$ is a probability distribution over $\mathscr{E}_{12}(C,Y)$, 
\[v_i(C)\,\leq\,\sum_{(s,a_1,a_2)\in\mathscr{E}_{12}(C,Y)}{\widetilde\lambda}(s,a_1,a_2)\cdot \E[{v}_i \mid s,a_1,a_2].\]
By Lemma \ref{remark:blocked}, $H_i(C,Y)=v_i(C)$, and \Eqref{toshowjointly} follows.
\end{proof}

\subsection{Notations and definitions for different transition functions}

In the following sections, we will construct auxiliary transition functions for the stochastic game that differ from $p$ in some states. In this section, as a preparation, we fix the notations and definitions for such new transition functions. Let $p':S\times A\to\Delta(S)$ be an arbitrary transition function.

Instead of $\prob$ and $\E$, we will use $\prob_{p'}$ to refer to a probability measure under $p'$ on the set $\mathscr{R}$ of runs, and use $\E_{p'}$ to denote the expectation operator under $p'$; cf.~Eq.~\eqref{oneshot-g}.

\begin{definition}[communicating set under $p'$]
A nonempty set $C\subseteq S_0$ is \emph{$X$-com\-mu\-ni\-cating under $p'$}, where $X\in\mathcal{X}$, if the conditions of Definition \ref{definition:communicating} hold when $p$ is replaced with $p'$ in Condition 1 and $\prob_{s,y}$ is replaced with $\prob_{p',s,y}$ in Condition 2. 
\end{definition} 

\begin{definition}[unilateral exit under $p'$]
Let $C \subseteq S_0$ be a nonempty set, $X \in \mathcal{X}$, and $i \in \{1,2\}$.
A triplet $(s,a_i,x_j(s))\in C\times A_i\times X_j(s)$
is an \emph{$X$-unilateral exit of player~$i$ from $C$ under $p'$} if the conditions of Definition \ref{def:unilateral:exit} hold under $p'$:
there exists $x_i(s) \in X_i(s)$ such that
$p'(C \mid s,x_i(s),x_j(s)) = 1$ and
$p'\bigl(C \mid s,a_i,x_j(s)\bigr) <1$.
\end{definition}

Let $\mathscr{E}_{i,p'}(C,X)$ denote the set of $X$-unilateral exits of player~$i$ from $C$ under $p'$.

\begin{definition}[blocked set under $p'$]
\label{newdedblocked}
Let $C \subseteq S_0$ be a nonempty set, $X\in \mathcal{X}$,
and $i \in \{1,2\}$. The set $C$ is \emph{$X$-blocked to player~$i$ under $p'$} if for each exit $(s,a_i,x_j(s))\in \mathscr{E}_{i,p'}(C,X)$,
\begin{equation}
\label{eqblockednewp}
\E_{p'}[v_i \mid s,a_i,x_j(s)] \,<\, \max_{s'\in C}v_i(s').
\end{equation}
\end{definition} 

Note that in \Eqref{eqblockednewp} we use the maxmin value $v_i$ for player $i$ with respect to the original transition function $p$. 

\subsection{Constructing a First Family of Sets}\label{sec-firstfam}

In this section, we define a family consisting of certain sets of states, which is an analog of the family defined on page 81 in Vieille \cite{vieille2000one}. This family will be used to construct an $\ep$-equilibrium in Section \ref{sec-finalstep}. 

\begin{definition}[the family of sets $\mathscr{F}_1$]
\label{def:calS}
Let $\mathscr{F}_1$ be the family of all nonempty sets $C\subseteq S_0$ that satisfy the following properties:
\begin{enumerate}
\item[(P1)] $C$ is $Y$-communicating.
\item[(P2)] For each $k\in\{1,2\}$, player $k$'s maxmin value is constant over $C$, denoted by~$v_k(C)$.
\item[(P3)] $C$ is $Y$-blocked to player~1.
\item[(P4)] No strict superset of $C$ satisfies (P1)--(P3).
\end{enumerate}
\end{definition}

The sets in $\mathscr{F}_1$ are pairwise disjoint. 
This claim follows from the fact that if $C$ and $C'$ both satisfy (P1)--(P3) and $C\cap C'\neq\emptyset$, then $C\cup C'$ also satisfies (P1)--(P3).

The following lemma is the analog of Lemma 38 in Vieille \cite{vieille2000one}.

\begin{lemma}
\label{lemma38}
Each $C \in \mathscr{F}_1$ is $Y$-controlled by player~2 or $Y$-jointly controlled.
\end{lemma}

 Since each set in $\mathscr{F}_1$ is $Y$-controlled by player~2 or $Y$-jointly controlled,
we will show that whenever the play reaches one of these sets, 
the players can ensure that the play leaves the set in a properly chosen way. This will allow us to turn all states that lie in some set in $\mathscr{F}_1$ into dummy states.
This will produce a simplified auxiliary stochastic game with the property that
all $\ep$-equilibria of this auxiliary game can be transformed into $\ep$-equilibria of the original game. 
\bigskip

\begin{proof}[Proof of Lemma~\ref{lemma38}]
If $C$ is $Y$-blocked to player~2, then Lemma~\ref{lemma:alternatives} implies that $C$ is $Y$-jointly controlled.

Assume then that $C$ is not $Y$-blocked to player~2. 
It follows that for some exit $(s,y_1(s),a_2)\in\mathscr{E}_2(C,Y)$ we have $\E[v_2\mid s,y_1(s),a_2]\,\geq\,v_2(C)$. Hence, by Lemma \ref{maxexit}, there is an exit $(s^*,y^*_1(s),a^*_2)\in\mathscr{E}_2(C,Y)$ that maximizes $\E[v_2\mid\cdot]$ among all exits in $\mathscr{E}_2(C,Y)$, and thus also
\begin{equation}
\label{reasonexit}
\E[v_2\mid s^*,y^*_1(s^*),a^*_2]\,\geq\,v_2(C).
\end{equation}

We show that $C$ is $Y$-controlled by player 2, by verifying that $\E[v_k\mid s^*,y^*_1(s),a^*_2]\geq H_k(C,Y)$ for each player $k\in\{1,2\}$.
For $k=2$, this inequality follows by the choice of $(s^*,y^*_1(s),a^*_2)$ and by \Eqref{reasonexit}.
For $k=1$, this follows by Lemmas \ref{lemma:expected:minmax} and \ref{remark:blocked}.
%
%
\end{proof}

\begin{notation}[the exit distribution $q^{\mathscr{F}_1}_C$]\rm 
Based on Lemma \ref{lemma38}, for each $C \in \mathscr{F}_1$ we choose an exit distribution $q^{\mathscr{F}_1}_C \in \Delta(S)$ as follows:
\begin{itemize}
\item
If $C$ is $Y$-controlled by player~2,
we set
\[ q^{\mathscr{F}_1}_C \,:=\, p(\cdot \mid s,x_1(s),a_2), \]
for some $Y$-unilateral exit $(s,x_1(s),a_2)\in\mathscr{E}_2(C,Y)$ of player~2 from $C$ such that for each player $k\in\{1,2\}$
\begin{equation}
\label{equ:qf1:1} 
\E[v_k\mid s,x_1(s),a_2]\, \geq \,H_k(C,Y). 
\end{equation}
Note that $H_1(C,Y)=v_1(C)$ by Lemma \ref{remark:blocked}.
\item
If $C$ is not $Y$-controlled by player 2, then $C$ is $Y$-jointly controlled, and 
we set
\[ q^{\mathscr{F}_1}_C \,:=\, \sum_{(s,a_1,a_2) \in \mathscr{E}_{12}(C,Y)} \mu(s,a_1,a_2) \cdot p(\cdot \mid s,a_1,a_2), \]
for some probability distribution $\mu$ over $\mathscr{E}_{12}(C,Y)$ such that for each player $k\in\{1,2\}$, 
\begin{equation}
\label{equ:qf1:2} 
\sum_{(s,a_1,a_2) \in \mathscr{E}_{12}(C,Y)} \mu(s,a_1,a_2)\cdot \E[v_k\mid s,a_1,a_2]\, \geq \,H_k(C,Y). 
\end{equation}
\end{itemize}
In both cases, we define for each player $k\in\{1,2\}$,
\[
\E[v_k\mid q^{\mathscr{F}_1}_C]\,:=\,\sum_{s'\in S}q^{\mathscr{F}_1}_C(s')\cdot v_k(s'),\]
which is the expectation of player $k$'s maxmin value in the next state if the trasition is according to the exit distribution $q^{\mathscr{F}_1}_C$. 
\end{notation}

By Eqs.~\eqref{equ:91}, \eqref{equ:qf1:1}, and~\eqref{equ:qf1:2}, 
for each player $k\in\{1,2\}$,
\begin{equation}
\label{exp-q}
\E[v_k\mid q^{\mathscr{F}_1}_C]\,\geq\,H_k(C,Y)\,\geq\,v_k(C).
\end{equation}

\begin{notation}[the sets of states $S_{\mathscr{F}_1}$ and $\widehat S_0$]\rm 
Let $S_{\mathscr{F}_1}$ denote the set of states $s\in S_0$ that belong to some $C\in\mathscr{F}_1$:
\[S_{\mathscr{F}_1}\,:=\bigcup_{C \in \mathscr{F}_1} C.\]
Denote $\widehat S_0 := S_0 \setminus S_{\mathscr{F}_1}$,
which is the set of nonabsorbing states that do not lie in any set in $\mathscr{F}_1$. 
\end{notation}

Thus, the set of states $S$ is partitioned into $\{S^*,\widehat S_0,S_{\mathscr{F}_1}\}$; recall that $S^*$ is the set of absorbing states.

\begin{notation}[uniform distribution $U_C$]\rm 
For every nonempty set $C \subseteq S$ of states, let $U_C\in\Delta(S)$ be the uniform distribution over $C$, i.e., $U_C(s)=1/|C|$ if $s\in C$ and $U_C(s)=0$ otherwise.
\end{notation}

\begin{notation}[the transition function $\widehat p$]\rm 
Let $\widehat p:S\times A\to \Delta(S)$ denote the following transition function: for each state $s\in S$ and action pair $a\in A$,
\[ \widehat p(\cdot \mid s,a) \,:=\, \left\{
\begin{array}{lll}
\frac{1}{2} q^{\mathscr{F}_1}_C + \frac{1}{2}U_C, & \ \ \ \ \ & \text{if }s\in S_{\mathscr{F}_1},\text{ i.e., }s \in C\text{ for some }C\in\mathscr{F}_1\\[0.2cm]
p(\cdot \mid s,a), & & \text{if }s \in S^*\cup \widehat S_0.
\end{array}
\right.
\]
\end{notation}

Thus,
$\widehat p$ coincides with $p$
at states that are not in $S_{\mathscr{F}_1}$,
and $\widehat p$ ensures that the play leaves $C$ according to $q^{\mathscr{F}_1}_C$
at states in $C \in S_{\mathscr{F}_1}$.

Note that $\widehat p$ is well defined, because the sets in $\mathscr{F}_1$ are pairwise disjoint. 
By definition, the transitions under $\widehat p$ in states in $S_{\mathscr{F}_1}$ are independent of the players' actions.
Whenever we make statements that refer to the transition function $\widehat p$, we will explicitly mention that. 

The following lemma is the analog of Lemma 42 in Vieille \cite{vieille2000one}.

\begin{lemma}
\label{lemma:42}
There is no subset $D\subseteq S_0$ that satisfies the following three properties simultaneously:
\begin{enumerate}
\item[(C1)] $\widehat p(D \mid s,y(s)) = 1$, for every $s \in D$ and every $y(s) \in Y(s)$.
\item[(C2)] $D$ is $Y$-blocked to player~1.
\item[(C3)] For each $k\in\{1,2\}$, player $k$'s maxmin value $v_k$ is constant over $D$.
\end{enumerate}
\end{lemma}

\begin{proof}
We argue by contradiction, and assume that $D \subseteq S_0$ satisfies (C1)--(C3).

\smallskip
\noindent\textbf{Step 1:} Constructing a $Y$-communicating set $\overline D$ under $\widehat p$.

By (C1), there exists a nonempty subset $\overline D \subseteq D$ such that
\begin{enumerate}
    \item[(a)] $\overline D$ is $Y$-communicating under $\widehat p$,
    \item[(b)] $\widehat p(\overline D\mid s,y(s))=1$ for every $s \in \overline D$ and every perturbation $y(s)$ of $Y(s)$ such that $\widehat p(D\mid s,y(s))=1$.
\end{enumerate} 
Indeed, for states $s,s'\in D$, write $s\to s'$ if there is a perturbation $y(s)$ of $Y(s)$ such that $\widehat p(D\mid s,y(s))=1$ and $\widehat p(s' \mid s,y(s)) >0$. In the directed graph with $D$ as the set of vertices and arcs given by $\to$, any set $\overline D$ of vertices that forms a maximal strongly connected component satisfies (a) and (b); cf.~also Lemma 41 in Vieille \cite{vieille2000one}.

\smallskip
\noindent\textbf{Step 2:} Two properties of $\overline D$.

We will use two properties of the construction regularly:
\begin{enumerate}
    \item[(i)] $\widehat p$ coincides with $p$ in states belonging to $\overline D\cap \widehat{S}_0$.
    \item[(ii)] For each $C\in\mathscr{F}_1$ we have $C\subseteq \overline D$ or $C\cap \overline D=\emptyset$.
\end{enumerate} 

Property (i) holds by definition of $\widehat p$. 
To see that property (ii) holds, assume that $C\cap \overline D\neq\emptyset$ for some set $C\in\mathscr{F}_1$. 
Fix $s\in C\cap \overline D$. 
At state $s$, the transition function $\widehat p$ places a positive probability to moving to all states in $C$ (through $U_C$). 
By property (b), we have $C\subseteq \overline D$.

We next show that $\overline D$ satisfies (P1)--(P3) and that $\overline D$ is disjoint from each $C\in\mathscr{F}_1$, and then we derive a contradiction 
to the minimality of $\mathscr{F}_1$.

\smallskip
\noindent\textbf{Step 3:} $\overline D$ satisfies (P1)--(P3).

\noindent\textsc{Property (P1) for $\overline D$:} By property (a), we already know that $\overline D$ is $Y$-communicating under $\widehat p$. 
We need to prove that $\overline D$ is also $Y$-communicating under $p$; cf.~also Lemma 40 in Vieille \cite{vieille2000one}.

First we show that $\overline D$ satisfies Condition 1 of Definition \ref{definition:communicating} under $p$:  for every $s\in\overline D$ and every $y(s)\in Y(s)$ we have $p(\overline D\mid s,y(s))=1$. Let $s\in\overline D$ and $y(s)\in Y(s)$. If $s\in \widehat{S}_0$, then by (i) and (b) above, $p(\overline D\mid s,y(s))=\widehat p(\overline D\mid s,y(s))=1$. If $s\in C$ for some $C\in\mathscr{F}_1$, then by (ii) above, $C\subseteq \overline D$, and hence (P1) for the set $C$ implies $p(\overline D\mid s,y(s))=1$.

We next argue that $\overline D$ satisfies Condition 2 of Definition \ref{definition:communicating} under $p$: for every $s,s'\in \overline D$, there exists a perturbation of $Y$ that allows moving from $s$ to $s'$ under $p$ without leaving $\overline D$. 
Recalling (a), this holds due to (i) and (ii) above. 
Indeed, first, in states belonging to $\overline D\cap \widehat{S}_0$, the same transitions are possible under $\widehat p$ and $p$. 
Second, in each set $C\in\mathscr{F}_1$ with $C\subseteq \overline D$, $\widehat p$ is 
an average
of $U_C$ and $q^{\mathscr{F}_1}_C$. 
By property~(P1) for $C$, it is possible also under $p$ to move from any state to any other state within $C$ using mixed actions in $Y$, and under $p$ the set $C$ can be left with the distribution $q^{\mathscr{F}_1}_C$ using a perturbation of $Y$.\smallskip

\noindent\textsc{Property (P2) for $\overline D$:} The property holds since $\overline D \subseteq D$ and by (C3).\smallskip

\noindent\textsc{Property (P3) for $\overline D$:} Take an arbitrary unilateral $Y$-exit $(s,a_1,y_2(s))\in\mathscr{E}_1(\overline D,Y)$ of player~1 from $\overline D$. In particular, $s\in \overline D$. We distinguish between two cases.\smallskip

\noindent\textsc{Case 1:} $s \in C$ for some $C \in \mathscr{F}_1$. 

By (ii) above, $C \subseteq \overline D$.
Hence, $(s,a_1,y_2(s))$ is a $Y$-unilateral exit of player~1 from $C$, i.e., $(s,a_1,y_2(s))\in\mathscr{E}_1(C,Y)$. Thus, we have
\[\E[v_1 \mid s,a_1,y_2(s)] \,<\, \max_{s' \in C} v_1(s')\,=\, \max_{s' \in \overline D} v_1(s'),\]
where the inequality is by Property (P3) for $C$, and the equality is by Property (C3) together with $C\subseteq \overline D\subseteq D$.
\smallskip

\noindent\textsc{Case 2:} $s \in \widehat S_0$. 

In this case, by (i), at state $s$ the transitions by $\widehat p$ and by $p$ coincide. 

First we argue that $p(D\mid s,a_1,y_2(s))<1$. Assume by contradiction the opposite: $p(D\mid s,a_1,y_2(s))=1$. For any $y_1(s)\in Y_1(s)$, the mixed action $y'_1(s)=\frac{1}{2}\cdot a_1+\frac{1}{2}\cdot y_1(s)$ is a perturbation of $Y_1(s)$, and $p(D\mid s,y'_1(s),y_2(s))=1$. Hence, by property (b) above, $p(\overline D\mid s,y'_1(s),y_2(s))=1$, and in particular,  $p(\overline D\mid s,a_1,y_2(s))=1$. This, however, contradicts $(s,a_1,y_2(s))\in\mathscr{E}_1(\overline D,Y)$. 
Thus, $p(D\mid s,a_1,y_2(s))<1$ as claimed.

Therefore, $(s,a_1,y_2(s))$ is a $Y$-unilateral exit of player~1 from $D$, i.e., $(s,a_1,y_2(s))\in\mathscr{E}_1(D,Y)$. Then, 
\[\E[v_1\mid s,a_1,y_2(s)] \,<\,\max_{s' \in D} v_1(s') \,=\, \max_{s' \in \overline D} v_1(s'), \]
where the inequality is by Property (C2), and the equality is by Property (C3) together with $\overline D\subseteq D$. 

These two cases show that property~(P3) holds for $\overline D$, as claimed.

\smallskip
\noindent\textbf{Step 4:} The contradiction.

We argue that $\overline D$ is disjoint from each $C\in \mathscr{F}_1$.
Indeed, 
suppose by way of contradiction that $\overline D\cap C\neq\emptyset$ for some $C\in\mathscr{F}_1$. Then, by (ii) above, $C \subseteq \overline D$. Since no set in $\mathscr{F}_1$ is closed under $\widehat p$, by (b) we conclude that $\overline D$ is a strict superset of $C$. 
This contradicts the fact that $C$ satisfies Property (P4).\smallskip

By Step~3, $\overline D$ satisfies (P1)--(P3). Hence, there is a set $\widetilde  D\subseteq S_0$ such that $\widetilde D$ contains $\overline D$ and $\widetilde D$ satisfies (P1)--(P4), i.e., $\widetilde D\in\mathscr{F}_1$. But then $\overline D$ is not disjoint from $\widetilde D\in\mathscr{F}_1$, which contradicts the previous argument of the proof.
\end{proof}

\subsection{Constructing a Second Family of Sets}\label{sec-secondfam}

In this section, we construct a second family of sets, 
denoted $\mathscr{F}_2$.
As an auxiliary step to define this family, we will define another family of sets, denoted $\overline {\mathscr{F}}_2$.

\begin{notation}[the set $Y_1^D(s)$]
\label{def:Y_1D}
\rm 
Let $D\subseteq S_0$ be a nonempty set, and let $s\in D$. 
Let $Y_1^D(s) \subseteq \Delta(A_1)$ be the set defined as follows:
\begin{itemize}
\item If $s\notin S_{\mathscr{F}_1}$, then $Y_1^D(s)$ is the set of mixed actions $x_1(s)\in\Delta(A_1)$ of player~1 with the following property: whenever an action $a_2\in A_2$ of player~2 satisfies $p(D\mid s,x_1(s),a_2)<1$ then $\E[v_2\mid s,x_1(s),a_2]<v_2(s)$. 
\item If $s\in S_{\mathscr{F}_1}$, then $Y_1^D(s):=Y_1(s).$
\end{itemize}
\end{notation}
In words, whenever $s \not\in S_{\mathscr{F}_1}$,
the set $Y_1^D(s)$ contains all mixed actions of player~1 at state~$s$ such that
all actions of player~2 that cause the game to leave $D$ with positive probability, lower the expected maxmin value of player~2.
The set $Y_1^D(s)$ can be empty for certain sets $D\subseteq S_0$ and states $s\in D\setminus S_{\mathscr{F}_1}$.

\begin{notation}[the product set $Y^D$]\label{notationY^D}\rm
Let $D\subseteq S_0$ be a nonempty set. 
Denote 
\[ Y^D\, :=\,\prod_{s\in D}\Big(Y_1^D(s)\times Y_2(s)\Big). \]
\end{notation}
The set $Y^D$ is nonempty if and only if $Y_1^D(s)$ is nonempty for each $s\in D\setminus S_{\mathscr{F}_1}$.

\begin{definition}[the family of sets $\overline{\mathscr{F}}_2$]
\label{def-F2}
Let $\overline{\mathscr{F}}_2$ be the family of all nonempty sets $D\subseteq S_0$ for which the following properties hold:
\begin{enumerate}
\item[(P'1)] Player 2's maxmin value $v_2(s)$ is constant over $D$, denoted by $v_2(D)$.
\item[(P'2)] The set $Y^D$ is nonempty, and for every $s,s' \in D$, there exists a perturbation $y$ of $Y^D$ that, under $\widehat p$, allows moving from $s$ to $s'$ without leaving $D$: 
\begin{equation}
\label{eqcommforD}
\prob_{\widehat{p},s,y}\big(\theta_D^{exit} = \infty\text{ and }\exists t\in\mathbb{N} \hbox{ such that } s^t=s'\big) \,=\, 1.
\end{equation}
\end{enumerate}
\end{definition}

The family $\overline{\mathscr{F}}_2$ is closely related to the family defined in Section 6.4.2 in Vieille \cite{vieille2000one}. However, our definition is somewhat different. The main reason is that our definition has to accommodate an entire set of mixed actions in each state in $S_0$, rather than single stationary strategies. In addition, our definition, which was partially suggested to us by Nicolas Vieille, circumvents a slight inaccuracy in Lemma~43 in Vieille \cite{vieille2000one}.

\begin{remark}[independence of $y$ in (P'2) from $s$]\label{uniform-pert}\rm 
Note that if (P'2) holds for some $D\subseteq S_0$, then given any fixed $s'\in D$, the perturbation $y$ in (P'2) can be chosen independently from the initial state $s\in S$; that is, there is a perturbation $y$ of $Y^D$ that satisfies \Eqref{eqcommforD} for all initial states $s\in D$. 
Indeed, let $y$ be a perturbation of $Y^D$ such that $y$ satisfies \Eqref{eqcommforD} for some arbitrary initial state $s\in D$. Let $Q$ denote the set of initial states in $D$ for which $y$ satisfies \Eqref{eqcommforD}; 
thus $s\in Q$. If $Q=D$ then we are done. Otherwise, take $w\in D\setminus Q$ and a perturbation $y^w$ of $Y^D$ such that $y^w$ satisfies \Eqref{eqcommforD} for the initial state $w$. 
Playing $y$ in states in $Q$ and playing $y^w$ in states in $D\setminus Q$ satisfies \Eqref{eqcommforD} for all initial states in $Q\cup\{w\}$. 
Continuing this argument recursively, we obtain the desired perturbation that works for all initial states in $D$. 
\end{remark}

The following lemma establishes a relation between the families $\mathscr{F}_1$ and $\overline{\mathscr{F}}_2$, and states a number of properties of the sets in $\overline{\mathscr{F}}_2$.

\begin{lemma}\label{lemma:13}\ 
\begin{enumerate}
\item\label{p1l13} Let $C \in \mathscr{F}_1$ and $D \in \overline{\mathscr{F}}_2$. Then,
either $C \subseteq D$ or $C \cap D = \emptyset$.
\item\label{p3l13} Let $D\in \overline{\mathscr{F}}_2$.  
Then $p(D \mid s,y(s))=\widehat{p}(D \mid s,y(s))=1$ for all $s\in D$ and $y(s)\in Y^D(s)$.\
\item\label{coml13} Each set $D\in \overline{\mathscr{F}}_2$ is $Y^D$-communicating.
\item\label{p4l13} Let $D,D'\in \overline{\mathscr{F}}_2$ such that $D\cap D'\neq \emptyset$. Then, $D\cup D'\in \overline{\mathscr{F}}_2$.
\end{enumerate}
\end{lemma}

\begin{proof}\ \\
\noindent\textsc{Proof of Part 1.} Let $C \in \mathscr{F}_1$ and $D \in \overline{\mathscr{F}}_2$,
and suppose that $C \cap D \neq \emptyset$.
We will show that $C \subseteq D$.

Fix then $s \in C \cap D$ and $s'\in C$.
As $s,s'\in C$, we have $\widehat p(s' \mid s,a) > 0$ for every action pair $a \in A$. 
Since $s\in D$ and by property (P'2), we have $s'\in D$.
Since $s' \in C$ is arbitrary, the claim follows.\medskip

\noindent\textsc{Proof of Part 2.} Let $D\in \overline{\mathscr{F}}_2$, $s\in D$, and $y(s)\in Y^D(s)$. We distinguish between two cases.

Assume first that $s\in S_{\mathscr{F}_1}$, i.e., $s\in C$ for some $C\in\mathscr{F}_1$. Then, $y(s)\in Y(s)$ and by Part 1 we have $C\subseteq D$. 
By (P1) for $C$, we obtain $p(C\mid s,y(s))=1$, hence a fortiori $p(D\mid s,y(s))=1$. 
The transition in state $s$ under $\widehat p$ is independent of the actions. Hence, by (P'2) for $D$, we have $\widehat p(D\mid s,y(s))=1$ as well.

Assume now that $s\notin S_{\mathscr{F}_1}$. Since $\widehat p$ and $p$ coincide in state $s$, it suffices to show $p(D\mid s,y(s))=1$. 
Denote \[B_2(s)\,=\,\{a_2\in A_2\,\colon\, p(D\mid s,y_1(s),a_2)=1\}.\] 
By (P'1) we have $\E[v_2\mid s,y_1(s),a_2]=v_2(D)$ for each $a_2\in B_2(s)$ and by the definition of $Y_1^D(s)$ we have $\E[v_2\mid s,y_1(s),a_2]<v_2(D)$ for each $a_2\in A_2\setminus B_2(s)$. Hence, by Lemma \ref{lemma:expected:minmax} we obtain $\supp(y_2(s))\subseteq B_2(s)$, which implies $p(D\mid s,y(s))=1$.\medskip

\noindent\textsc{Proof of Part 3.} Let $D\in \overline{\mathscr{F}}_2$. By Part 2, we have $p(D\mid s,y(s))=1$ for all $s\in D$ and $y(s)\in Y^D(s)$. Hence, 
Condition 1 of Definition \ref{definition:communicating} satisfied. 

We next verify that Condition 2~of Definition \ref{definition:communicating} holds for each two distinct states $s,s'\in D$. Consider the perturbation $y$ of $Y^D$ given by (P'2) for these states. 
We will define a perturbation $y'$ of $Y^D$ that ensures that under $p$ the play moves from $s$ to $s'$ without leaving $D$.
By Part~\ref{p1l13}, it is sufficient to define $y'$ for every $C \in  \mathscr{F}_1$ such that $C\subseteq D$, and on $D\setminus S_{\mathscr{F}_1}$.
Define then $y'$ as follows:
\begin{itemize}
    \item Definition of $y'$ on $D\setminus S_{\mathscr{F}_1}$: Let $y'(s)=y(s)$ for all states $s\in D\setminus S_{\mathscr{F}_1}$. Note that, in these states, $p$ and $\widehat p$ coincide.
    \item Definition of $y'$ on each $C\in \mathscr{F}_1$ such that $C\subseteq D$: We distinguish between two cases. If $s'\in C$, then $y'$ uses a perturbation of $Y$ to move to state $s'$ from all other states in $C$ without leaving $C$, which is possible by (P1) for $C$. If $s'\notin C$, then $y'$ uses a perturbation of $Y$ to leave $C$ according to the exit distribution $q_C^{\mathscr{F}_1}$. Note that $Y^D$ coincides with $Y$ on $C$.
\end{itemize}
By construction, $y'$ is a $Y^D$-perturbation, and $y'$ satisfies Eq.~\eqref{eqcomm}.
\medskip

\noindent\textsc{Proof of Part 4.} Let $D,D'\in \overline{\mathscr{F}}_2$ such that $D\cap D'\neq \emptyset$. 
Plainly, $D\cup D'$ satisfies (P'1), and $v_2(D)=v_2(D')=v_2(D\cup D')$. 
Thus, we only need to argue that $D\cup D'$ satisfies (P'2). 

First we prove that $Y^{D\cup D'}$ is nonempty. It is enough to verify that $Y_1^{D\cup D'}(s)$ is nonempty for each $s\in D\cup D'$. Let $s\in D\cup D'$. If $s\in S_{\mathscr{F}_1}$ then $Y_1^{D\cup D'}(s)=Y_1(s)$ is nonempty. 
Assume next that $s\in D\setminus S_{\mathscr{F}_1}$; the argument is similar when $s\in D'\setminus S_{\mathscr{F}_1}$.
To prove that $Y_1^{D\cup D'}(s)$ is nonempty, we show that $Y_1^D(s)\subseteq Y_1^{D\cup D'}(s)$. Let $y_1(s)\in Y_1^D(s)$, 
and assume that for some action $a_2\in A_2$ of player 2 we have $p((D\cup D')\mid s,y_1(s),a_2)<1$. 
Then, $p(D\mid s,y_1(s),a_2)<1$, and since $y_1(s)\in Y_1^D(s)$ we obtain $\E[v_2\mid s,y_1(s),a_2]<v_2(D)=v_2(D\cup D')$. This implies that $y_1(s)\in Y_1^{D\cup D'}(s)$, which yields $Y_1^D(s)\subseteq Y_1^{D\cup D'}(s)$, as desired.

We next prove the second part of (P'2). Let $s,s'\in D\cup D'$. If both $s$ and $s'$ belong to $D$ or both belong to $D'$, then \Eqref{eqcommforD} follows from (P'2) for $D$ or from (P'2) for $D'$, respectively. 
So, assume that, say, $s\in D\setminus D'$ and $s'\in D'\setminus D$. Let $s''\in D\cap D'$. Let $y$ be a perturbation of $Y^{D}$ such that $y$ satisfies \Eqref{eqcommforD} for the set $D$, for initial state $s$ and target state $s''$. As mentioned in Remark \ref{uniform-pert}, given $s'$, there is a perturbation $y'$ of $Y^{D'}$ such that $y'$ satisfies \Eqref{eqcommforD} for the set $D'$, for all initial states in $D'$, and target state $s'$. Then, playing $y$ in $D\setminus D'$ and playing $y'$ in $D'$ satisfies \Eqref{eqcommforD} for the set $D\cup D'$, for initial state $s$, and target state $s'$. Thus, $D\cup D'$ also satisfies (P'2).
\end{proof}

\begin{notation}[the family of sets $\mathscr{F}_2$]\rm  
Let $\mathscr{F}_2$ be the family of maximal sets in $\overline{\mathscr{F}}_2$ with respect to set inclusion.
\end{notation}
By Lemma~\ref{lemma:13}(4), the sets in $\mathscr{F}_2$ are disjoint,
and each set in $\overline{\mathscr{F}}_2$ is a subset of some set in $\mathscr{F}_2$.

\begin{notation}[the set of states $S_{{\mathscr{F}}_2}$]\rm Let 
\[ S_{{\mathscr{F}}_2} \,:=\, \cup_{D \in {\mathscr{F}}_2} D \]
be the set of all states that belong to some set in ${\mathscr{F}}_2$.
\end{notation}

The following lemma is the analog of Lemma 44 in Vieille \cite{vieille2000one}.

\begin{lemma}
\label{lemma:44}
For every $D \in \mathscr{F}_2$, the set $D$ is $Y^D$-controlled by player~1.
\end{lemma}

\begin{proof}
Fix $D \in \mathscr{F}_2$. As $\mathscr{F}_2\subseteq \overline{\mathscr{F}}_2$, we have $D\in \overline{\mathscr{F}}_2$. Denote $w^D_1=\max_{s \in D} v_1(s)$.\medskip

\noindent\textbf{Step 1:} If a triplet $(s,a_1,y_2(s))\in D\times A_1\times Y_2(s)$ satisfies $p(D\mid s,a_1,y_2(s))<1$, then $(s,a_1,y_2(s))\in \mathscr{E}_1(D,Y^D)$.\smallskip

Assume that for some $(s,a_1,y_2(s))\in D\times A_1\times Y_2(s)$ we have $p(D\mid s,a_1,y_2(s))<1$. Let $y_1(s)\in Y_1^D(s)$. By Lemma~\ref{lemma:13}(\ref{p3l13}),  $p(D\mid s,y_1(s),y_2(s))=1$. This implies that $(s,a_1,y_2(s))\in \mathscr{E}_1(D,Y^D)$. \medskip

\noindent\textbf{Step 2:} There is a unilateral exit $(s,a_1,y_2(s)) \in \mathscr{E}_1(D,Y^D)$ such that $\E[v_1\mid s,a_1,y_2(s)] \,\geq\, w_1^D$.

Suppose the opposite: for each $(s,a_1,y_2(s))\in\mathscr{E}_1(D,Y^D)$,
\begin{equation}
\label{contra}
\E[v_1\mid s,a_1,y_2(s)] < w_1^D.
\end{equation}
Denote the set of states in $D$ where player 1's maxmin value is maximal by
\[ D^{max} := \{s \in D \colon v_1(s) = w_1^D\}. \]
We will prove that $D^{max}$ satisfies (C1)--(C3), which is impossible due to Lemma \ref{lemma:42}.\smallskip

\noindent\textsc{Property~(C3) for $D^{max}$:} 
The property holds for player 1 by the definition of $D^{max}$ and for player 2 by (P'1).\smallskip

\noindent\textsc{Property~(C2) for $D^{max}$:} Take an arbitrary unilateral exit $(s,a_1,y_2(s))\in \mathscr{E}_1(D^{max},Y)$. We need to show that
\begin{equation}
\label{toshowC2}
\E[v_1\mid s,a_1,y_2(s)] \,<\,w_1^D.
\end{equation}

If $p(D \mid s,a_1,y_2(s)) < 1$, then by Step 1, $(s,a_1,y_2(s))\in \mathscr{E}_1(D,Y^D)$, and consequently, \Eqref{contra} implies \Eqref{toshowC2}.
Otherwise, $p(D \mid s,a_1,y_2(s)) = 1$ holds, and then $p(D^{max}\mid s,a_1,y_2(s)) < 1$ implies \Eqref{toshowC2}.\smallskip

\noindent\textsc{Property~(C1) for $D^{max}$:}
Take $s \in D^{max}$ and $y(s) \in Y(s)$.
We need to show that $\widehat p(D^{max} \mid s,y(s)) = 1$. We distinguish between two cases.\smallskip

\noindent Case 1: $s\in S_{\mathscr{F}_1}$, i.e., $s \in C$ for some $C\in\mathscr{F}_1$. 
In this case,
(i) by the definition of $\widehat p$, we have $\widehat p(\cdot \mid s,y(s)) = \frac{1}{2}q^{\mathscr{F}_1}_C + \frac{1}{2}U_C$;
(ii) by Lemma~\ref{lemma:13}(\ref{p1l13}) we have $C \subseteq D$, and therefore, 
since $s\in D^{max}$ and by (P2) for $C$, we have $C\subseteq D^{max}$; and (iii) by (P'2) for $D$ and the fact that the actions in states in $C$ do not influence the transitions under $\widehat p$, it follows that $\widehat p(D\mid s,y(s))=1$. 

By \Eqref{exp-q}, the fact that $U_C$ only induces transitions to states in $C$,
and points (i) and (ii), 
\[ \E_{\widehat p}[v_1 \mid s,y(s)] \,\geq\, v_1(C)\, =\,w_1^D. \]
Hence, by point (iii), we have $\widehat p(D^{max} \mid s,y(s)) = 1$, as desired.\smallskip

\noindent Case 2: $s \notin S_{\mathscr{F}_1}$. 
In this case, $\widehat p(\cdot \mid s,y(s)) = p(\cdot \mid s,y(s))$. By Eq.~\eqref{equ:81},
\begin{equation}
\label{atleastv1}
\E[v_1 \mid s,y(s)] \,\geq\, v_1(s).
\end{equation} 
Take any $a_1\in\supp(y_1(s))$. 

We show that: (a) if $p(D\mid s,a_1,y_2(s))=1$ then $\E[v_1\mid s,a_1,y_2(s)]\leq v_1(s)$, while (b) if $p(D\mid s,a_1,y_2(s))<1$ then $\E[v_1\mid s,a_1,y_2(s)]<v_1(s)$. By \Eqref{atleastv1} and the fact that $s\in D^{max}$, these will imply that $\widehat p(D^{max} \mid s,y(s))= p(D^{max} \mid s,y(s)) = 1$, as desired.

As $s\in D^{max}$, the claim of (a) follows. 
To show that (b) holds, assume $p(D\mid s,a_1,y_2(s))<1$. 
By Step 1 and \Eqref{contra}, $\E[v_1\mid s,a_1,y_2(s)]<w_1^D=v_1(s)$.\medskip

\noindent\textbf{Step 3:} 
The set $D$ is $Y^D$-controlled by player~1.

By Lemma~\ref{lemma:13}(\ref{coml13}), the set $D$ is $Y^D$-communicating.
It remains to show that there is a unilateral exit in $\mathscr{E}_1(D,Y^D)$ that yield to both players high payoff.

In view of Step~2, we can apply Lemma \ref{maxexit}, and therefore there is an exit $(s,a_1,y_2(s))\in \mathscr{E}_1(D,Y^D)$ that maximizes $\E[v_1\mid\cdot\,]$ among all exits in $\mathscr{E}_1(D,Y^D)$. 
By Step 2 and respectively \Eqref{equ:81} and (P'1),
\begin{equation}
\label{atleastw}
\E[v_1\mid s,a_1,y_2(s)]\,\geq\, w_1^D\quad\text{ and }\quad\E[v_2\mid s,a_1,y_2(s)]\,\geq\, v_2(D).
\end{equation}
It remains to show that for each player $k\in\{1,2\}$,
\begin{equation}
\label{claimstep3}
\E[v_k\mid s,a_1,y_2(s)] \,\geq\, H_k(D,Y^D),
\end{equation}
which will complete the proof of Step 3.

First we prove \Eqref{claimstep3} for $k=1$ by showing a stronger statement: $\E[v_1\mid s,a_1,y_2(s)]= H_1(D,Y^D)$. By \Eqref{defH}, we need to show that for each $(s',a'_1,y'_2(s'))\in D\times A_1\times Y_2(s')$ we have
\begin{equation}
\label{eqpl1}
\E[v_1\mid s,a_1,y_2(s)]\,\geq\,\E[v_1\mid s',a'_1,y'_2(s')].
\end{equation}
If $p(D\mid s',a'_1,y'_2(s'))=1$ then $\E[v_1\mid s',a'_1,y'_2(s')]\,\leq\,w_1^D$, and \Eqref{atleastw} implies \Eqref{eqpl1}. If $p(D\mid s',a'_1,y'_2(s'))<1$, then by Step 1 we have $(s',a'_1,y'_2(s'))\in\mathscr{E}_1(D,Y^D)$, and hence the choice of $(s,a_1,y_2(s))$ implies \Eqref{eqpl1}.

Next we prove \Eqref{claimstep3} for $k=2$. By \Eqref{atleastw}, it suffices to show that 
$H_2(D,Y^D)\,\leq\, v_2(D)$, i.e., that for each  $(s',y'_1(s'),a'_2)\in D\times Y^D_1(s')\times A_2$  we have
\begin{equation}
\label{ineqpl2}
\E[v_2\mid s',y'_1(s'),a'_2]\,\leq\,v_2(D).
\end{equation} 
We distinguish between three cases.\smallskip

\noindent\textsc{Case 1:} $p(D\mid s', y'_1(s'),a'_2)=1$. In this case, \Eqref{ineqpl2} holds with equality 
because $v_2(s)$ is contant over $D$.
\smallskip

\noindent\textsc{Case 2:} $p(D\mid s',y'_1(s'),a'_2)<1$ and $s'\notin S_{\mathscr{F}_1}$. In this case, as $y'_1(s')\in Y_1^D(s')$ we obtain \Eqref{ineqpl2}
by the definition of $Y^D_1(s')$.\smallskip

\noindent\textsc{Case 3:} $p(D\mid s',y'_1(s'),a'_2)<1$ and $s'\in C$ for some $C\in\mathscr{F}_1$. In this case, $y'_1(s')\in Y_1^D(s')=Y_1(s')$ by the definition of $Y_1^D(s')$. 
Note that $q_C^{\mathscr{F}_1}(D)=1$, i.e., $q_C^{\mathscr{F}_1}$ only induces transitions to $D$, which follows by (P'2) and by the fact that $\widehat p$ assigns probability $\frac{1}{2}$ to $q_C^{\mathscr{F}_1}$ in state $s'$.
Therefore,
\[\E[v_2\mid s',y'_1(s'),a'_2]\,\leq\, H_2(C,Y)\,\leq\,\E[v_2\mid q_C^{\mathscr{F}_1}]\,\leq\,v_2(D),\]
where the first inequality holds by \Eqref{defH}, the second inequality holds by \Eqref{exp-q}, and the last inequality holds since $q_C^{\mathscr{F}_1}(D)=1$. 
\Eqref{ineqpl2} follows.
\end{proof}

\begin{notation}[the exit distribution $q_D^{\mathscr{F}_2}$]
\rm
By Lemma~\ref{lemma:44},
for every  $D \in \mathscr{F}_2$ there exists a unilateral exit $(s,a_1,y_2(s))\in\mathscr{E}_1(D,Y^D)$ such that $\E[v_k \mid s,a_1,y_2(s)] \geq H_k(D, Y^D)$ for each player $k\in\{1,2\}$. 
Let $q^{\mathscr{F}_2}_D\in \Delta(S)$ be the corresponding exit distribution from $D$:
\begin{equation}
\label{choiceexitD}
q^{\mathscr{F}_2}_D \,:=\, p(\cdot \mid s,a_1,y_2(s)).
\end{equation} 
\end{notation}
Thus, we have for each $D\in\mathscr{F}_2$ and each player $k\in\{1,2\}$,
\begin{equation}
\label{exp-q2}
\E[v_k\mid q^{\mathscr{F}_2}_D]\,\geq\,H_k(D,Y^D)\,\geq\,\max_{s'\in D}v_k(s'),
\end{equation}
where the last inequality holds by \Eqref{equ:91}. 
We note that if $(s,a_1,y_2(s))$ is the exit on the right-hand side of \Eqref{choiceexitD},
then $s\notin S_{\mathscr{F}_1}$. Indeed, assume the opposite: $s\in C$ for some $C\in \mathscr{F}_1$. Then, $Y_1^D(s)=Y_1(s)$  and by Lemma \ref{lemma:13}(\ref{p1l13}) we have $C\subseteq D$. Hence, $(s,a_1,y_2(s))\in\mathscr{E}_1(C,Y)$, and then (P3) implies $\E[v_1\mid s,a_1,y_2(s)]<v_1(C)$. This is impossible due to \Eqref{atleastw}.

\begin{notation}[the perturbation $y^D$ and the stationary strategy $z_1$]
\label{notz1} \rm

By Remark \ref{uniform-pert}, there is a perturbation $y^D$ of $Y^D$ that satisfies \Eqref{eqcommforD} for all initial states $s'\in D$ and for state $s$ as the target state:
\begin{equation}
\label{propyD}
\prob_{\widehat{p},s',y^D}\big(\theta_D^{exit} = \infty\text{ and }\exists t\in\mathbb{N}: s^t=s\big) \,=\, 1.
\end{equation}
That is, $y^D$ allows moving from any state in $D$ to $s$ under $\widehat p$ without leaving $D$. 

Since $y^D$ is a perturbation of $Y^D$, for each state $s'\in D$ there is a mixed action $\overline y^D_1(s')\in Y_1^D(s')$ such that $\supp(\overline y^D_1(s'))\subseteq \supp(y^D_1(s'))$. Note that, for each $s'\in D\cap S_{\mathscr{F}_1}$, we have $\overline y^D_1(s')\in Y_1^D(s')=Y_1(s')$ by the definition of $Y^D_1(s')$.

Let $z_1^D$ be a stationary strategy for player 1 such that $z^D_1$ coincides with $\overline y_1^D$ on $D \setminus \{s\}$ and $z^D_1(s)\in Y_1^D(s)$ is arbitrary. Since the elements of $\mathscr{F}_2$ are pairwise disjoint, we can assume that $z_1^D$ is independent of $D\in\mathscr{F}_2$, and we denote this stationary strategy by $z_1$.
\end{notation}

\begin{lemma}\label{lemz1}The following properties hold for $z_1$:
\begin{enumerate}
\item\label{p1lemz1} For each 
$D \in \mathscr{F}_2$ and each
$s\in D$
we have $z_1(s)\in Y_1^D(s)$.
\item\label{p2lemz1} For each $s\in S_{\mathscr{F}_1}\cap S_{\mathscr{F}_2}$ we have $z_1(s)\in Y_1(s)$.
\item\label{p3lemz1} For each $D\in\mathscr{F}_2$ there is a perturbation $w_1$ of $z_1$ and a perturbation $w_2$ of $Y_2$ such that: 
\begin{itemize}
    \item[(a)] In each state $s\in D$, under $\widehat p$, the pair of mixed action $(w_1(s),w_2(s))$ induces transitions to states in $D\cup
\supp(q_D^{\mathscr{F}_2})$ only:
\[ \widehat p( D\cup
\supp(q_D^{\mathscr{F}_2})\mid s,w_1(s),w_2(s))\,=\,1.\]
\item[(b)] Regardless of the initial state in $D$, the strategy pair $(w_1,w_2)$ leads outside of $D$ with probability 1 under $\widehat p$: For each $s\in D$,
\[\prob_{\widehat p,s,(w_1,w_2)}(\theta^{exit}_D<\infty)\,=\,1.\]
\end{itemize}
\end{enumerate}
\end{lemma}

\begin{proof}
Properties 1 and 2 follow directly from the definition of $z_1$. 

We argue that Property 3 holds as well. Let $(s,a_1,y_2(s))$ be the exit chosen for $D$ in \Eqref{choiceexitD}. Recall that $s\notin S_{\mathscr{F}_1}$. This implies that $\widehat p$ and $p$ coincide in state $s$. Define $(w_1(s'),w_2(s'))=y^D(s')$ for each $s'\in D\setminus \{s\}$, and $(w_1(s),w_2(s))=(\frac{1}{2}z_1(s)+\frac{1}{2}a_1,y_2(s))$. 

Part (a) of Property 3 follows (i) for states in $D\setminus\{s\}$ by \Eqref{propyD}, (ii) for $(z_1(s),y_2(s))$ in state $s$ by Lemma \ref{lemma:13}(\ref{p3l13}), and (iii) for $(a_1,y_2(s))$ in state $s$ since $\widehat p(\cdot\mid s,a_1,y_2(s))=p(\cdot\mid s,a_1,y_2(s))=q_D^{\mathscr{F}_2}$. 

Part (b) of Property 3 follows by \Eqref{propyD} and since $\widehat p(\cdot\mid s,a_1,y_2(s))=q_D^{\mathscr{F}_2}$.
\end{proof}

\subsection{Constructing a Third Family of Sets}\label{sec-thirdfam}

In this section, we construct a third and final family of subsets of $S_0$.

In view of Lemma \ref{lemma:13}(\ref{p1l13}), and because the elements of $\mathscr{F}_2$ are pairwise disjoint, it follows that for every $C\in\mathscr{F}_1$, we have either $C\subseteq D$ for a unique $D\in\mathscr{F}_2$ or $C\cap D=\emptyset$ for all $D\in\mathscr{F}_2$. 
The latter condition is equivalent to $C \cap S_{\mathscr{F}_2} = \emptyset$.

\begin{definition}[the family of sets $\mathscr{F}_3$]
Let $\mathscr{F}_3$ be the family of sets that consists of all sets in $\mathscr{F}_2$ and all sets $C\in \mathscr{F}_1$ that are disjoint of $S_{\mathscr{F}_2}$:
\[ \mathscr{F}_3 \,:=\, \mathscr{F}_2 \cup \{ C \in \mathscr{F}_1 \colon C \cap S_{\mathscr{F}_2} = \emptyset\}. \]
\end{definition}
By definition, the elements of $\mathscr{F}_3$ are pairwise disjoint. 

\begin{notation}[the set of states $S_{{\mathscr{F}}_3}$]\rm Let 
\[ S_{{\mathscr{F}}_3} := \cup_{E \in {\mathscr{F}}_3} E = S_{{\mathscr{F}}_1} \cup S_{{\mathscr{F}}_2}\]
be the set of all states that belong to some set in ${\mathscr{F}}_3$.
\end{notation}


\begin{notation}[the transition function $\widetilde p$]\rm 
Define a transition function $\widetilde p$ as follows: for each state $s\in S$ and each action pair $a\in A$,
\[ \widetilde p(\cdot \mid s,a) := \left\{
\begin{array}{lll}
\frac{1}{2}q^{\mathscr{F}_2}_D+\frac{1}{2}U_D, & \ \ \ \ \ & \text{if }s \in D\text{ for some } D\in \mathscr{F}_2\text{ (i.e., if $s\in S_{\mathscr{F}_2}$),}\\[0.1cm]
\frac{1}{2}q^{\mathscr{F}_1}_C+\frac{1}{2}U_C, & & \text{if }s \in C\in\mathscr{F}_1\text{ with } C\cap S_{\mathscr{F}_2}=\emptyset,\\[0.1cm]
p(s,a), & & \text{otherwise (i.e., if $s\notin S_{\mathscr{F}_3}$).}
\end{array}
\right.
\]
\end{notation}

The transition function $\widetilde p$ is well defined, because the elements of $\mathscr{F}_3$ are pairwise disjoint. 
Two useful properties of $\widetilde p$ that are evident from its definition are the following:
\begin{itemize}
\item $\widetilde p$ coincides with $\widehat p$ on $S\setminus S_{\mathscr{F}_2}$.
\item $\widetilde p$ is independent of the players' actions on  $S_{\mathscr{F}_3}$.
\end{itemize}

\begin{notation}[the set $Y^{x_1}$]\rm 
When $x_1 = (x_1(s))_{s\in S}$ is a stationary strategy of player~1,
we denote 
\[ Y^{x_1}\, :=\prod_{s\in S_0}\Big(\{x_1(s)\}\times Y_2(s)\Big). \]
Since $Y$ is compact, so is $Y^{x_1}$.
\end{notation}

The following lemma is the analog of Lemma~\ref{lemma:42} to the family $\mathscr{F}_3$,
see Lemma 45 in Vieille \cite{vieille2000one}.

\begin{lemma}
\label{lemma:45}
There exist no set $F \subseteq S_0$ and no stationary strategy $x_1$ of player~1 such that the following three properties are satisfied simultaneously:
\begin{enumerate}
\item[(C'1)]    $\widetilde p(F \mid s,y(s)) = 1$ for every $s \in F$ and every $y(s) \in Y^{x_1}(s)$.
\item[(C'2)]    $F$ is $Y^{x_1}$-blocked to player~2 under $\widetilde p$.
\item[(C'3)]    Player 2's maxmin value $v_2$ is constant over $F$, denoted by $v_2(F)$.
\end{enumerate}
\end{lemma}

\begin{proof}
Assume to the contrary that there exist $F$ and $x_1$ with the properties above. Since $\widetilde p$ is independent of the players' actions on $S_{\mathscr{F}_3}$, we can assume without loss of generality that
\begin{itemize}
\item[(i)] $x_1(s)\in Y_1(s)$ for each $s\in S_{\mathscr{F}_1}$,
\item[(ii)] $x_1(s)=z_1(s)$ for each $s\in S_{\mathscr{F}_2}
$; see Notation \ref{notz1} for the definition of $z_1$.
\end{itemize}
These two conditions do not contradict, because $z_1(s)\in Y_1(s)$ whenever $s\in S_{\mathscr{F}_1}\cap S_{\mathscr{F}_2}$ by Lemma \ref{lemz1}(\ref{p2lemz1}). 
Note that $x_1(s)\in Y_1^D(s)$ for each $D\in\mathscr{F}_2$ and each $s\in D$ by assumption (ii) above and by Lemma \ref{lemz1}(\ref{p1lemz1}).\black\medskip

\noindent\textbf{Step 1:} For each $D\in\mathscr{F}_2$ such that $D\cap F\neq\emptyset$, we have $D\subseteq F$ and $q_D^{\mathscr{F}_2}(F)=1$, i.e., $q_D^{\mathscr{F}_2}$ induces transitions only to states in $F$.

Let $D\in\mathscr{F}_2$ such that $D\cap F\neq\emptyset$. Let $s\in D\cap F$. By (C'1) and 
since $\widetilde p$ in state $s$ assigns probability $\frac{1}{2}$ to $U_D$ and probability $\frac{1}{2}$ to $q_D^{\mathscr{F}_2}$, we have $D\subseteq F$ and  $q_D^{\mathscr{F}_2}(F)=1$.\medskip
\black 

\noindent\textbf{Step 2:} For every $s \in F$ and every $y(s) \in Y^{x_1}(s)$ we have
\begin{equation}
\label{toprovestep1}
\widehat p(F \mid s,y(s)) = 1.
\end{equation} 
Let $s \in F$ and $y(s) \in Y^{x_1}(s)$, so that $y_1(s)=x_1(s)$. We distinguish between two cases.\medskip

\noindent Case 1: $s\notin S_{\mathscr{F}_2}$. Then $\widetilde p$ and $\widehat p$ coincide in state $s$, and hence (C'1) implies \Eqref{toprovestep1}.\smallskip

\noindent Case 2: $s\in S_{\mathscr{F}_2}$. In this case $s\in D$ for some $D\in \mathscr{F}_2$.
By Step 1, $D\subseteq F$. Since $x_1(s)\in Y_1^D(s)$ by assumption,  Lemma~\ref{lemma:13}(\ref{p3l13}) implies that $\widehat p(D \mid s,y(s)) = 1$, and \Eqref{toprovestep1} follows.\black\medskip

\noindent\textbf{Step 3:}  There exists a nonempty subset $\overline F \subseteq F$ such that
\begin{enumerate}
    \item[(a)] $\overline F$ is $Y^{x_1}$-communicating under $\widehat p$.
    \item[(b)] $\widehat p(\overline F\mid s,y(s))=1$ for every $s \in \overline F$ and every perturbation $y(s)$ of $Y^{x_1}(s)$ such that $\widehat p(F\mid s,y(s))=1$.
    \item[(c)] There is no $D\in\mathscr{F}_2$ such that $\overline F\subseteq D$.
\end{enumerate}

For $s,s'\in F$, write $s\to s'$ if there is a perturbation $y(s)$ of $Y^{x_1}(s)$ such that $\widehat p(F\mid s,y(s))=1$ and $\widehat p(s' \mid s,y(s)) >0$. 
Step 2 implies that in the directed graph with $F$ as the set of vertices and arcs given by $\to$, there is a maximal strongly connected component $\overline F$. It follows that $\overline F$ satisfies (a) and (b); cf.~also Lemma 41 in Vieille \cite{vieille2000one}.

We next show that $\overline F$ also satisfies (c). Assume by way of contradiction that $\overline F\subseteq D$ for some $D\in\mathscr{F}_2$. 
Let $(w_1,w_2)$ be as in Lemma \ref{lemz1}(\ref{p3lemz1}) for this set $D$. 
Then, by Lemma~\ref{lemz1}(\ref{p3lemz1}.a)) and Step 1 above, $\widehat p(F\mid s,w_1(s),w_2(s))=1$ for each $s\in D$. Hence, by condition (b) of Step 3 we have $\widehat p(\overline F\mid s,w_1(s),w_2(s))=1$ for each $s\in D$. 
Since by assumption $\overline F\subseteq D$, this implies $\widehat p(D\mid s,w_1(s),w_2(s))=1$ for each $s\in D$. This however contradicts of Lemma~\ref{lemz1}(\ref{p3lemz1}.b).\black\medskip

\noindent\textbf{Step 4:} $\overline F\in\overline{\mathscr{F}}_2$. 

We verify that $\overline F$ satisfy the conditions of Definition \ref{def-F2}. 
Property (P'1) for $\overline F$ holds since $\overline F\subseteq F$ and by (C'3). 

We next verify that Property (P'2) for $\overline F$ holds as well. In view of condition (a) of Step~3, we only need to check that $x_1(s)\in Y_1^{\overline F}(s)$ for each $s\in \overline{F}$ (for the definition of $Y_1^{\overline F}(s)$, see Notation \ref{def:Y_1D}). 
Fix $s\in\overline F$. 
We distinguish between three cases, which are not necessarily mutually exclusive.\smallskip

\noindent Case 1: $s\in S_{\mathscr{F}_1}$, i.e., $s\in C$ for some $C\in\mathscr{F}_1$. Then, $Y_1^{\overline F}(s)=Y_1(s)$ by the definition of $Y_1^{\overline F}(s)$, and by assumption, $x_1(s)\in Y_1(s)$. Hence, $x_1(s)\in Y_1^{\overline F}(s)$, as desired.\smallskip

\noindent Case 2: $s\in S_{\mathscr{F}_2}$, i.e., $s \in D$ for some $D\in\mathscr{F}_2$. Then, $x_1(s)\in Y_1^D(s)$ by assumption.\smallskip

\noindent Case 3: $s\notin S_{\mathscr{F}_3}$. Let $a_2\in A_2$ such that $p(\overline F\mid s,x_1(s),a_2)<1$. We need to show that 
\begin{equation}
\label{toshowp2}
\E[v_2 \mid s,x_1(s),a_2]\,<\,v_2(F).
\end{equation}
Note that $\widetilde p$, $\widehat p$, and $p$ all coincide in state $s$. 

We argue that $p(F\mid s,x_1(s),a_2)<1$. Indeed, let $y_2(s)\in Y_2(s)$, and consider the perturbation $y'(s)=(x_1(s),\frac{1}{2}y_2(s)+\frac{1}{2}a_2)$ of $Y^{x_1}(s)$. 
By the choice of $a_2$ we have $p(\overline F\mid s,y'(s))<1$. Hence, condition (b) of Step~3 implies that $p(F\mid s,y'(s))<1$. 
By condition (a) of Step~3, $p(F\mid s,x_1(s),y_2(s))=1$, and this yields $p(F\mid s,x_1(s),a_2)<1$, as desired.

Condition (a) of Step~3 implies that $(s,x_1(s),a_2)\in \mathscr{E}_{2,\widetilde p}(F,Y^{x_1})$. Thus, by (C'2) we have $\E[v_2\mid s,x_1(s),a_2]\,=\,\E_{\widetilde p}[v_2\mid s,x_1(s),a_2]\,<\,v_2(F)$, which proves \Eqref{toshowp2}.\medskip

\noindent\textbf{Step 5:} The contradiction.

By Step 4 and the definition of $\mathscr{F}_2$, there is a set $D\in\mathscr{F}_2$ such that $\overline F\subseteq D$. This is however in contradiction with property (c) of Step 3. 
\end{proof}

\subsection{An Auxiliary Stochastic Game}\label{sec-auxgame}

In this section we will define an auxiliary stochastic game,
and compare the players' maxmin values in this game to their maxmin values in the original game.

\begin{notation}[the auxiliary stochastic game $\Gamma^R$]\rm 
Let $\Gamma^R$ be the following two-player stochastic game:
\begin{enumerate}
\item The set of states is $S$.
\item In each state $s \in S_{\mathscr{F}_3}$, each player has a single action,
so these states are dummy states in $\Gamma^R$.
\item In each state $s \not\in S_{\mathscr{F}_3}$, the sets of actions of the two players are $A_1$ and $A_2$.
\item The transition function is $\widetilde p$. 
Note that here we made a slight abuse of notation: 
in states in $S_{\mathscr{F}_3}$ the players have a single action,
yet in those states the transition is independent of the action pair.
\item The payoff function of each player $i$, denoted by $f_i^R$, is defined as follows: Let $r\in\mathscr{R}$.
\begin{itemize}
    \item If the run $r$ never reaches an absorbing state, then $f_i^R(r)=0$.
    \item If the run $r$ reaches an absorbing state $s\in S^*$, then $f_i^R(r)=(\gamma_1^s,\gamma^s_2)$; recall the definition of  $(\gamma_1^s,\gamma^s_2)$ in Assumption \ref{As2}.
\end{itemize}
\end{enumerate}
\end{notation}

\begin{comment}
In a stochastic game as defined in Definition~\ref{def:stochastic:game},
the sets of actions of the players are state independent.
In $\Gamma^R$, some states may be dummy states,
where the action set of both players is a singleton.
Since the transition out of these states is predetermined,
one can remove these states from the game by incorporating the transitions out of them in the transition that leads to them.
Therefore, to prove the existence of $\ep$-equilibria in $\Gamma^R$, 
we can assume that the actions are state-independent.
\end{comment}

The game $\Gamma^R$ is a \emph{recursive} game: 
the payoff is 0 if no absorption occurs, and otherwise it is fully determined by the state where the run is absorbed; cf.~Everett \cite{Everett1957}. A similar game is defined on page 85 in Vieille \cite{vieille2000one}. 

\begin{notation}[the probability measure $\prob^R$, the expectation $\E^R$]\rm
Every strategy pair $(\sigma_1,\sigma_2)$ in $\Gamma^R$, 
together with the initial state $s \in S$, 
induce a probability measure $\prob^R_{s,\sigma_1,\sigma_2}$ on the set $\mathscr{R}$ of runs in $\Gamma^R$.
Denote by $\E^R_{s,\sigma_1,\sigma_2}$ the corresponding expectation operator.
\end{notation}

\begin{notation}[the maxmin values $v_i^R$]\rm 
For each $s \in S$ and each $i \in \{1,2\}$,
denote by $v_i^R(s)$ player $i$'s maxmin value at the initial state $s$ in $\Gamma^R$.
\end{notation}
\black

\begin{definition}[absorbing stationary strategy pair]
A stationary strategy pair $(x_1,x_2)$ is called \emph{absorbing} in a stochastic game if under $(x_1,x_2)$, regardless of the initial state, the run eventually reaches an absorbing state with probability 1.
\end{definition}

The next lemma shows that in $\Gamma^R$,
for every stationary strategy of player~1, there is a stationary strategy of player~2 such that the arising stationary strategy pair is absorbing, and player 2's expected absorbing payoff in $\Gamma^R$ is at least her maxmin value in $\Gamma$. 
This is the analog of Lemma 46 in Vieille \cite{vieille2000one}. 
One consequence of this result is that the set $S^*$ of absorbing states is nonempty.

\begin{lemma}
\label{lemma:46}
The following property holds for the game $\Gamma^R$:
For every stationary strategy $x_1$ of player 1,
there exists a stationary strategy $x_2$ of player 2 such that:
\begin{enumerate}
    \item $(x_1,x_2)$ is absorbing in $\Gamma^R$.
    \item $\E^R_{s,x_1,x_2}[f_2^R] \geq v_2(s)$ for each $s \in S_0$.
\end{enumerate}
\end{lemma}

\begin{proof}
Fix a stationary strategy $x_1$ of player~1.

\smallskip
\noindent\textbf{Step 1:} The definition of the stationary strategy $x_2^*$ for player 2.

For every state $s \in S_0$ define
\[ B_2(s) := \left\{ a_2 \in A_2\, \colon\, \E_{\widetilde p}[v_2 \mid s,x_1(s),a_2] \geq v_2(s)\right\}.\]

We argue that $B_2(s) \neq\emptyset$ for each $s \in S$. Indeed, suppose first that $s\in\ S_{\mathscr{F}_3}$, i.e., $s \in E$ for some $E \in \mathscr{F}_3$. Then $E\in\mathscr{F}_\ell$ for some $\ell\in\{1,2\}$, and for every $a_2\in A_2$ we have
\begin{equation}
\label{atleastv2ptilde}
\E_{\widetilde p}[v_2 \mid s,x_1(s),a_2] \,=\, \frac{1}{2}\cdot\E[v_2\mid q_E^{\mathscr{F}_\ell}]+\frac{1}{2}\cdot v_2(E)\, \geq\, v_2(E)\,=\,v_2(s), 
\end{equation}
where the inequality is by Eqs.~\eqref{exp-q} and \eqref{exp-q2}. 
Since at each state $s \in S_{\mathscr{F}_3}$, 
the transition is independent of the players' actions,
in this case, $B_2(s)=A_2$. 

Suppose next that $s\notin S_{\mathscr{F}_3}$. Then $\widetilde p$ coincides with $p$ in state $s$, and hence $B_2(s) \neq\emptyset$ by Eq.~\eqref{equ:minmax}.

Let $x^*_2(s)$ be the uniform distribution over $B_2(s)$,
for every $s \in S_0$. 
This implies in particular that the process $(v_2(s^t))_{t \in \dN}$ of player~2's maxmin values 
is a submartingale under the stationary strategy pair $(x_1,x^*_2)$ and $\widetilde p$, i.e., under $\prob^R_{s,x_1,x_2^*}$. 
We will prove that $x_2^*$ satisfies properties 1 and 2 of Lemma \ref{lemma:46}.\medskip

\smallskip
\noindent\textbf{Step 2:} $(x_1,x_2^*)$ is absorbing in $\Gamma^R$. 

If $(x_1,x_2^*)$ is not absorbing in $\Gamma^R$,
then there is a nonempty set $F\subseteq S_0$ such that $\widetilde p(F\mid s,x_1(s),x_2^*(s))=1$ for each state $s\in F$. 

Let $w_2^F=\max_{s\in F}v_2(s)$, and define
\[F^{max}\,=\,\Big\{s\in F\colon v_2(s)=w_2^F\Big\}.\]
We will use the following facts: (i) $\widetilde p(F^{max}\mid s,x_1(s),x_2^*(s))=1$ for each $s\in F^{max}$, by the definition of $x_2^*$; (ii) $\E_{\widetilde p}[v_2\mid s,x_1(s),a_2]=w_2^F$ for each $s\in F^{max}$ and $a_2\in B_2(s)$; and (iii) $\E_{\widetilde p}[v_2\mid s,x_1(s),a_2]<w_2^F$ and $\widetilde p(F^{max}\mid s,x_1(s),a_2)< 1$ for each  $s\in F^{max}$ and each $a_2\in A_2\setminus B_2(s)$.

We focus on the conditions of Lemma~\ref{lemma:45}. 
Note that $F^{max}$ satisfies (C'3) by definition. 

We now argue that $F^{max}$ satisfies (C'1). To this end, we fix $s\in F^{max}$ and $y_2(s)\in Y_2(s)$, and show that 
\begin{equation}
\label{forC'1}
\widetilde p(F^{max}\mid s,x_1(s),y_2(s))=1.
\end{equation}
If $s\in S_{\mathscr{F}_3}$, then $B_2(s)=A_2$, and hence (i) implies \Eqref{forC'1}. So, assume that $s\notin S_{\mathscr{F}_3}$. Then $\widetilde p$ coincides with $p$ in state $s$, and \Eqref{equ:81} yields $\E_{\widetilde p}[v_2\mid s,x_1(s),y_2(s)]\geq v_2(s) = w_2^F$. Hence, by (ii) and (iii) we have $\supp(y_2(s))\subseteq B_2(s)$, and then by (i) we obtain \Eqref{forC'1}.

In view of Lemma~\ref{lemma:45}, $F^{max}$ cannot satisfy (C'2);
that is, $F^{max}$ is not $Y^{x_1}$-blocked to player~2 under $\widetilde p$.
Hence, there are $s\in F^{max}$ and $a_2\in A_2$ such that $\widetilde p(F^{max}\mid s,x_1(s),a_2)<1$ and \[\E_{\widetilde p}[v_2\mid s,x_1(s),a_2]\,\geq\, w_2^F.\] 
This implies that $\widetilde p(F\mid s,x_1(s),a_2)<1$ and $a_2\in B_2(s)$. This, however, contradicts the assumption that $\widetilde p(F\mid s,x_1(s),x_2^*(s))=1$.\medskip

\smallskip
\noindent\textbf{Step 3:}
$\E^R_{s,x_1,x_2^*}[f^R_2] \geq v_2(s)$ for each $s \in S_0$.

Let $\theta$ denote the stage of absorption; if no absorption occurs then $\theta=\infty$. For every $s\in S_0$, we have
\[\E^R_{s,x_1,x^*_2}[f^R_2] \,=\, \E^R_{s,x_1,x^*_2}[v_2(s^{\theta})] \,\geq\, v_2(s); \]
here the equality holds because by Step~2 $(x_1,x^*_2)$ is absorbing in $\Gamma^R$, and the inequality holds because, as mentioned above, the process $(v_2(s^t))_{t \in \dN}$ of player~2's maxmin values 
is a submartingale under $\prob^R_{s,x_1,x_2^*}$. 
\end{proof}

\bigskip

The next lemma compares each player's maxmin values in the games $\Gamma$ and $\Gamma^R$. This is the analog of Lemma 47 in Vieille \cite{vieille2000one}.

\begin{lemma}
\label{lem47}
$v^R_i(s) \geq v_i(s)$ for each player $i\in\{1,2\}$ and each state $s \in  S_0$.
\end{lemma}

\begin{proof}
Fix $s\in S_0$.

We start by proving the lemma for player 2. 
Let $\ep>0$. In the game $\Gamma^R$, by Everett \cite{Everett1957} or by Thuijsman and Vrieze \cite{thuijsman1992note}, player 1 has a stationary strategy $x_1$ such that $\E^R_{s,x_1,y_2}[f^R_2]\leq v_2^R(s)+\ep$ for every stationary strategy $y_2$ of player 2. Let $x_2$ be as in Lemma \ref{lemma:46} for $x_1$. 
By property 2 of Lemma \ref{lemma:46}, we have $v_2(s)\leq \E^R_{s,x_1, x_2}[f^R_2]\leq v_2^R(s)+\ep$. As $\ep>0$ is arbitrary, we have $v_2(s)\leq v_2^R(s)$ as desired.

We next prove the claim for player 1.
Our proof will use the assumption that player~1's absorbing payoffs are negative, see Assumption~\ref{As1}.
Let $y_1\in Y_1$ be a stationary strategy for player 1. We first argue that for every stationary strategy $x_2$ for player 2,
\begin{equation}
\label{arby}
\E^R_{s,y_1,x_2}[f^R_1]\,\geq\, v_1(s).
\end{equation} 
Fix then a stationary strategy $x_2$ for player 2. 
Let $\theta$ denote the stage of absorption; if no absorption occurs then $\theta=\infty$. By Lemma \ref{lemma:expected:minmax}, the process $(v_1(s^t))_{t\in\dN}$ of player 1's maxmin values is a submartingale under $\prob_{s,y_1,x_2}$, and hence by Eqs.~\eqref{exp-q} and \eqref{exp-q2}, $(v_1(s^t))_{t\in\dN}$ is also a submartingale under $\prob^R_{s,y_1,x_2}$. Let $v_1^\infty$ denote its limit. Note that $v_1^\infty\leq 0$ by Assumption \ref{As1}, and also that $v_1^\infty=v_1(s^\theta)$ if $\theta<\infty$. Hence,
\begin{align*}
\E^R_{s,y_1,x_2}[f^R_1]\,&=\,\prob^R_{s,y_1,x_2}(\theta < \infty) \cdot \E^R_{s,y_1,x_2}[v_1(s^\theta) \mid \theta < \infty]\\
&=\,\E^R_{s,y_1,x_2}[\mathbf{1}_{\{\theta < \infty\}} v_1(s^\theta)]\\
&\geq\,\E^R_{s,y_1,x_2}[v_1^\infty]\\
&\geq\,v_1(s),
\end{align*}
which proves \Eqref{arby}.

As $\Gamma^R$ is a recursive game, player 2 has a stationary best response in $\Gamma^R$ against the stationary strategy $y_1$; cf.~for example, Everett \cite{Everett1957} or Thuijsman and Vrieze \cite{thuijsman1992note}. 
Hence, \Eqref{arby} implies that $v_1^R(s)\geq v_1(s)$, as desired. 
\end{proof}

\subsection{Completing the Proof of Theorem~\ref{theorem:1}}\label{sec-finalstep}

In this section, we complete the proof of Theorem~\ref{theorem:1}. We show that, for any initial state in $S_0$, the game $\Gamma$ admits an $\ep$-equilibrium for all $\ep>0$. Since this is in contradiction with Assumption \ref{As3}, the proof of Theorem~\ref{theorem:1} will be complete.
Our arguments are similar to those in Vieille~\cite{vieille2000one},
hence we only sketch the construction.

The recursive game $\Gamma^R$, defined in Section \ref{sec-auxgame}, has the following properties:
\begin{itemize}
\item $\Gamma^R$ is positive in the sense of Vieille \cite{vieille2000one,vieille2000two}: $v_2^R(s)\geq v_2(s)\geq 1>0$ for each absorbing state $s\in S^*$,
where the inequalities hold by Assumption \ref{As1} and Lemma \ref{lem47}.
\item $\Gamma^R$ has the absorbing property in the sense of Vieille \cite{vieille2000one,vieille2000two}: player 2 has a stationary strategy $x_2$ such that for each stationary strategy $x_1$ of player~1, the pair $(x_1,x_2)$ is absorbing in $\Gamma^R$. Indeed, by Lemma~\ref{lemma:46}(1), any stationary strategy for player 2 that uses each action with positive probability will do.
\end{itemize}
In view of these properties, player 2 wants to reach absorption in $\Gamma^R$, and she can also force it. The following statement follows from the main result in Vieille \cite{vieille2000two}.

\begin{lemma}
The recursive game $\Gamma^R$ admits an $\ep$-equilibrium for every $\ep>0$.
\end{lemma}

Fix $\ep \in(0,\frac{1}{2})$, and let $\sigma^R$ be an $(\frac{\ep}{5})$-equilibrium for $\Gamma^R$. The next statement is Lemma 48 in Vieille \cite{vieille2000one}. As before, $\theta$ denotes the stage of absorption, with the convention that $\theta=\infty$ if the play is never absorbed.

\begin{lemma}\label{lem48}
The strategy pair $\sigma^R$ leads to absorption in $\Gamma^R$ with probability at least $1-\frac{\ep}{2}$, i.e., $\prob^R_{s, \sigma^R}(\theta<\infty)\geq 1-\frac{\ep}{2}$ for each initial state $s\in S$. 
\end{lemma}

In Section 7 of \cite{vieille2000one}, based on $\sigma^R$, Vieille defines a strategy pair $\sigma^*$ for the original game~$\Gamma$, and shows that $\sigma^*$ is an $\ep$-equilibrium for $\Gamma$. Roughly speaking, the play according to $\sigma^*$ has the following properties, for any initial state $s\in S_0$:
\begin{itemize}
\item 
For any absorbing state $s'\in S^*$, the probability under $\sigma^*$ of absorption in $s'$ is close to the corresponding probability under $\sigma^R$ in the game $\Gamma^R$. 
Consequently, each player $i$'s expected payoff is close to the corresponding payoff under $\sigma^R$ in the game $\Gamma^R$.
\item 
If the play enters a set $E\in\mathscr{F}_3$, then the players switch to a strategy pair that implements the exit distribution $q_E^{\mathscr{F}_1}$ (if $E\in \mathscr{F}_1$) or $q_E^{\mathscr{F}_2}$ (if $E\in \mathscr{F}_2$). If a player is suspected of deviation, the opponent switches to a punishment strategy.
\item Whenever the play is in $S_0 \setminus S_{\mathscr{F}_3}$, the players follow $\sigma^R$. 
\item
Whenever the play reaches an absorbing state $s \in S_*$,
the players forget past play and follow an $\frac{\ep}{5}$-equilibrium for the initial state $s$ that yields payoffs $\frac{\ep}{5}$-close to $(\gamma^1_s,\gamma^2_s)$.
\end{itemize}

In the game $\Gamma$, the play under $\sigma$ may need some time to exit a set $E\in\mathscr{F}_3$, whereas in the game $\Gamma^R$, the sets $E\in\mathscr{F}_3$ are artificially exited by the transition function. 
This creates a difference between histories in $\Gamma$ and $\Gamma^R$. 
Therefore, when the play is outside $S_{\mathscr{F}_3}$, the players should not follow $\sigma^R$ with the actual history in $\Gamma$, 
but rather they should delete the part of the history that was spent on leaving the set $S_{\mathscr{F}_3}$. 

To deal with this problem, Vieille \cite{vieille2000one} defines a mapping that identifies  
histories in $\Gamma$ with histories in $\Gamma^R$, by using a deletion operation. 
We do the same.
Since the payoff (or in Vieille's case, the long-term average payoff) is shift-invariant, this deletion of finitely many stages does not affect the payoffs.

Therefore, Vieille's \cite{vieille2000one} construction and arguments lead to the following statement.

\begin{lemma}
The game $\Gamma$ admits an $10\ep$-equilibrium. 
\end{lemma}

Since $\ep\in(0,\frac{1}{2})$ was arbitrary, this is in contradiction with
Assumption \ref{As3}. Hence, the proof of Theorem \ref{theorem:1} is complete.\quad\scriptsize$\blacksquare$\normalsize

\section{Discussion}\label{sec-discuss}

\noindent\textbf{Finiteness of the set of states.} Our proof relies on the assumption that the set of states $S$ is finite. Indeed, we closely follow the main lines of the proof of Vieille \cite{vieille2000one,vieille2000two}, and his proof makes strong use of this assumption.
Also, the crucial Lemma~\ref{lemma:alternatives} uses this assumption at several places; 
for example, in Step 0 to obtain a uniform bound of loss $\eta^\rho$ for exits, and in Step 3 to find a state $s^{**}$ satisfying certain properties. 

For a countably infinite set of states $S=\dN$, additional difficulties arise. Even for a fixed pair of stationary strategies, the induced Markov chain on the set of states $\dN$ can generally no longer be classified into transient states and ergodic sets of states. In particular, it would not be sufficient to work with communicating sets as defined in Definition \ref{definition:communicating}. 
It is not known if Vieille's result \cite{vieille2000one,vieille2000two}  can be extended to  
stochastic games with a countable set of states,
and a fortiori
the same holds for our result.\medskip

\noindent\textbf{Finiteness of the action sets.} We assumed that the action sets are finite. The reason is that even one-shot games with countably infinite action sets may fail to have an $\ep$-equilibrium. 
A well-known example is the zero-sum game with action sets $A_1=A_2=\dN$ in which player 1 wins when her action is larger than player 2's action,
and loses otherwise.
\medskip

\noindent\textbf{Boundedness and Borel-measurability of the payoffs.} 
These two assumptions are standard. They are used to guarantee that each strategy pair induces a well-defined expected payoff, and they are needed for the application of the results of Martin \cite{martin1998determinacy} 
or
 Maitra and Sudderth \cite{maitra1998finitely}.\medskip

\noindent\textbf{Shift-invariance of the payoffs.} 
In the main result, Theorem \ref{theorem:1}, we assumed that the payoff functions are shift-invariant. 
This assumption was used at various places: 
\begin{itemize}
\item It ensured that the maxmin values only depend on the state, cf.~Lemma \ref{const-surr}.
\item Without this assumption, whether an $\ep$-equilibrium exists in the subgame that starts at a certain state
may depend on past play, 
so the proper set of states is the set of finite histories, which is countably infinite.
In addition, in Step 3 of the proof of Lemma \ref{lemma:alternatives} we used the finiteness of the set of states 
to find a state $s^{**}$ satisfying certain properties.
\end{itemize}

We now briefly discuss a condition that is slightly weaker than shift-invariance. A payoff function $f_i$ is called \emph{tail-measurable} if whenever two runs $r = (s^1,a^1,s^2,a^2,\ldots) \in \mathscr{R}$
and $r' = (s'^1,a'^1,s'^2,a'^2,\ldots)\in \mathscr{R}$ satisfy $s^t = s'^t$ and $a^t = a'^t$ for every $t$ sufficiently large, we have $f_i(r)=f_i(r')$. Intuitively,
$f_i$ is tail-measurable if the payoff remains the same when we change finitely many coordinates of the run. 

Every shift-invariant payoff function is also tail-measurable. Indeed, suppose that $f_i$ is shift-invariant, and consider two runs $r = (s^1,a^1,s^2,a^2,\ldots)$
and $r' = (s'^1,a'^1,s'^2,a'^2,\ldots)$ satisfying $s^t = s'^t$ and $a^t = a'^t$ for every $t\geq T$, for some $T\in\dN$. Then, by shift-invariance, $f(r)=f(s^T,a^T,s^{T+1},a^{T+1},\ldots)=f(r')$, as desired.

The opposite is however not true: there are tail-measurable but not shift-invariant payoff functions.
For example, consider a repeated game 
(that is, a stochastic game with a single state)
with a single player who has two actions: $A_1 = \{a,b\}$.
Suppose that
\begin{itemize}
    \item $f_1(r) = 1$ if there is $t_0 \in \dN$ such that $a_1^{2t}=a$ and $a_1^{2t+1} = b$ for every $t \geq t_0$.
    \item $f_1(r) = 0$ otherwise.
\end{itemize}
Then $f_1$ is tail-measurable but not shift-invariant.

As the following example shows, tail-measurable payoffs pose difficulty for our proof technique:
in such a case, the maxmin value of a player in a subgame does not only depend on the current state, but also on the current stage. 
Therefore, studying the slightly more general tail-measurable payoffs leads us to stochastic games with countably many states, and entails additional technical difficulties.\medskip

\begin{example}\rm 
\label{example:1}
Consider a game with three states $S=\{1,2,3\}$, with state 1 being the initial state. Player 1 has two actions $T$ and $B$
in state 1 and only one action $B$ in states 2 and 3, while player 2 has only one action in each state (i.e., player 2 is dummy). 
In state 1, the action $T$ of player 1 keeps the play in state 1 whereas the action $B$ leads to state 2 with probability 1. From state 2 
the play moves to state 3 with probability 1, and from state 3 the play moves to state 2 with probability 1. 
The transition function appears graphically in Figure~\arabic{figurecounter}.
The payoff of player 1 is equal to 1 if state 2 is reached at an even stage, and it is equal to 0 otherwise.\vspace{0.2cm}\it

\centerline{
\begin{picture}(270,60)(-10,-15)
\put(0,00){\numbercellonga{1}}
\put(0,20){\numbercellonga{2}}
\put(-10,8){$B$}
\put(-10,28){$T$}
\put(5,-15){State 1}
\put(100,0){\numbercellonga{3}}
\put(200,0){\numbercellonga{2}}
\put(90,8){$B$}
\put(190,8){$B$}
\put(105,-15){State 2}
\put(205,-15){State 3}
\end{picture}}
\smallskip
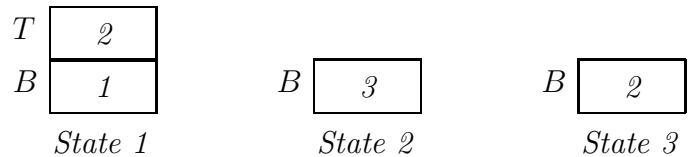
\centerline{{Figure $\arabic{figurecounter}$: The transition function in the game in Example \ref{example:1}.}}
\medskip

\rm Player 1's payoff function is tail-measurable,
since it is determined by the parity of the stage in which we leave state 1. 
Since player~2 is a dummy player, the transitions are deterministic,
and no two actions lead to the same transition,
we can identify a run with the sequence of states.
Note that the sequence of states $(1,1,3,2,3,2,3,2,3,...)$ is not a run, as it does not respect the transitions in the game.

Player 1's payoff function is not shift-invariant though, as 
$f_1(1,1,2,3,2,3,2,3,...) = 0$ while $f_1(1,2,3,2,3,2,3,\dots) = 1$. 

The maxmin value of player 1 does not only depend on the current state, but also on the current stage. For example, if $h$ is a history with final state $s_h=2$, then player 1's maxmin value in the subgame that starts at $h$ is equal to 1 if $\len(h)$ is even, and to 0 if $\len(h)$ is odd. 
\end{example}

\noindent\textbf{More than two players.} The existence of $\ep$-equilibrium for stochastic games with more than two-players is an open problem, and it is already challenging in games with a simple structure 
(e.g., in quitting games, see Solan and Vieille \cite{solan2001quitting} or Simon \cite{simon2012topological}). 

Let us provide some more details in terms of our proof techniques. For two-player stochastic games with the long-term average payoff, 
Vieille \cite{vieille2000one} uses auxiliary zero-sum games,
and relies on Mertens and Neyman \cite{mertens1981stochastic} for the existence and structure of $\ep$-optimal strategies in
these zero-sum games. For bounded and Borel-measurable payoffs, Martin \cite{martin1998determinacy} or Maitra and Sudderth \cite{maitra1998finitely} provide an analog to Mertens and Neyman \cite{mertens1981stochastic}, and this is an important building block in our proof.

For more than two players, however, the maxmin value and the minmax value of a player are generally different (i.e., the last inequality in \Eqref{def-maxmin} fails). As a consequence, the results of Mertens and Neyman \cite{mertens1981stochastic}, Martin \cite{martin1998determinacy}, and Maitra and Sudderth \cite{maitra1998finitely}, only take a weaker form, cf.~Neyman \cite{neyman2003stochastic} and Ashkenazi-Golan, Flesch, Predtetchinski, and Solan \cite{ashkenazigolan2021regularity}. In particular, in the auxiliary zero-sum games where one player maximizes her own payoff and the opponents minimize the same (without correlation), one cannot take $\delta$-optimal strategies as we do in Section \ref{section:martin}.

In certain classes of games with more than two players and the long-term average payoff, $\ep$-equilibria have been established by using limit properties of discounted equilibria as the discount factor vanishes, as in Solan \cite{solan1999three}, which rely on the existence of discounted equilibria by Fink \cite{fink1964equilibrium} or Takahashi \cite{takahashi1964equilibrium}. However, at the moment, Fink \cite{fink1964equilibrium} and Takahashi \cite{takahashi1964equilibrium} have no analog for bounded and Borel-measurable payoffs. Perhaps a suitable approximation of the bounded and Borel-measurable payoffs, as in  Ashkenazi-Golan, Flesch, Predtetchinski, and Solan \cite{ashkenazigolan2021regularity}, can pave the way forward.\medskip


\noindent\textbf{The choice of the subgame-perfect $\delta$-maxmin strategies.} In Section \ref{section:martin}, we constructed a subgame-perfect $\delta$-maxmin strategy for each player through the functions $D^\delta_i$ given by Lemma \ref{theorem:martin}. 
This very specific construction was used in Step 2 of the proof of Lemma \ref{lemma:alternatives} to ensure that the strategy pair $(\pi^{\rho,\delta}_1,\pi^{\rho,\delta}_2)$, which was derived from the subgame-perfect $\delta$-maxmin strategies, still yields high expected payoffs to the players. 
We do not know whether the same would remain true for $(\pi^{\rho,\delta}_1,\pi^{\rho,\delta}_2)$ if one defines them using arbitrary subgame-perfect $\delta$-maxmin strategies for the players. \medskip 

\noindent\textbf{Subgame-perfect $\ep$-equilibrium.} 
The main result of the paper, Theorem \ref{theorem:1}, cannot be strengthened to the existence of subgame-perfect $\ep$-equilibrium. Indeed, Flesch, Kuipers, Mashiah-Yaakovi, Schoenmakers, Shmaya, Solan, and Vrieze \cite{flesch2014non} presented a two-player stochastic game with two states and bounded, Borel-measurable, and shift-invariant payoffs that admits no subgame-perfect $\ep$-equilibrium for small $\ep>0$.\medskip

\printbibliography
\end{document}